\documentclass[11pt,reqno]{amsart}

\usepackage[T1]{fontenc}
\usepackage[boldsans]{ccfonts}
\usepackage{eulervm}
\renewcommand{\mathbf}{\mathbold}
\usepackage{bbm}

\usepackage{amsmath, amsthm, amsfonts, amssymb, %xy,  
mathrsfs, graphicx, paralist} %, ulem
\usepackage[usenames, dvipsnames]{color}
\usepackage[margin=1in]{geometry} 
\usepackage[hidelinks]{hyperref}
\usepackage{comment}
\usepackage[shortlabels]{enumitem}

\usepackage{mathabx} %This is for \sqbullet
\usepackage{tikz-cd}
\usetikzlibrary{chains,decorations.pathreplacing,backgrounds,decorations.markings,babel,positioning,patterns,shapes}
\usepackage{tikz}
\usepackage{nicematrix}

\usepackage{ytableau}

\numberwithin{equation}{section}
\newtheorem{Theorem}[equation]{Theorem}
\newtheorem{Proposition}[equation]{Proposition}
\newtheorem{Lemma}[equation]{Lemma}
\newtheorem{Corollary}[equation]{Corollary}

\theoremstyle{definition}
\newtheorem{Remark}[equation]{Remark}

\newtheorem{eg}[equation]{Example}

\newtheorem{Definition}[equation]{Definition}

\newcommand{\cA}{\mathcal{A}}

\newcommand{\bF}{\mathbf{F}}

\newcommand{\bG}{\mathbf{G}}

\newcommand{\fJ}{\mathfrak{J}}

\newcommand{\bN}{\mathbf{N}}
\newcommand{\cN}{\mathcal{N}}

\newcommand{\cO}{\mathcal{O}}

\newcommand{\bP}{\mathbf{P}}
\newcommand{\cP}{\mathcal{P}}

\newcommand{\cR}{\mathcal{R}}

\newcommand{\cS}{\mathcal{S}}

\newcommand{\bT}{\mathbf{T}}

\newcommand{\fb}{\mathfrak{b}}

\newcommand{\bc}{\mathbf{c}}

\newcommand{\fg}{\mathfrak{g}}

\newcommand{\fh}{\mathfrak{h}}

\newcommand{\bk}{\mathbf{k}}

\newcommand{\bm}{\mathbf{m}}

\newcommand{\bn}{\mathbf{n}}

\newcommand{\bu}{\mathbf{u}}
\newcommand{\fu}{\mathfrak{u}}

\newcommand{\bv}{\mathbf{v}}

\newcommand{\bz}{\mathbf{z}}

\newcommand{\CC}{\mathbb{C}}

\newcommand{\GG}{\mathbb{G}}

\newcommand{\NN}{\mathbb{N}}

\newcommand{\PP}{\mathbb{P}}
\newcommand{\QQ}{\mathbb{Q}}

\newcommand{\ZZ}{\mathbb{Z}}

\renewcommand{\phi}{\varphi}
\renewcommand{\emptyset}{\varnothing}
\newcommand{\eps}{\varepsilon}

\renewcommand{\tilde}[1]{\widetilde{#1}}

\newcommand{\leftexp}[2]{\vphantom{#2}^{#1} #2}

\makeatletter
\def\Ddots{\mathinner{\mkern1mu\raise\p@
\vbox{\kern7\p@\hbox{.}}\mkern2mu
\raise4\p@\hbox{.}\mkern2mu\raise7\p@\hbox{.}\mkern1mu}}
\makeatother

\newcommand{\twoheadlongrightarrow}{\relbar\joinrel\twoheadrightarrow}

\DeclareMathOperator{\supp}{Supp}

\DeclareMathOperator{\Hom}{Hom}

\DeclareMathOperator{\gr}{gr}

\newcommand{\GL}{\mathbf{GL}}

\newcommand{\Gr}{\mathrm{Gr}}

\newcommand{\fsl}{\mathfrak{sl}}

\newcommand{\suchthat}{:}

\newcommand{\textand}{\text{ }\mathrm{and}\text{ }}

\newcommand{\pt}{\mathrm{pt}}

\newcommand{\kk}{\Bbbk}

\renewcommand{\bv}{\mathbf{v}} % Notation for the dimension vectors of gauge nodes
\newcommand{\su}{\mathsf{u}} % Notation for the shift operators
\newcommand{\sv}{\mathsf{v}} 
 
\newcommand{\sm}{\mathsf{m}} 
\newcommand{\sn}{\mathsf{n}}
\renewcommand{\sc}{\mathsf{c}}
\newcommand{\loc}{\mathrm{loc}} % Notation for localization of a ring
\newcommand{\M}{\mathsf{M}} 
\newcommand{\sfu}{\mathsf{u}}
\newcommand{\sfr}{{\mathsf{r}}}

\newcommand{\source}{\mathsf{s}}
\newcommand{\target}{\mathsf{t}}
\newcommand{\quiver}{Q}
\newcommand{\vertexset}{I}
\newcommand{\arrowset}{E}

\newcommand{\GT}{\Lambda}
\newcommand{\tGT}{\widetilde{\Lambda}}
\newcommand{\base}{R}

\newcommand{\Partitions}{\cP}
\newcommand{\la}{\lambda}
\newcommand{\bla}{{\boldsymbol{\lambda}}}
\newcommand{\bmu}{{\boldsymbol{\mu}}}

\newcommand{\Partition}{\mathtt{Partition}}

\newcommand{\qzast}{\cA}
\newcommand{\qzastplus}{\cA^{+}}
\newcommand{\res}{\mathrm{res}}

\newcommand{\GKLO}[1]{\tilde{A}(#1)_{loc}}

\newcommand{\shu}{\cS}
\newcommand{\preshu}{\leftexp{\mathrm{big}}{\cS}}
\newcommand{\extshu}{\leftexp{\mathrm{ext}}{\cS}}
\newcommand{\shuzero}{\cS^0}

\newcommand{\Qvar}{Q}

\begin{document}

\title[Coulomb branches, quantized zastavas, Kac polynomials, and shuffle algebras]{Coulomb branches, quantized zastavas, Kac polynomials, and shuffle algebras}
\author{Dinakar Muthiah}
\address{D.~Muthiah:
School of Mathematics and Statistics,
University of Glasgow,
University Place,
Glasgow G12 8QQ,
Scotland
}
\email{dinakar.muthiah@glasgow.ac.uk}
\author{Alex Weekes}
\address{A.~Weekes: Département de mathématiques, Université de Sherbrooke, Sherbrooke J1K 2R1, Canada}
\email{alex.weekes@usherbrooke.ca}

\maketitle

\begin{abstract}
We previously constructed closed embeddings of Kac-Moody affine Grassmannian slices using fundamental monopole operators. These spaces are defined via the Braverman-Finkelberg-Nakajima construction of Coulomb branches for quiver gauge theories, and the embeddings do not quantize in general. However, there is variant of the BFN construction that produces \emph{zastava spaces}, and we show that the closed embeddings do quantize in that case. This allows us to take a limit and construct the \emph{limit quantized zastava} $\mathcal{A}$ for an arbitrary quiver. This algebra plays the role of the Borel Yangian. We also construct a positive part $\mathcal{A}^+$, which plays the role of the unipotent Yangian. By taking the limit of the monopole formula, we show that both $\mathcal{A}$ and $\mathcal{A}^+$ have Hilbert series given by a version for Hua's formula for Kac polynomials.  We also show that $\mathcal{A}^+$ is isomorphic to a certain shuffle algebra. Finally, using these results we obtain a proof of Negu{\c t}'s conjecture on the spherical generation of localized shuffle algebras via the second author's work on generators of Coulomb branches.
\end{abstract}

%%%%%%%%%%%%%%%%%%%%%%%%%%%%%%%%%%%%%%%%%%%%%%%%%%%%%%%%%%%%%%%%%

\section{Introduction}

This paper is a natural continuation of our previous work \cite{MW2}, where we studied Coulomb branches of $3d$ $\cN=4$ quiver gauge theories, as defined mathematically by Braverman, Finkelberg, and Nakajima. In the case of quivers without edge loops, these Coulomb branches are precisely \emph{Kac-Moody affine Grassmannian slices}. In \cite{MW2}, we showed that these Kac-Moody affine Grassmannian slices admit natural embeddings into one another. However, the construction of these embeddings is not directly geometric, and instead, we work indirectly using the \emph{fundamental monopole operators} to define these closed embeddings on the level of coordinate rings. We also explain that such a subtle approach is necessary because these embeddings do not quantize in general.

There is a variant of the BFN construction, which is built from the case of a quiver gauge theory with no framing: the case of \emph{zastava spaces}. In \cite{MW2}, we built closed embeddings for these zastava spaces, and as we will see in this paper, these closed embeddings \emph{do} {quantize}. Additionally, unlike \cite{MW2}, we will explicitly allow edge loops.

Our present work is about studying these closed embeddings of quantized zastava spaces. Our main results are the following:
\begin{enumerate}
 \item We show that the closed embedding of zastava spaces quantize, and we construct a limit quantized zastava, which one may think of as the  Borel Yangian associated to an arbitrary quiver.
 \item We show that the limit monopole formula for the limit quantized zastava is equal to a version of Hua's formula for Kac polynomials.
 \item We show that the positive part of the limit quantized zastava is equal to a certain shuffle algebra, and that the entire limit quantized zastava is equal to an extended shuffle algebra.
 \item We prove a conjecture of Negut's that the shuffle algebra is spherically generated after suitable localization, by building on the second author's work on spherical generation for Coulomb branches.
\end{enumerate}

\subsection{Zastava spaces and the data of an unframed quiver gauge theory}
Let us elaborate on the notion of zastava spaces. Following \cite[\S 3(ii)]{BFN2}, for any quiver gauge theory one may use a variation on the BFN construction to define a related algebraic variety $Z(\bv)$ called a \emph{zastava space}.  The BFN construction  also provides a deformation quantization $\qzast(\bv)$ of $Z(\bv)$ over $\kk[\hbar]$ called a \emph{quantized zastava space}, which is closely related to affine quantum groups. We note that framing is not important in this definition, as explained in \cite[Remark 3.15]{BFN2}.

In particular, we work with arbitrary unframed quiver, and we allow loops and multiple edges, for example:
\begin{equation*}
  {\begin{tikzcd}[sep=normal,cells={nodes={draw=black, ellipse,anchor=center,minimum height=2em}}]
          & \,\,\, \ar[d, bend right =20] \ar[d, bend left = 20] & & \\
         \arrow[out=220,in=140,loop,distance = 4em, start anchor={south west}, end anchor = {north west}] \,\,\, & \,\,\,  \ar[l] \ar[r] & \,\,\,  \ar[ul] \ar[r]& \,\,\, \ar[l,bend left = 20] \ar[l, bend right = 20] 
    \end{tikzcd}}
\end{equation*}
\vspace{-0.9cm}

By \cite[Corollary 3.4]{BFN2}, for finite ADE quivers the spaces $Z(\bv)$ are isomorphic to the original definition of zastava spaces going back to Drinfeld \cite{Finkelberg-Mirkovic}: spaces of based quasi-maps $\PP^1 \rightarrow G/B$ to flag varieties.   This is the reason for calling the spaces $Z(\bv)$ obtained from more general quivers ``zastava spaces''. Note that $Z(\bv)$ recovers zastava spaces in certain symmetric  Kac-Moody types \cite[Theorem 3.22]{BFN2}, and agrees with Mirkovi\'c's construction of zastava spaces via local spaces \cite{Dong}.

There are a number of benefits to working with the BFN construction of (quantized) zastava spaces.  For one, the $\qzast(\bv)$ are free modules over natural subrings, including  $\kk$.  For another, there are natural elements $\M_\bm(f) \in \qzast(\bv)$ called \emph{fundamental monopole operators} (FMOs), which generate the algebra \cite{Weekes,FT2}.

%%%%%%%%%%%%%%%%%%%%%%%%%%%%%%%%%%%%%%%%%%
\subsection{Main results}

 Throughout, we will fix an arbitrary quiver $\quiver = (\vertexset, \arrowset)$.
\subsubsection{Limit quantized zastava} (See \S \ref{sec:quantized_limit_zastava}.)
In finite type (i.e. an ADE Dynkin quiver), there is a surjective map 
\begin{equation}
  \label{eq:59}
    Y_\hbar(\mathfrak{b}) \twoheadlongrightarrow \qzast(\bv)
\end{equation}
 from the Borel Yangian\footnote{Precisely, we mean the Drinfeld-Gavarini integral form of $Y_\hbar(\mathfrak{b})$}
$Y_\hbar(\mathfrak{b}) \subset Y_\hbar(\fg)$ associated to the Lie algebra $\fg = \fg_\quiver$ of the quiver \cite[Corollary B.28]{BFN2}. This built upon the previous works \cite{GKLO,FR,KWWY}, but was ultimately proven using the Coulomb branch machinery.

In particular, the BFN construction produces natural quotients $\qzast(\bv)$ of the Borel Yangian, but it does not obviously produce the Borel Yangian itself. Our first result is about a construction of an algebra $\cA$ associated to an arbitrary quiver, which plays the role of the Borel Yangian. This is the topic of \S \ref{sec:quantized_limit_zastava}. 

First, as we mentioned before, the closed embeddings of Kac-Moody affine Grassmannian slices do not quantize in general \cite{MW2}. However, we show that they do in the case of zastava spaces: 
\begin{Theorem}
  For each $\bv \in \NN^\vertexset$, the algebras $\qzast(\bv)$ carry natural $\ZZ^\vertexset\times \ZZ$-gradings. For $\bv \geq \bv'$ there is a graded surjection
  \begin{equation}
    \label{eq:55}
   \Phi_{\bv, \bv'}:\qzast(\bv) \twoheadlongrightarrow \qzast(\bv') 
  \end{equation}
         which respect FMOs. 
\end{Theorem}
In the classical limit, these surjections correspond to closed embeddings $Z(\bv') \hookrightarrow Z(\bv)$ of zastava spaces. Recall that in finite ADE type $Z(\bv), Z(\bv')$ are spaces of based quasimaps $\PP^1 \rightarrow G/B$. We showed previously \cite[Theorem 2.40]{MW2} that the above closed embedding is natural from this point of view: it simply ``adds defect'' to quasimaps at $0\in \PP^1$. For general quivers, we should recover the growth structure in the local space definition of zastava spaces, see \cite[\S 1.3]{Dong}.

From the above theorem the $\qzast(\bv)$ form a directed system, and we can form their limit
\begin{equation}
  \label{eq:56}
        \qzast = \displaystyle{\lim_{\longleftarrow}}~ \cA(\bv)
\end{equation}
in the category of $\ZZ^I \times \ZZ$-graded rings. Furthermore, we prove the following.
\begin{Theorem}
  The algebra $\qzast$ is generated by FMOs, and for each $\bv \in \NN^I$, the kernel of the natural surjection
  \begin{equation}
   \qzast \twoheadlongrightarrow \qzast(\bv) 
  \end{equation}
  is  generated as an ideal by FMOs (i.e.~those FMOs lying in the kernel of the map).
\end{Theorem}

Furthermore, for each $\bv$ we define a strictly positive part $\qzastplus(\bv) \subseteq \qzast(\bv)$ which is the subalgebra generated by FMOs that have \emph{restricted dressing}, see \S \ref{ssec: restricted dressing and positive part}. We show that the maps in \eqref{eq:55} respect these subalgebras. In particular, we define:
\begin{equation}
  \label{eq:57}
        \qzastplus = \displaystyle{\lim_{\longleftarrow}}~ \qzastplus(\bv)
\end{equation}
The algebra $\qzast$ plays the role of the Borel Yangian $Y_\hbar(\fb)$ and the algebra $\qzastplus$ plays the role of the unipotent Yangian $Y_{\hbar}(\fu)$. In particular, we have a tensor product decomposition
\begin{equation}
  \label{eq:58}
\qzast = \Lambda \otimes_{\kk} \qzastplus 
\end{equation}
where $\Lambda$ is the limit of the integrable system subalgebras of $\cA(\bv)$. It plays the role of the Cartan part $Y_{\hbar}(\fh)$ of $Y_\hbar(\fb)$. In finite-type ADE, the maps \eqref{eq:59} indeed induce an isomorphism $Y_\hbar(\fb) \overset{\sim}{\rightarrow} \qzast$ identifying the subalgebras $\qzastplus$ and $\GT$ with $Y_{\hbar}(\fu)$ and $Y_{\hbar}(\fh)$. See Remark \ref{rmk:finiteADE} for more details. 

Beyond finite type, one can naively try to define the Borel Yangian by generalizing the definitions in terms of Chevalley and Cartan generators. If one works over $\kk[\hbar^{\pm 1}]$ or specializes $\hbar \in \kk^\times$, then for simple quivers (no edge loops or multiple edges) one can show that this naive Borel Yangian surjects onto $\qzast$ \cite[Theorem A and Remark 3.12]{Weekes}, though it is unclear if this map is injective. However, if one does not invert $\hbar$, then the naive Borel Yangian does \emph{not} surject onto  $\qzast$ for a general quiver. This was a major motivation for the more intrinsic definition of $\qzast$.

We establish a variety of other structural properties of these limit algebras, such as the existence of natural filtrations, and freeness over the base ring $\kk$.  See \S \ref{sec:quantized_limit_zastava} for more details.

\subsubsection{Limit monopole formula and Hua's identity for Kac polynomials}
(See \S \ref{sec:monopolesandkac}.)
The Hilbert series of (quantized) Coulomb branches are given by the celebrated \emph{monopole formula} introduced by Cremonesi-Hanany-Zafaroni \cite{CHZ2}, which is given by counting monopole operators. This formula holds for the BFN construction, see \cite[\S 2(iii)]{BFN1}, and is readily inherited by quantized zastava spaces. 

In particular, the monopole formula for quantized zastava spaces limits well and we have our second main result.

\begin{Theorem}
  The Hilbert series for the limit quantized zastava $\qzast$ is given by the following limit monopole formula:
  \begin{equation}
    \label{eq:61}
        \fJ(q,z) =  \frac{1}{(1-q)^{|\arrowset|+1}}  \frac{1}{(q)_{\boldsymbol{\infty}}}\sum_{\bla \in \Partitions^\vertexset} \frac{q^{\frac{1}{2}(\bla, \bla)}}{\prod_{i\in \vertexset} \prod_{k\geq 1} (q)_{\bm_k} }z^{|\bla|}  
  \end{equation}
\end{Theorem}
Here the sum is over $\cP^\vertexset$, the set of $I$-colored multipartitions, and the quantity $(\bla,\bla)$ comes from a pairing defined on the set of multipartitions that depends on the edges of the quiver, see \eqref{eq: pairing on tuples of partitions}.

This limit monopole formula is particularly striking because the sum which appears within is also the subject of a well-known generating function identity due to Hua \cite[Theorem 4.9]{Hua}, which expresses the Kac polynomials of the quiver $\quiver$.  Applying Hua's identity gives the following consequence, see \S \ref{ssec:hilb-for-limit-zastava} for more details:
\begin{Corollary}
    The Hilbert series of $\qzast$ is given by:
    \begin{align*}
        \fJ(q,z) &=  \frac{1}{(1-q)^{|\arrowset|+1}}  \frac{1}{(q)_{\boldsymbol{\infty}}} \operatorname{Exp}\Bigg( \frac{q}{1-q}\sum_{\alpha\in \Delta_\quiver^+} A_\alpha(q^{-1}) \bz^\alpha \Bigg) \\
        & = \frac{1}{(1-q)^{|\arrowset|+1}}  \frac{1}{(q)_{\boldsymbol{\infty}}} \prod_{\alpha \in \Delta_Q^+} \prod_{r\geq 1} \prod_{s=0}^{1 - \frac{1}{2} (\alpha,\alpha)} \frac{1}{( 1 -  q^{r-s} z^\alpha)^{t_s^\alpha}}
    \end{align*}
    where $\operatorname{Exp}$ denotes the Plethystic exponential, $\Delta_\quiver^+$ is the set of positive roots for $\quiver$ (see \S\ref{ssec: Kac and Hua}), and $A_\alpha(q) = \sum t^\alpha_s q^s$ is the Kac-polynomial corresponding to $\alpha$.
\end{Corollary}
We mention that there is a related formula for $\qzastplus$, see Corollary \ref{cor:HilbandKac}. 

%%%%%%%%%%%%%%%%%%%%%%%%%%%%%%%%%%%%%%%%%%
\subsubsection{Shuffle algebras}
(See \S \ref{sec: shuffle}.)
One of the most important tools in the study of affine quantum groups (especially for general quivers) are shuffle algebras, first introduced in the work of Feigin-Odesskii. There  is an extensive literature on this topic, see e.g.~\cite{FO,Negut1,Negut2,FT2,Tsymbaliuk2}. We will use the terminology \emph{big shuffle algebra} to refer to these rings with no wheel conditions imposed. Thus a big shuffle algebra $\preshu$ is a direct sum of (partially) symmetric polynomial rings, equipped with a shuffle product but with no wheel conditions. See \S \ref{sec: shuffle} for more details.

Using the localization theorem in equivariant Borel-Moore homology, BFN construct an embedding $\qzast(\bv) \hookrightarrow \GKLO{\bv}$ into a localized ring of difference operators, see \cite[Remark 5.23]{BFN1} and \cite[Appendix A]{BFN2}. This embedding is usually called the Gerasimov, Kharchev, Lebedev and Oblezin (GKLO) homomorphism in reference to \cite{GKLO} where related maps were first defined. 

The work of Finkelberg, Frassek and Tsymbaliuk \cite{FT2,Frassek-Tsymbaliuk-2022,Tsymbaliuk2} provides an explicit map, which we call the \emph{FFT homomorphism}:
\begin{equation}
  \label{eq:60}
 \Phi_{\bv}^+: \preshu \rightarrow \GKLO{\bv} 
\end{equation}
where $\preshu$ is a big shuffle algebra. The \emph{shuffle algebra} is a subalgebra $\shu \subseteq \preshu$ defined by imposing certain \emph{refined wheel conditions} (see \S \ref{ssec:wheel-conditions}). We define these conditions specifically so that  we have:
\begin{equation}
  \label{eq:63}
 \Phi_{\bv}^+(\shu) \subseteq \qzast(\bv)
\end{equation}
for all $\bv \in \NN^\vertexset$. Our result is the following:
\begin{Theorem}
For all $\bv \in \NN^I$, the FFT homomorphism induces a surjective map of $\ZZ^\vertexset \times \ZZ$-graded rings
  \begin{equation}
    \label{eq:11}
   \Phi^+_\bv: \shu \twoheadrightarrow \qzastplus(\bv).
  \end{equation}
 Moreover, for all $\bv', \bv \in \NN^I$ with $\bv' \leq \bv$, the following diagram commutes: 
\begin{center}
\begin{tikzcd}%[row sep=tiny]
& \qzastplus(\bv) \arrow[dd,"{\Phi^+_{\bv,\bv'}}"] \\
\cS \arrow[ur,"\Phi^+_{\bv}"] \arrow[dr,"\Phi^+_{\bv'}"'] & \\
& \qzastplus(\bv')
\end{tikzcd}
\end{center}
In particular, we obtain a natural map to the limit. This map to the limit is an isomorphism:
\begin{equation}
  \label{eq:030}
  \Phi^+ : \cS \overset{\sim}{\longrightarrow} \qzastplus 
\end{equation}
\end{Theorem}
We note that the FFT homomorphism involves reciprocals of integers, so we assume that $\QQ \subseteq \kk$ for the above results involving shuffle algebras. That said, we discuss an integral variation for $\kk = \ZZ$ in \S \ref{ssec:integralform}.

Observe that the isomorphism \eqref{eq:030} does not see the integrable system subalgebra of $\cA$, i.e. the Cartan part $\GT$. To remedy this we define an \emph{extended shuffle algebra} $\extshu$ that incorporates a Cartan part into the shuffle algebra. For this extended shuffle algebra, we prove (Theorem \ref{thm:isoextshu}) that the isomorphism \eqref{eq:030} extends to an isomorphism:
\begin{equation}
  \label{eq:64}
\Phi: \extshu \overset{\sim}{\longrightarrow} \qzast
\end{equation}

\subsubsection{Spherical generation}
(See \S \ref{sec:spherical}.)
Suppose that $\QQ \subseteq \kk$.  By analyzing the geometry of the BFN construction, the second author showed that the algebras $\qzast(\bv)$ are spherically generated after base change to $\kk[\hbar^{\pm 1}]$ \cite[\S 3]{Weekes}, i.e.~they are generated by their homogeneous components $\qzast(\bv)_0$ and $\qzast(\bv)_{\alpha_i}$ for $i \in \vertexset$. Here $0\in\ZZ^\vertexset$ is the zero element, and $\{\alpha_i \}_{i \in \vertexset}$ is the standard basis for $ \ZZ^\vertexset$. 

That said, the argument in \emph{loc.cit.}~is for simple quivers, and does not incorporate flavour symmetry deformations along the edges of the quiver. It thus does not immediately apply to the algebras $\qzast(\bv)$ in the generality considered in the current paper.  However, by using similar techniques, we establish an appropriate generalization of this result  in Theorem \ref{thm:sphericalgenv} below. 

Passing to the limit, it follows that $\qzast$ is also spherically generated after appropriate localization. Let us denote this localization by $\qzast_\loc$; see Theorem \ref{thm:sphericalgen} for the precise statement.  Using the triangular decomposition $\cA = \Lambda \otimes \qzastplus$, we prove that spherical generation passes to  $\qzastplus$:

\begin{Theorem}
Suppose that $\QQ \subseteq \kk$. Then the algebra $\qzastplus_\loc$ is spherically generated.  In particular, the same is true of the shuffle algebra $\shu \cong \qzastplus$.
\end{Theorem}

This spherical generation for shuffle algebras was conjectured by Negut in \cite[Conjecture 2.12]{Negut-2025}, and our work gives a proof of this which ultimately makes use of the geometry of the BFN construction. 

We note that, in general,  spherical generation is \emph{false} without sufficient localization. 

%%%%%%%%%%%%%%%%%%%%%%%%%%%%%%%%%%%%%%%%%%
\subsection{Possible extensions}
There is a candidate definition of Coulomb branches for quivers with symmetrizers \cite{NW}, and in particular a definition of quantized zastava spaces $\qzast(\bv)$.  We expect that the present work generalizes to this setting, but there is an obstacle to applying our current techniques: although the resulting algebras $\qzast(\bv)$ do contain FMOs $\M_\bm(f)$, they \emph{do not} generate the ring in general.\footnote{Consider the generalized Cartan matrix $\begin{pmatrix} 2 & -1 \\ -m & 2 \end{pmatrix}$. In  \cite[\S 3(ii)]{NW}, the coordinate ring of the corresponding zastava space $Z(\bv)$ with $\bv = (1,1)$ is computed. In the notation used there, the FMOs are the elements of the form $f, f\mathsf{y}_{0,1}, f\mathsf{y}_{1,0}$ and $ f \mathsf{y}_{1,1}$, for $f \in \kk[w_1,w_2]$. One can see that they do not generate the coordinate ring when $m >1$.}
Thus a different approach seems necessary.

Another natural generalization of the present work would be to the setting of \emph{(quantized) trigonometric zastava spaces}, defined via the K-theoretic version of the BFN construction. (These spaces/algebras are thus related to Coulomb branches of $4d $ $\cN=2$ theories in a generic complex structure.) Many of our results have very natural  analogs in this setting, with largely the same proof techniques.  We hope to return to this question in future work.

%%%%%%%%%%%%%%%%%%%%%%%%%%%%%%%%%%%%%%%%%%
\subsection{Relation to other work}

The Yangian $Y_\hbar(\fg)$ associated to a simple Lie algebra $\fg$ was first introduced by Drinfeld \cite{Drinfeld1}. Several geometric constructions of Yangians associated to general quivers have subsequently emerged, notably including the stable envelope construction of Maulik-Okounkov \cite{Maulik-Okounkov}, and the construction of preprojective cohomological Hall algebras (CoHAs) \cite{KS,SV1,SV2,SV3,YangZhao}. These two constructions have more recently been shown to be directed related \cite{BD,SV4}.

In particular, the CoHA construction admits numerous variations using critical cohomology,  dimensional reductions, and imposing various nilpotency conditions. The resulting algebras generally admit Hilbert series related to  Kac polynomials and Hua's identity, see e.g.~\cite{BSV,SV2,Davison}.  Comparing with the  properties of the quantized limit zastava $\qzast$, it is natural to expect that  $\qzast$ may be related to some variant of a CoHA.

While this paper was in preparation, we learned of independent on-going work of Jindal-Negu{\c t} which relates to this topic \cite{JN26_2}. They define the loop-nilpotent CoHA $\mathcal{A}^{T, \omega-\text{nilp}}_{\tilde{Q}, \tilde{W}}$ and study its relation to quantized Coulomb branches for quiver gauge theories, building upon their recent results in K-theory \cite{JN26}. Taken together, our results prove that there is an isomorphism
\begin{equation}
    \mathcal{A}^{T, \omega-\text{nilp}}_{\tilde{Q}, \tilde{W}} \xrightarrow{\sim} \qzastplus
\end{equation}
Similarly, we expect that the loop-nilpotent K-HA from \cite{JN26} would be related to the K-theoretic generalization of the present work, cf.~the previous subsection.  Jindal-Negu{\c t} also prove spherical generation of their algebras after appropriate generalization, by different techniques, and define the integral form of the loop-nilpotent CoHA via explicit integral wheel conditions \cite[Definition 2.3, Theorem 2.4]{JN26_2}.

We have also learned about on-going work of Botta-Tamagni \cite{BT}, who study geometric constructions of maps between CoHAs and quantized Coulomb branches, as well as the geometric construction of modules via spaces of quasimaps.

\subsection*{Acknowledgements}
We would like to thank Shivang Jindal, Andrei Negu{\c t}, Alexander Tsymbaliuk, and  Curtis Wendlandt for helpful conversations at key points of this project. We are also very grateful to Shivang Jindal and Andrei Negu{\c t} for sharing a preliminary version of their work with us. D.M. was supported by the Engineering and Physical Sciences Research Council grant UKRI167 ``Geometry of Double Loop Groups''. A.W.~was supported by an NSERC Discovery Grant.

%%%%%%%%%%%%%%%%%%%%%%%%%%%%%%%%%%%%%%%%%%%%%%%%%%%%%%%%%%%%%%%%%
\section{Preliminaries}
\label{sec: prelim}
Fix a commutative ring $\kk$. Throughout this paper, we will work with algebras and schemes over $\kk$. For a free $\kk$--module $V $ we will write $\dim_\kk V$ for the rank of $V$ over $\kk$. From \S \ref{sec: shuffle} onwards, we will require $\QQ \subseteq \kk$.
We write $\NN$ for the set of non-negative integers. In particular, $0 \in \NN$.

\subsection{Combinatorial preliminaries}
\subsubsection{Partitions}
\label{sssec:Partitions}

We follow the conventions of \cite[I.1]{Macdonald} regarding partitions.  Recall that a partition $\lambda = (\lambda_1, \lambda_2, \ldots)$ is a  non-increasing sequence $\lambda_1\geq \lambda_2 \geq \ldots \geq 0$ of integers such that $\lambda_r=0$ for sufficiently large $r$.  The \emph{size}  of $\lambda$ is the integer $|\lambda| = \sum_r \lambda_r$, and the \emph{length} of $\lambda$ is the number $\ell(\lambda)$ of non-zero parts.  If $|\lambda| = n$, we also write $\lambda \vdash n$. Every partition can be written in exponential notation $\lambda = (1^{\sm_1} 2^{\sm_2} \cdots)$, where $\sm_k = \#\{ r : \lambda_r = k\}$.  With this notation we have $|\lambda| = \sum_k k \sm_k$ and $\ell(\lambda) = \sum_k \sm_k$.
Occasionally we will depict partitions via Young diagrams in English notation, for example:
$$
    \ytableausetup{boxsize=1em,centertableaux}
    \lambda = (5,2,2) = (2^2 5^1) = \ydiagram{5,2,2}
$$
We denote the set of all partitions by $\Partitions$, and the set of partitions of length at most $\sv$ by $\Partitions_{\leq \sv}$. We write $\leq$ for the {dominance order} on partitions.  The sum of partitions $\la, \mu$ is defined by adding them componentwise: $\la+\mu = (\la_1+\mu_1, \la_2 +\mu_2, \ldots)$.

Following Hua \cite[Lemma 3.1]{Hua}, we define a pairing on partitions by the formula:
\begin{equation}
    \label{eq: Hua pairing}
    \langle \lambda, \mu \rangle  = \sum_r \lambda_r' \mu_r ' = \sum_{k, \ell \geq 1} \min\{k, \ell\} \sm_k \sn_\ell
    \end{equation}   
Here $\lambda' = (\lambda_1',\lambda_2',\ldots)$ and $ \mu' = (\mu_1',\mu_2',\ldots)$ denote the transposes of $\lambda$ and $\mu$, and $\lambda = (1^{\sm_1} 2^{\sm_2}\cdots) $ and $\mu = (1^{\sn_1} 2^{\sn_2}\cdots)$ in exponential notation.

%%%%%%%%%%%%%%%%%%%%%%%%%%%%%%%%%%%%%%%%%%%%%%%%%%%%%%%%%%%%%%%%%

\subsubsection{Quivers}
\label{sec: quivers}
Let $Q = (\vertexset, \arrowset)$ be a finite quiver, with $\vertexset$ being its set of vertices and $\arrowset$ its set of edges. For an edge $e \in E$, we write $\source(e)$ and $\target(e)$ for the source and target of $e$. Given an edge $e \in E$, we will sometimes write $e : i \rightarrow j$ to mean that $\source(e) = i$ and $\target(e) = j$.

For any $i,j \in \vertexset$, let $a_{ij} = \#\{ e \in \arrowset \ : \source(e) = i, \target(e) = j\}$ denote the number of edges from $i$ to $j$, and set
\begin{equation}
c_{ij} = 2\delta_{ij} - a_{ij} - a_{ji}
\end{equation}
If $Q$ contains no edge loops (i.e.~$a_{ii} = 0$ for all $i\in\vertexset$), then $(c_{ij})_{i,j \in \vertexset}$ is the Cartan matrix of a symmetric Kac-Moody algebra $\fg_Q$, whose Dynkin diagram is precisely the unoriented graph underlying $Q$. 

Let $\bv' = (\sv'_i)_{i \in \vertexset}$ and $ \bv = (\sv_i)_{i \in \vertexset}$ be elements of $\ZZ^\vertexset$. We consider the partial order on $\ZZ^\vertexset$ defined by declaring $\bv' \leq \bv$ if $\sv'_i \leq \sv_i $ for all $ i \in \vertexset$.  The symmetrized Euler form associated to $Q$ is the integral bilinear form on $\ZZ^\vertexset$ defined by:
\begin{equation}
\label{eq: Euler form}
(\bv', \bv) = \sum_{i,j \in \vertexset} c_{ij} \sv'_i \sv_j
\end{equation}
The corresponding integral quadratic form, called the Tits form of $Q$, satisfies:
$$
\tfrac{1}{2}(\bv, \bv)   = \tfrac{1}{2} \sum_{i,j \in \vertexset} c_{ij} \sv_i \sv_j  =  \sum_{i \in \vertexset} \sv^2_i  - \sum_{a \in \arrowset} \sv_{\source(a)} \sv_{\target(a)}
$$

\subsubsection{Multipartitions}

We write $\Partitions^{\vertexset}$ for the set of $I$-colored multipartitions. Let $ \bla = (\la_i)_{i \in \vertexset} \in \Partitions^{\vertexset}$ be a multipartition.  The multipartition $\bla$ determines a sequence $\bm_k = (\sm_{i, k})_{i \in \vertexset} \in \NN^\vertexset$ for $k \in \NN$ such that for each $i \in \vertexset$:
\begin{equation}
  \label{eq:10}
 \la_i = (1^{\sm_{i,1}} 2^{\sm_{i,2}} \cdots) 
\end{equation}
With this in mind, we adopt the following multipartition version of exponential notation and write: 
\begin{equation}
\label{eq: product exponential notation}
\bla = (1^{\bm_1} 2^{\bm_2} \cdots) 
\end{equation}
Given $\bv = (\sv_i)_{i \in \vertexset} \in \NN^\vertexset$, we will write $\Partitions^\vertexset_{\leq \bv}$ for the set of multipartitions $\bla = (\la_i)_{i \in \vertexset}$ such that $\ell(\la_i)\leq \sv_i$ for all $i\in \vertexset$.  We also write $|\bla| = (|\la_i|)_{i\in \vertexset} \in \NN^\vertexset$ for the tuple of sizes of $\bla$, and if $|\bla| = \bn$ then we write $\bla \vdash \bn$. We will extend the dominance order to $\Partitions^\vertexset$ by defining
\begin{equation}
\label{eq: product dominance order}
\bmu \leq \bla \ \ \Longleftrightarrow \mu_i \leq \la_i \text{ for all }i \in \vertexset.
\end{equation}

Given $\bla, \bmu \in \Partitions^\vertexset$ we define the sum $\bla + \bmu = (\lambda_i + \mu_i)_{i\in I}$ where $\lambda_i + \mu_i$ is the usual sum of partitions obtained by adding the parts of $\lambda_i$ and $\mu_i$ . Explicitly, if $\lambda_i = (\lambda_{i,1},\lambda_{i,2},\ldots)$ and $\mu_i = (\mu_{i,1},\mu_{i,2},\ldots)$, then $\lambda_i + \mu_i = (\lambda_{i,1}+ \mu_{i,1}, \lambda_{i,2} + \mu_{i,2}, \ldots)$.

Finally, for $\bla, \bmu \in \Partitions^\vertexset$ we use Hua's pairing (\ref{eq: Hua pairing}) on partitions to define a variation on the symmetrized Euler form: 
\begin{equation}
\label{eq: pairing on tuples of partitions}
( \bla, \bmu)  =  \sum_{i,j \in \vertexset} c_{ij} \langle \la_i, \mu_j \rangle   = \sum_{i,j \in \vertexset}c_{ij} \sum_{r \geq 1}  \la'_{i,r} \mu'_{j,r}  = \sum_{i,j \in \vertexset} c_{i,j}  \sum_{k,\ell \geq 1}\min\{ k, \ell\} \sm_{i,k} \sn_{i,\ell}
\end{equation}
Here $\la'_i = (\la'_{i,1}, \la'_{i,2},\ldots)$ denotes the conjugate partition to $\la_i$, and similarly for $\mu'_j$, and we have used exponential notation $\bla = (1^{\bm_1} 2^{\bm_2}\cdots)$ and $\bmu = (1^{\bn_1} 2^{\bn_2}\cdots)$ 

\begin{Remark}
\label{rmk: one-column}
We say that $\mu$ is a \emph{one-column} multipartition if it has the form $\bmu = (1^{\bc})$ for  $\bc \in \NN^\vertexset$.
 These are the minimal elements of $\Partitions^\vertexset$ under the  dominance order (\ref{eq: product dominance order}), and satisfy $ ( \bmu, \bmu) = (\bc, \bc)$.  Similarly, $\bla$ is a  \emph{one-row} multipartition if $\bla = (\bn^1)$ for some $\bn \in \NN^\vertexset$.
\end{Remark}

\subsubsection{Multisets (a.k.a. multicompositions) }
\label{subsubsec:multisets}
Fix a dimension vector $\bv \in \NN^I$. We consider tuples of multisets of the form $A = (A_i)_{i \in I}$, where for each $i \in I$, $A_i$ is a multiset supported on the set $[\bv_i]$. Consider tuples of mutisets of the form $A = (A_i)_{i \in I}$ where $A_i$ is a multisubset of $[\bv_i]$. We will use the same notation for multisets and their multiplicity functions, writing $A_i : [\bv_i] \rightarrow \NN$ for the multiplicity function of $A_i$. We will write $|A|$ for the dimension vector $(|A_i|)_{i \in I}$.
We write $A! = \prod_{i\in I} \prod_{r \in [\bv_i]} A_i(r)!$. We call $A$ as above an $I$-colored multiset supported on $[\bv]$, but we will abuse terminology and refer to $A$ as a multisubset of $[\bv]$ or simply a multiset.

Given two multisets $A$ and $B$ supported on $[\bv]$, we can form their sum $A + B$, which is also a multiset supported on $[\bv]$ characterized by requiring $(A+B)_i(r) = A_i(r) + B_i(r)$ for all $i\in I$, and $r \in [\bv_i]$.

Every multiset $A$ gives rise to an associated multipartition $\bla = (\lambda_i)_{i \in I}$ where $\lambda_i$ is the partition obtained by sorting $(A_i(1),\ldots, A_i(\sv_i))$ in decreasing order. We write $\bla = \Partition(A)$ to denote that $\bla$ is the multipartition associated to $A$. Conversely, each multipartition $\bla = (\lambda_i)_{i \in I}$ can be viewed as a multiset by defining
\begin{equation}
  \label{eq:17}
  \bla_i(r) = \lambda_{i,r}
\end{equation}
for all $i \in I$ and $r \in [\sv_i]$.

Observe that that the multisets described here are exactly the multipartition analogue of compositions. For that reason, we will sometimes refer to multisets as \emph{multicompositions}. The two terms are synonymous, but depending on the context one of the two terms will be more descriptive.

\subsection{BFN construction of quantized zastava}
We will briefly recall the Braverman-Finkelberg-Nakajima (BFN) construction of Coulomb branches of quiver gauge theories. We refer the reader to the original papers \cite{BFN1,BFN2} and the overview from our previous paper \cite[\S 3]{MW2}.

\subsubsection{The data of an unframed quiver gauge theory}

Fix a dimension vector $\bv = (\sv_i)_{i \in \vertexset} \in \NN^\vertexset$, and for all $i \in \vertexset$ write $V_i = \CC^{\sv_i}$. Define:
\begin{equation}
\label{eq:01}
\bG = \prod_{i \in \vertexset} \GL(V_i), \qquad \bN = \bigoplus_{e \in \arrowset}  \Hom_\CC( V_{\source(e)}, V_{\target(e)} ), 
\end{equation}
Then $\bG$ acts naturally on $\bN$, and we can apply the BFN construction \cite{BFN1} to the pair $(\bG,\bN)$.

We will also incorporate \emph{flavour symmetry} \cite[\S 3(viii)]{BFN1}, and to this end consider the following group, which depends only on our choice of quiver $Q$:
\begin{equation}
\bF = \prod_{e \in \arrowset} \CC^\times
\end{equation}
The group $\bF$ acts on $\bN$, commuting with the action of $\bG$: each factor of $\CC^\times$ in $\bF$ acts by rescaling the summand  $ 
\Hom(V_{\source(e)}, V_{\target(e)}) \subseteq \bN$ with weight 1, and acts trivially on all other summands.

\subsubsection{The quantized Coulomb branch of an unframed quiver gauge theory}
Associated to this data, one can form the space of triples $\cR_{\bG,\bN}$ (see \cite[\S 2(i)]{BFN1} and also \cite[\S 3.1]{MW2} for the definition). The key facts we need are the following. There is a map $\cR_{\bG,\bN} \rightarrow \Gr_{\bG}$ from the space of triples to the affine Grassmannian of $\bG$. The space $\cR_{\bG,\bN}$ carries an action of $\bG_{\cO} \rtimes \CC^\times \times \bF$, were $\bG_{\cO}$ is the positive loop group of $\bG$, and the middle $\CC^\times$ factor is the so-called loop rotation.  The \emph{quantized Coulomb branch} (over base $\kk$) is the vector space
\begin{equation}
  \label{eq:50}
  H^{\bG_\cO \rtimes \CC^\times \times \bF}_{\bullet}(\cR_{\bG,\bN},\kk)
\end{equation}
which carries an associative algebra structure. 

\begin{Remark}
    Following the conventions of \cite[\S 2(i)]{BFN1}, we let the loop rotation $\CC^\times$ act on $\bN$ with weight $\frac{1}{2}$. This explains the appearance of $\frac{1}{2}\hbar$ in the definition of $\base$ in \eqref{eq: base ring} below. (Technically we should thus replace $\CC^\times$ by its double-cover, but we will ignore this.) 
\end{Remark}

\subsubsection{Quantized zastava}

Because $\bG$ is a product of general linear groups, there is a notion of the positive part $\Gr^+_{\bG} \subseteq \Gr_{\bG}$ of the affine Grassmannian of $\bG$. Correspondingly there is a positive part $\cR^+_{\bG,\bN} \subseteq \cR_{\bG,\bN}$ defined as the preimage $\Gr_{\bG}^+$, see \cite[\S 3(ii)]{BFN2}. Moreover,
\begin{equation}
  \label{eq:51}
  H^{\bG_\cO \rtimes \CC^\times \times \bF}_{\bullet}(\cR^+_{\bG,\bN}, \kk)
\end{equation}
is a subalgebra of $H^{\bG_\cO \times \CC^\times \times \bF}_{\bullet}(\cR_{\bG,\bN},\kk)$ called the \emph{quantized zastava space} of degree $\bv$. We write
\begin{equation}
  \label{eq:52}
  \qzast(\bv) = H^{\bG_\cO \rtimes \CC^\times \times \bF}_{\bullet}(\cR^+_{\bG,\bN},\kk)
\end{equation}
Occasionally, we will instead write $\qzast_\kk(\bv)$ to explicitly indicate the choice of base ring $\kk$.

\begin{Remark}
If we specialize the equivariant parameters for $\CC^\times \times \bF$ to zero, we obtain a commutative algebra, which is the coordinate ring of the zastava space $Z(\bv)$ (see \cite[\S 3(ii)]{BFN2}).
\end{Remark}

\subsubsection{Integrable system}
\label{sec: integrable system}
To fix notations, we first consider the ring
\begin{equation}
\label{eq: base ring}
\base = H_{\CC^\times\times \bF}^\bullet(\pt, \kk)  = \kk[\tfrac{1}{2} \hbar, \theta_e : e \in \arrowset]
\end{equation}
It will serve as a base ring for the algebras we consider below. Note that $\base$ depends only on the quiver $Q$, and not on any choice of dimension vector.

Given a dimension vector $\bv$,  consider the group $\bG = \prod_i \GL(V_i)$ as in (\ref{eq:01}).  We  fix a maximal torus $\bT = \prod_{i \in \vertexset} \bT_i$ of $\bG$, i.e.~the product of diagonal matrices with respect to some chosen basis.  Then we may identify the following equivariant cohomology ring with a polynomial ring:
\begin{equation}
\label{eq: GT with no Weyl}
\tGT(\bv) = H_{\bT\times \CC^\times\times \bF}^\bullet(\pt, \kk) = \kk[w_{i,r}, \tfrac{1}{2}\hbar, \theta_e \ : \ i \in \vertexset, 1 \leq r \leq \sv_i, e \in \arrowset]
\end{equation}
The Weyl group of $\bG$ is the product $S_{\bv} = \prod_{i \in \vertexset} S_{\sv_i}$ of symmetric groups $S_{\sv_i}$, which acts on $\tGT(\bv)$ with $S_{\sv_i}$ permuting the variables $w_{i,r}$ for $1 \leq r \leq \sv_i$.   Replacing $\bT$ with $\bG$ in the equivariant cohomology above corresponds to taking Weyl invariants, giving a subring:
\begin{equation}
\label{eq: GT subalgebra}
\GT(\bv) = H_{\bG\times \CC^\times\times \bF}^\bullet(\pt, \kk)  =\tGT(\bv)^{S_\bv}
\end{equation}
This is a tensor product of symmetric polynomial rings (over the base $\base$).

By \cite[\S 3(vi)]{BFN1}, there are natural inclusions of rings:
\begin{equation}
\label{eq: integrable system}
\base \subseteq \ \GT(\bv) \ \subseteq \ \qzast(\bv)
\end{equation}
The subring $\base$ is central in $\qzast(\bv)$, while the subring $\GT(\bv)$ is commutative but not central.  The left action of $\GT(\bv)$ on $\qzast(\bv)$ is by cap product, and $\qzast(\bv)$ is free as a left (or right) module.

\begin{Remark}
    By the results of \cite[\S 2(iii)]{BFN1}, viewed as a left module over $\GT(\bv)$ the ring $\qzast(\bv)$ is isomorphic to an infinite direct sum of cohomology rings $H_{\bG \times \CC^\times \times \bF}^\bullet(\bG / \bP, \kk)$ over various parabolic subgroups $\bP \subseteq \bG$. See also \S \ref{sec: filtration} below. For quiver gauge theories the group $\bG$ is a product of factors $\GL(V_i)$, so these cohomology rings are all free modules over $\GT(\bv)$, for any base ring $\kk$. They are also free as modules over $\kk$ itself.   It follows that these rings behave well under base change.  In particular, we have 
    \begin{equation}
        \qzast_\kk(\bv) = \qzast_\ZZ(\bv) \otimes_\ZZ \kk
    \end{equation}
    and a natural inclusion $\qzast_\ZZ(\bv) \subset \qzast_\QQ(\bv)$.
\end{Remark}

\subsubsection{GKLO embedding}
\label{ssec: GKLO embedding}
Consider an $\base$--algebra $\widetilde{\cA}(\bv)$ of difference operators, generated by variables $w_{i,r}$ and $\sfu_{i,r}$ for each $i \in \vertexset$ and $1 \leq r\leq \sv_i$, with relations
$$
\sfu_{i,r} w_{j,s} = (w_{j,s} + \hbar \delta_{i,j} \delta_{r,s}) \sfu_{i,r}, \qquad [\sfu_{i,r}, \sfu_{j,s}] = [w_{i,r}, w_{j,s}] = 0
$$
We also consider its Ore localization $\GKLO{\bv}$ at the multiplicative set generated by
$$
\left\{ w_{i,r} - w_{i,s} + n \hbar \ : \ i \in \vertexset, 1 \leq r \neq s \leq \sv_i, n \in \ZZ \right\}
$$
There is a natural action of the group $S_{\bv} =\prod_{i \in \vertexset} S_{\sv_i}$ on $\GKLO{\bv}$ given by permuting the variables $w_{i,r}$ and $\sfu_{i,r}$.
By \cite[\S 5]{BFN1}, there is an embedding of rings (denoted $\bz^\ast \iota_\ast^{-1}$ in \emph{loc. cit.}) 
\begin{equation}
\label{eq: GKLO}
\qzast(\bv) \ \hookrightarrow \ \GKLO{\bv}^{S_{\bv}}
\end{equation}
which is the identity on the subring $\GT(\bv)$ from (\ref{eq: GT subalgebra}). This map is often called the \textbf{GKLO embedding}, since it generalizes the work of Gerasimov, Kharchev, Lebedev and Oblezin \cite{GKLO}.

\subsubsection{Grading by $\ZZ^\vertexset$}
\label{ssec: gradings 0}

The ring $\qzast(\bv)$ has a $\ZZ^\vertexset$-grading coming from the connected components of $\cR^+_{\bG,\bN}$, defined via the general procedure outlined in \cite[\S 3(v)]{BFN1} and  \cite[Remark 3.2]{BFN2}.  

In more detail, the connected components $\Gr_{\GL(V_i)}^{+,k_i} \subset \Gr_{\GL(V_i)}^+$ are labelled by integers $k_i \in \NN$, with $\Gr^{\leq \lambda_i}_{\GL(V_i)}$ lying in component $k_i = |\lambda_i|$.  Taking products over $i \in \vertexset$, the components $\Gr_{\bG}^{+, \bk} \subset \Gr_\bG^+$ are labelled by tuples $\bk \in \ZZ^{\vertexset}_{\geq 0}$. Their preimages $\cR^{+,\bk}_{\bG,\bN}$ under $\cR^+_{\bG,\bN} \rightarrow \Gr^+_{\bG}$ provide the connected components of $\cR^+_{\bG,\bN}$ (which is homotopy equivalent to $\Gr_\bG$).  They induce  a $\ZZ^\vertexset$-graded ring structure:
\begin{equation}
    \label{eq: Z^I grading}
    \qzast(\bv) = \bigoplus_{\bk \in \ZZ^\vertexset} \qzast(\bv)_\bk
\end{equation}
with homogeneous components $\qzast(\bv)_\bk = H_\bullet^{(\bG_\cO \rtimes \CC^\times)\times \bF} (\cR^{+,\bk}_{\bG,\bN})$.

\begin{Remark}
    We think of $\ZZ^\vertexset$ as a root lattice (associated to $\fg_\quiver$ in the case where $\quiver$ has no loops), and will sometimes denote its standard basis by $\{ \alpha_i\}_{i \in \vertexset}$.
\end{Remark}

\subsubsection{Filtration by multipartitions}
\label{sec: filtration}
Recall from \S\ref{sec: quivers} that $\Partitions_{\leq \bv}^\vertexset$ denotes the set of  tuples of partitions $\bla = (\la_i)_{i \in \vertexset}$ of length at most $\bv \in \NN^\vertexset$.  We may equip this set with the product dominance order $\bmu \leq \bla$ from (\ref{eq: product dominance order}).

The set $\Partitions_{\leq \sv_i}$ of partitions of length $\leq \sv_i$ may be identified with the set of \emph{non-negative} dominant coweights of $\GL(V_i)$, by identifying each partition $\lambda_i  \in \Partitions_{\leq \sv_i}$ with a tuple of integers, some of which may be zero:
\begin{equation}
\la_i \mapsto (\la_{i,1},\ldots, \la_{i,\sv_i}) \in \ZZ^\vertexset_{\geq 0}
\end{equation}
The partitions $\Partitions_{\leq \sv_i}$ thus label the $\GL(V_i)_{\cO}$--orbits $\Gr_{\GL(V_i)}^{\la_i} $ inside $\Gr_{\GL(V_i)}^+$, and by taking products the $\bG_\cO$--orbits $\Gr_\bG^{\bla}$ in $\Gr_\bG^+$ are labelled by multipartitions $\bla \in \Partitions_{\leq \bv}^\vertexset$.  Their Zariski closures $\Gr_\bG^{\leq \bla} = \bigsqcup_{\boldsymbol{\mu} \leq \bla} \Gr_\bG^{\boldsymbol{\mu}}$ are spherical Schubert varieties. Define $\cR^{\bla}_{\bG,\bN}$ and $\cR^{\leq \bla}_{\bG,\bN}$ as their preimages under the map $\cR^+_{\bG,\bN} \rightarrow \Gr_\bG^+$, and put
\begin{equation}
\qzast(\bv)^{\bla}  = H_\bullet^{(\bG_\cO \rtimes \CC^\times) \times \bF}( \cR^{\bla}_{\bG,\bN}, \kk), \qquad \qzast(\bv)^{\leq \bla}  =   H_\bullet^{(\bG_\cO \rtimes \CC^\times) \times \bF}( \cR^{\leq \bla}_{\bG,\bN}, \kk)
\end{equation}
Then by construction $\bmu \leq \bla$ implies $\qzast(\bv)^{\leq \bmu} \subseteq \qzast(\bv)^{\leq \bla}$.  

\begin{Proposition}[\mbox{\cite[Proposition 6.1]{BFN1}}]
\label{prop: filtration defined}
The subspaces $\qzast(\bv)^{\leq \bla}$ for $\bla \in \Partitions_{\leq \bv}^\vertexset$ make $\qzast(\bv)$  into a filtered graded ring, filtered by the poset $(\Partitions_{\leq \bv}^\vertexset, \leq)$.
That is:
\begin{enumerate}
    \item We have $1 \in \qzast(\bv)^{\leq \mathbf{0}}$ and  $\qzast(\bv)^{\leq \bmu} \cdot \qzast(\bv)^{\leq \bla} \ \subseteq \ \qzast(\bv)^{\leq \bmu + \bla}$  for all $\bla,\bmu \in \Partitions_{\leq \bv}^\vertexset$.
    
    \item For each $\bk \in \NN^\vertexset$, we have a decomposition of the homogeneous component
    $$
    \qzast(\bv)_\bk = \bigcup_{\substack{\bla \in \Partitions^\vertexset_{\leq \bv},~ |\bla| = \bk}} \qzast(\bv)^{\leq \bla}
    $$
    In particular $\qzast(\bv)_{\bk} = \qzast(\bv)^{\leq \bk}$, where on the right side we interpret $\bk$ as a tuple of one-row partitions (i.e.~maximal for the dominance order).
\end{enumerate}
\end{Proposition}

We warn the reader that a \emph{filtered graded ring} is \emph{not} a filtered ring in the usual sense. Instead, we mean rather that the homogeneous components of the graded ring  $\qzast(\bv)$ are filtered  as in part (2) above, compatibly with the ring structure as in part (1). Let us  denote:
\begin{equation}
    \qzast(\bv)^{<\bla} = \sum_{\substack{ \bmu < \bla}} \qzast(\bv)^{\leq \bmu}
\end{equation}
Then the associated graded ring can be explicitly described \cite[Proposition 6.2]{BFN1}. 
\begin{equation}
\label{eq: gr ring}
\gr \qzast(\bv) = \bigoplus_{\bla \in \Partitions_{\leq \bv}^{\vertexset}} \raisebox{.1em}{$\qzast(\bv)^{\leq \bla}$} \big/ \raisebox{-0.1em}{$\qzast(\bv)^{<\bla}$} \cong \bigoplus_{\bla \in \Partitions_{\leq \bv}^\vertexset} \qzast(\bv)^{\bla}  = \bigoplus_{\bla \in \Partitions_{\leq \bv}^{\vertexset}} \GT^{\bla}(\bv) \cdot \sfr_{\bla}
\end{equation}
On the right hand side $\sfr_{\bla} = [\cR_{\bG,\bN}^\bla] \in \qzast(\bv)^\bla$ denotes the fundamental class, which is acted on (on the left, via cap product) by elements of the equivariant cohomology ring
\begin{equation}
\label{eq: GT for bla}
\GT^{\bla}(\bv) = H_{\operatorname{Z}_\bG(\bla) \times \CC^\times \times \bF}^\bullet(\pt, \kk) \cong  \tGT(\bv)^{S_{\bla}}
\end{equation}
Here $Z_\bG(\lambda)$ denotes the centralizer of $\lambda$ in $\bG$, which is a Levi subgroup having Weyl group $S_{\bla} \subseteq S_\bv$ the stablizer of $\bla$.  The relations in this associated graded ring are given explicitly \cite[\S 4(ii)]{BFN1}:
\begin{equation}
\label{eq: rels in gr}
f(w) \sfr_\bla \cdot g(w) \sfr_\bmu = a_{\bla, \bmu} f(w) g(w+\hbar \bla) \sfr_{\bla+\bmu}
\end{equation}
Here $a_{\bla, \bmu} \in \GT^{\bla+\bmu}(\bv)$ is a certain explicit product \cite[(4.7)]{BFN1}:
\begin{equation}
    a_{\bla, \bmu} = \prod_{e \in \arrowset} \prod_{r \in [\sv_i], s\in [\sv_j]}a_{\bla,\bmu}(e;r.s)
\end{equation}
Here, to  define $a_{\bla,\bmu}(e; r,s)$ on the right side, fix $e\in E$ and $r\in [\sv_i], s\in [\sv_j]$. To simplify notation let us write $e: i \rightarrow j$, and define integers $\ell = \lambda_{j,s} - \lambda_{i,r}$ and $m = \mu_{j,s}-\mu_{i,r}$. Then:
\begin{equation}
     a_{\bla,\bmu}(e;r,s) = \left\{ 
     \begin{array}{cl} 
     \displaystyle{\prod_{p=1}^{\min\{|\ell|,|m|\}}} \big(w_{j,s}-w_{i,r}+\theta_e + (\ell-p+\tfrac{1}{2})\hbar\big), & \text{if } \ell \geq 0 \geq m, \\ 
     \displaystyle{\prod_{p=1}^{\min\{|\ell|,|m|\}}} \big(w_{j,s}-w_{i,r}+\theta_e + (\ell+p-\tfrac{1}{2})\hbar\big), & \text{if } \ell \leq 0 \leq m, \\ 
     1, & \text{otherwise.}
    \end{array} \right.
\end{equation}

\begin{Remark}
    \label{rmk:aproperties}
    For most purposes, we will not need the precise expression above, but it will be important that:
    \begin{enumerate}
        \item $a_{\bla,\bmu} = 1$ if and only if $\la_{j,s} - \la_{i,r}$ and $\mu_{j,s} - \mu_{i,r}$ have the same sign in the weak sense (i.e.~are both $\geq 0$ or are both $\leq 0$) for all arrows $e:i\rightarrow j$ and all $1\leq r \leq \sv_i$, $1 \leq s\leq \sv_j$.

        \item $a_{\bla,\bmu}$ only involves variables $w_{i,r}$ for $i\in \vertexset$ and $1\leq r \leq \min\{\ell(\lambda_i), \ell(\lambda_j)\}$. In particular it is independent of $\bv$.
    \end{enumerate}
\end{Remark}

\begin{Remark}
\label{rmk: filtration by all partitions}
We  may also consider $\qzast(\bv)$ to be a filtered graded ring for the larger poset $(\Partitions^{\vertexset}, \leq)$.  Indeed, for any $\bla \in \Partitions^{\vertexset}$ define
$$
\qzast(\bv)^{\leq \bla} = \left\{ \begin{array}{cl} \qzast(\bv)^{\leq \bla},& \text{if } \bla \in \Partitions_{\leq \bv}^{\vertexset}, \\ 0, & \text{otherwise.} \end{array} \right.
$$
This is well-defined because the subset $\Partitions_{\leq \bv}^\vertexset \subset \Partitions^{\vertexset}$ has the following basic property:
\begin{equation}
\label{eq: partitions length decreases under dominance}
\text{For any } \bla, \bmu \in \Partitions^\vertexset \text{ with } \bmu \leq \bla, \text{ if } \bla \notin \Partitions_{\leq \bv}^\vertexset \text{ then } \bmu \notin \Partitions_{\leq \bv}^{\vertexset}
\end{equation}
Indeed, if a partition $\la_i$ has length $\ell(\lambda_i)> \sv_i$, then any partition $\mu_i \leq \la_i$ also has $\ell(\mu_i)>\sv_i$.
\end{Remark}

\subsubsection{Defining the filtration via GKLO}

There is a another way to characterize the filtration algebraically via the GKLO embedding, and this is how we will work with the filtration in practice.
First, we introduce the following notation. For $A$ an $I$-colored multiset supported on $[\bv]$, we denote: 
\begin{align}
  \label{eq:14}
  \sfu^A = \prod_{i \in I} \prod_{r \in [\bv_i]} u_{i,r}^{A_i(r)} \in \GKLO{\bv}
\end{align}
In particular, for $\bla \in \Partitions^\vertexset_{\bv}$, we can form the element $\sfu^\bla$.
For each $\bk \in \NN^I$, every element of $\left(\GKLO{\bv}\right)_{\bk}$ is of the form
    $\sum_{A: |A| = \bk} f_A u_A$
where the sum is over multicompositions of $\bk$ supported on $[\bv]$.

We define a filtration of $\GKLO{\bv}^{S_{\bv}}$ as follows. For $\bla \in \Partitions^\vertexset_{\bv}$, we define
\begin{equation}
   \left(\GKLO{\bv}^{S_{\bv}}\right)^{\leq \bla}
\end{equation}
to be the set of elements 
\begin{equation}
    \sum_{A} f_A u_A \in \left(\GKLO{\bv}^{S_{\bv}}\right)_{|\bla|}
\end{equation}
satisfying $f_{\bmu} = 0$ for $\bmu \nleq \bla$. Note that by $S_{\bv}$ symmetry, this also implies $f_A = 0$ for any multiset with associated partition $\bmu$ having $\bmu \nleq \bla$. The construction of the GKLO embedding is via torus localization at fixed points, and the condition $\bmu \nleq \bla$ exactly excludes the fixed points lying outside of $\cR^\bla_{\bG,\bN}$, so the following is immediate.
\begin{Proposition}
For all $\bla \in \cP^\vertexset_{\bv}$, we have:
\begin{equation}
    \qzast(\bv)^{\leq \bla} = \qzast(\bv) \cap  \left(\GKLO{\bv}^{S_{\bv}}\right)^{\leq \bla}
\end{equation}
\end{Proposition}

\subsection{Fundamental monopole operators}
As in Remark \ref{rmk: one-column}, we consider {one-column multipartitions}  $\bmu = (1^{\bm})$ for $\bm = (\sm_i)_{i \in \vertexset} \in \NN^\vertexset$. Then $\bmu \in \Partitions^{\vertexset}_{\leq \bv}$ if and only if $\bm \leq \bv$. Viewed as coweights for $\bG = \prod_{i} \GL(V_i)$,  these elements $\mu$ correspond to tuples of fundamental coweights, and the corresponding $\bG_\cO$--orbits $\Gr_\bG^{\leq \bmu} = \Gr_\bG^{\bmu}$ are closed. Thus by \cite[\S 6(ii)]{BFN1}, for any 
$$
f \in \GT^{\bmu}(\bv) =  \tGT(\bv)^{S_{\bm} \times S_{\bv-\bm}},
$$
we may consider a corresponding element $f \cdot [ \cR_{\bG,\bN}^{\leq \bmu}] \in \qzast(\bv)^{\leq \bmu}$, and every element of $\qzast(\bv)^{\leq \bmu}$ has this form.  Following the notation of \cite{MW2}, we denote this element by 
\begin{equation}
\M_\bm(f) \ \in \ \qzast(\bv)
\end{equation}
and refer to these elements as \textbf{(dressed) fundamental monopole operators}, or \textbf{FMOs} for short. Their images under the GKLO embedding (\ref{eq: GKLO}) are given explicitly, similarly to \cite[\S A(ii)]{BFN2}:
\begin{equation}
\label{eq: FMOs under GKLO}
\M_\bm(f) \ \mapsto \ \sum_{\substack{\Gamma = (\Gamma_i)_{i \in I} \\ \Gamma_i \subseteq [\sv_i],  \# \Gamma_i = \sm_i}} f\vert_\Gamma \cdot  \frac{\prod_{ e \in \arrowset} \prod_{r\in \Gamma_{\source(e)}, s\notin \Gamma_{\target(e)}} (w_{\target(e),s} - w_{\source(e),r}+\theta_e - \tfrac{1}{2} \hbar)}{\prod_{i \in I} \prod_{r \in \Gamma_i, s\notin \Gamma_i}(w_{i,r} - w_{i,s})}  \sfu_{\Gamma} 
\end{equation}
Here we borrow the notations $f|_\Gamma$ and $u_\Gamma$ from  \cite{MW2}: for each tuple $\Gamma = (\Gamma_i)_{i \in I}$ above, let $\sigma = (\sigma_i) \in S_\bv$ be an element such that $\sigma_i([\sv_i]) = \Gamma_i$. Then we define $f|_\Gamma = \sigma(f)$, an element of $\tGT(\bv)$.  Finally we define $\sfu_\Gamma = \prod_{i \in I} \prod_{r \in \Gamma_i} \sfu_{i,r}$.
\begin{Remark}
When $\bm = \mathbf{0}$ we interpret $\M_\mathbf{0}(f) = f \in \GT(\bv)$. In general there is a linearity property over the subring $\GT(\bv) \subset \qzast(\bv)$: for any $\mathbf{0} \leq \bm \leq \bv$ with corresponding $\bmu = (1^{\bm})$,
$$
\M_\bm(fg ) = f \cdot \M_\bm(g), \quad \text{for all } f\in \GT(\bv), g \in \GT^{\bmu}(\bv)
$$
We note however that $\M_\bm(gf) \neq \M_\bm(g) \cdot f$ in general.
\end{Remark}

The FMOs $\M_\bm(f)$ are particularly interesting elements of $\qzast(\bv)$. Not only do we know their explicit images \eqref{eq: FMOs under GKLO} under the GKLO homomorphism, but they also generate the ring $\qzast(\bv)$: 
\begin{Proposition}[\mbox{\cite[Proposition 3.1]{Weekes}}]
\label{generation by FMOs}
\mbox{}

\begin{enumerate}[label=(\alph*),topsep=3pt,itemsep = 5pt]
\item The associated graded ring $\gr \qzast(\bv)$ from \S \ref{sec: filtration} is generated by the elements $f \sfr_{\bmu}$ with $\bmu = (1^{\bc})$, taken over all $\mathbf{0} \leq \bc\leq \bv$ and all $f \in \GT^{\bmu}(\bv)$.

\item $\qzast(\bv)$ is generated by the FMOs $\M_\bc(f)$, taken over all $\mathbf{0} \leq \bc\leq \bv$ and all $f \in \GT^{\bmu}(\bv)$  where $\bmu = (1^{\bc})$.
\end{enumerate}
\end{Proposition}
We include a proof here which is more constructive than the one from \cite{Weekes}:
\begin{proof}
Part (b) follows from part (a), since $\M_\bc(f) \in \qzast(\bv)$ is a lift of $f \sfr_{\bmu} \in \gr \qzast(\bv)$.  So let us prove part (a):

Given a tuple $\bla \in \Partitions_{\leq \bv}^\vertexset$ of partitions, we extract $\bmu = (\mu_i)_{i \in  \vertexset}$ and $\boldsymbol{\kappa} = (\kappa_i)_{i \in  \vertexset}$ as follows: define $\mu_i$ to be the one-column partition given by the first (largest) column of the partition $\lambda_i$, and define $\kappa_i$ by removing the first column from $\lambda_i$. Then $\lambda_i = \mu_i + \kappa_i$. For example:
$$
\ytableausetup{boxsize=1em,centertableaux}
\text{if }\  \lambda_i = \ydiagram{4,3,1}, \qquad \text{ then }\ \mu_i = \ydiagram{1,1,1} \  \ \text{ and } \ \kappa_i = \ydiagram{3,2}
$$
Note that $\bmu = (1^{\bc})$ where $\mathbf{0} \leq \bc \leq \bv$ encodes the lengths of the first columns of $\bla$.  

To prove part (a) it suffices, by induction on the number of columns of $\bla$, to show that every element of $\GT^\bla(\bv)\cdot \sfr_\bla$ can be written as a sum of elements of the form $f(w) \sfr_\bmu \cdot g(w) \sfr_{\boldsymbol{\kappa}}$.  Because of the formula \eqref{eq: rels in gr} for this product, this amounts to the following two claims:

First, we claim that $\sfr_{\bmu} \sfr_{\boldsymbol{\kappa}}  = \sfr_{\bla}$. In other words, we claim that $a_{\bmu, \boldsymbol{\kappa}} = 1$ in the formula (\ref{eq: rels in gr}).  To prove this we must verify that the integers
$$
\mu_{\target(e),s} - \mu_{\source(e), r} \ \ \ \text{ and } \ \ \ \kappa_{\target(e), s} - \kappa_{\source(e), r}
$$
are either both $\geq 0$ or are both $\leq 0$, for all $e \in \arrowset$ and all $1 \leq r \leq \sv_{\source(a)}$, $1 \leq s \leq \sv_{\target(a)}$. Since each $\mu_i = (1^{\sc_i})$ is a one-column partition of length $\sc_i$, we find:
$$
\mu_{\target(e),s} - \mu_{\source(e), r} = \left\{ \begin{array}{rl} 1, & s \leq \sc_{\target(e)} \text{ and } r > \sc_{\source(e)}, \\ -1, & s > \sc_{\target(e)} \text{ and } r \leq \sc_{\source(e)}, \\ 0, & \text{else}. \end{array} \right.
$$
Note that only the first two cases are relevant.  In the first case, since $r > \sc_{\source(e)}$  and $\kappa_{\source(e)}$ has length at most $\sc_{\source(e)}$, we have $\kappa_{\source(e),r} = 0$. Therefore $\kappa_{\target(e), s} - \kappa_{\source(e), r} = \kappa_{\target(e), s} \geq 0$ has the same sign as $\mu_{\target(e),s} - \mu_{\source(e), r}$.  The second case is similar.

Second, we claim that the following map of partially symmetric polynomial rings is surjective:
\begin{equation}
\label{eq: psym polys}
\GT^{\bmu}(\bv) \otimes \GT^{\boldsymbol{\kappa}}(\bv)  \twoheadlongrightarrow \GT^{\bla}(\bv),  \ \ \  f(w) \otimes g(w) \longmapsto f(w) g(w+\hbar \bmu)
\end{equation}
By taking tensor products, it is enough to prove this in the case where our quiver $Q$ has a single vertex, which is straightforward.  
\end{proof}
Iterating the procedure from the proof of the theorem, we obtain a recipe to construct lifts of the elements $f \sfr_\bla$: we separate $\bla$ into its constituent tuples of columns lengths $\bc_1,\bc_2,\ldots, \bc_k$, with corresponding one-column multipartitions $\bmu_s = (1^{\bc_s})$ Then every element $f \sfr_\bla$ can be written as a sum of terms of the form
\begin{equation}
    \M_{\bc_1}(f_1) \M_{\bc_2}(f_2) \cdots \M_{\bc_k}(f_k)    
\end{equation}
Indeed, consider the generalization of \eqref{eq: psym polys}:
\begin{equation}
     \bigotimes_{s=1}^k \GT^{\bmu_s}(\bv)  \twoheadlongrightarrow \GT^\bla(\bv), \quad f_1(w) \otimes \cdots \otimes f_k(w) \longmapsto  \prod_{s=1}^k f_s\big( w + \hbar(\bmu_1+\ldots+\bmu_{s-1}) \big)
\end{equation}
Given $f \in \GT^\bla(\bv)$,  choose some preimage $\sum f_1(w)\otimes \cdots \otimes f_k(w)$ under this map.  (Here we mimic Sweedler notation, for brevity; note that a preimage will generally not be a simple tensor.)  Then we may define a corresponding  dressed monopole operator:
\begin{equation}
    \label{eq: general MOs}
    \M_\bla(f) := \sum \M_{\bc_1}(f_1) \M_{\bc_2}(f_2) \cdots \M_{\bc_k}(f_k)    
\end{equation}
This element defines a lift of $f \sfr_\bla$, but is non-canonical, as it depends on the choice of lift of $f$ above. If we fix such an element for each $\bla$ and $f$, then we will obtain a \emph{basis} $\{ \M_\bla(f)\}$ for $\qzast(\bv)$ over $\kk$.

\begin{Corollary}
\label{cor: filtration generated by FMOs}
For each $\bla \in \Partitions_{\leq \bv}^\vertexset$, the filtered piece of $\qzast(\bv)^{\leq \bla}$ is spanned by all products 
\begin{equation}
\label{eq: prod of monopoles}
\M_{\bc_1}(f_1) \M_{\bc_2}(f_2) \cdots \M_{\bc_k}(f_k)
\end{equation}
where  $\mathbf{0} \leq \bc_j \leq \bv$ for each $j \in I$, the corresponding multipartitions $\bmu_j = (1^{\bc_j})$ satisfy
$$
\bmu_1 + \ldots + \bmu_k \leq \bla,
$$
and where $f_j \in \GT^{\bmu_j}(\bv)$.

Furthermore, the coefficient of $u^\bla$ in (\ref{eq: prod of monopoles}) under the GKLO embedding is:
\begin{equation} 
  \label{eq:20}
  f \cdot \frac{\prod_{e \in E} \prod_{r \in [\bv_{s(e)}],s \in [\bv_{t(e)}], \bla_{s(e)}(r) > \bla_{t(e)}(s)} (w_{t(e),s} - w_{s(e),r} + \theta_e - \tfrac{1}{2} \hbar)^{\underline{\bla_{s(e)}(r) - \bla_{t(e)}(s)}} }{\prod_{i \in I} \prod_{r,s \in [\bv_i], r < s} (w_{i,r} - w_{i,s} )^{\overline{\bla_i(r) - \bla_i(s)} }}
\end{equation}
where
\begin{equation}
  \label{eq:21}
  f = f_1(w) f_2(w + \hbar \bmu_1) \cdots f_k(w + \hbar (\bmu_1 + \cdots + \bmu_{k-1}) ) 
\end{equation}

\end{Corollary}

In the above Corollary, we make use of ($\hbar$-versions) of Knuth's rising and falling notations for Pochhammer symbols (see \S Appendix \ref{appendix-sec:rising-and-falling-notation-for-pochhammer-symbols} for precise details).

\subsection{Restricted dressing and $\qzastplus(\bv)$}
\label{ssec: restricted dressing and positive part}

Our goal in this section is to define an important subalgebra $\qzastplus(\bv) \subset \qzast(\bv)$.  Recall from \eqref{eq: GT for bla}  that for any $\bv \in \NN^\vertexset$ and multipartition $\bla  = (\la_i)\in \Partitions_{\leq \bv}^\vertexset$, we define:
$$
\GT^\bla(\bv) = \tGT(\bv)^{S_\bla}
$$
where $S_\bla \subset S_\bv$ is the stabilizer of $\bla$, viewed as a coweight of $\bG = \prod_{i\in \vertexset} \GL(V_i)$.  Note that $S_\bla$ thus implicitly depends on $\bv$.  Explicitly, write $\bla =  (1^{\bm_1} 2^{\bm_2} \cdots p^{\bm_p})$ in the exponential notation \eqref{eq: product exponential notation}. As in the proof of Theorem \ref{thm: grading}, we define $\bm_0 = (m_{i,0}) \in \NN^\vertexset$ by $m_{i,0} = \sv_i - |\la_i|$. Then we have:
\begin{equation}
    \label{eq:Sbladefined}
    S_\bla = S_{\bm_p} \times S_{\bm_{p-1}} \times \cdots \times S_{\bm_2} \times S_{\bm_1} \times S_{\bm_0}
\end{equation}
Now, consider the subgroup defined by omitting the final factor:
\begin{equation}
    S_{\bla, \res} = S_{\bm_p} \times S_{\bm_{p-1}} \times \cdots \times S_{\bm_2} \times S_{\bm_1}
\end{equation}
Then  $S_{\bla, \res}$ preserves the subalgebra $\tGT(|\bla|) \subseteq \tGT(\bv)$ generated by the variables $w_{i,r}$ for all $i \in \vertexset$ and  $1 \leq r \leq |\la_i|$, and we will denote the corresponding fixed subalgebra by 
\begin{equation}
    \GT_{\res}^\bla(\bv) = \tGT(|\bla|)^{S_{\bla,\res}} \subseteq \GT^\bla(\bv)
\end{equation}
In particular, we can identify  $\GT_{\res}^\bla(\bv) = \GT^\bla(|\bla|)$. If $\bv \geq \bv'$ and $\bla \in \Partitions_{\leq \bv'}^\vertexset$, then the surjection $\GT^\bla(\bv) \twoheadrightarrow \GT^\bla(\bv')$ from \eqref{eq: directed system of symmetric polynomials} restricts to an \emph{isomorphism} 
\begin{equation}
    \label{eq: directed system of restricted symmetric polynomials}
    \GT^\bla(\bv)_\res \stackrel{\sim}{\longrightarrow} \GT^\bla(\bv')_\res
\end{equation}
\begin{Remark}
If $\bmu = (1^\bm)$ is a one-column multipartition, so $\mathbf{0} \leq \bm \leq \bv$, then 
$$
    \GT_{\res}^\bmu(\bv) = \tGT(\bm)^{S_\bm} \subseteq \tGT(\bv)^{S_\bm \times S_{\bv-\bm}} = \GT^\bmu(\bv)
$$
In particular for $\bm = \mathbf{0}$ we simply recover the base ring $\GT_{\res}^\bmu(\bv) = \base$ from \eqref{eq: base ring}.
\end{Remark}

\begin{Definition}
    The \emph{positive subalgebra} $\qzastplus(\bv) \subset \qzast(\bv)$ is defined to be the subalgebra generated by the FMOs $\M_{\bm}(f)$ for all $\mathbf 0 \leq \bm \leq \bv$ and $f \in \GT_{\res}^\bmu(\bv)$, where $\bmu = (1^\bm)$.

    We will sometimes write $\qzastplus_\kk(\bv)$ to make the choice of base ring more explicit.

\end{Definition}

Consider the filtration  on $\qzastplus(\bv)$ by the poset $(\Partitions_{\leq \bv}^\vertexset, \leq)$, induced from the filtration on $\qzast(\bv)$:
\begin{equation}
    \qzastplus(\bv)^{\leq \bla} = \qzastplus(\bv) \cap \qzast(\bv)^{\leq \bla}    
\end{equation}
Since this is the induced filtration from $\qzast(\bv)$, there is a natural injection $\gr \qzastplus(\bv) \subseteq \gr \qzast(\bv)$.
\begin{Theorem}
    \label{thm:grpositivepart}
    Under the identification $\gr \qzast(\bv) = \bigoplus_{\bla \in \Partitions_{\leq \bv}^\vertexset} \GT^\bla(\bv) \cdot \sfr_\bla$ from \eqref{eq: gr ring}, we have 
    $$
    \gr \qzastplus(\bv) = \bigoplus_{\bla \in \Partitions_{\leq \bv}^\vertexset} \GT_{\res}^\bla(\bv) \cdot \sfr_\bla
    $$
    In particular, $\qzastplus(\bv)$ is a free module over $\base$ as well as  over $\kk$, and has good behaviour under base change: we have $\qzastplus_\kk(\bv) = \qzastplus_\ZZ(\bv) \otimes_\ZZ \kk$ and a natural inclusion $\qzastplus_\ZZ (\bv) \subset \qzastplus_\QQ(\bv)$.
\end{Theorem}
\begin{proof}
    First note that the right hand side in the statement is indeed a subring, because the structure constants $a_{\bla,\bmu}$ from \eqref{eq: rels in gr} satisfy $a_{\bla,\bmu} \in \GT_{\res}^\bla(\bv)$  by part (2) of Remark \ref{rmk:aproperties}.
    
    Using the  argument from the proof of Proposition \ref{generation by FMOs}, we will show that there is an inclusion $\gr \qzastplus(\bv) \supseteq \bigoplus_{\bla \in \Partitions_{\leq \bv}^\vertexset} \GT_{\res}^\bla(\bv) \cdot \sfr_\bla$.  First note that by definition $\gr \qzastplus(\bv)$ contains the elements $f \sfr_\bmu$ for any one-column multipartition $\bmu = (1^\bm)$ and any $f \in \GT_{\res}^\bmu(\bv)$.  Now take  any $\bla \in \Partitions_{\leq \bv}^\vertexset$, and decompose $\sfr_\bla = \sfr_\bmu \sfr_{\boldsymbol{\kappa}}$ as in the proof of Proposition \ref{generation by FMOs}.  It is not hard to see that the restricted version of \eqref{eq: psym polys} is also surjective:
    $$
    \GT_{\res}^\bmu(\bv) \otimes \GT^{\boldsymbol{\kappa}}_{\res}(\bv) \twoheadlongrightarrow \GT_{\res}^\bla(\bv)
    $$
    Inducting on the number of columns of $\bla$, it follows that every element of $\GT^\bla(\bv)_\res \cdot \sfr_\bla$ can be expressed in terms of elements $ f \sfr_\bmu \in \gr \qzastplus(\bv)$.

    We will defer our proof of the opposite inclusion $\gr \qzastplus(\bv) \subseteq \bigoplus_{\bla \in \Partitions_{\leq \bv}^\vertexset} \GT_{\res}^\bla(\bv) \cdot \sfr_\bla$ until \S \ref{ssec:proofgrAplus}, where we will use the machinery of shuffle algebras.
\end{proof}

For each $\bv$, there is a multiplication map  (of subalgebras)
\begin{equation}
    \GT(\bv) \otimes_\base \GT_\res^\bla(\bv) \longrightarrow \GT^\bla(\bv)
\end{equation}
which is surjective (but not an isomorphism). This leads to the following basic result, which will ultimately yield a triangular decomposition in the limit as $\bv \rightarrow \infty$, see Section \ref{sec:quantized_limit_zastava}.

\begin{Lemma}
    \label{lem: truncated triangular decomp}
    The following multiplication maps are both surjective (but \emph{not} isomorphisms): 
    $$
    \GT(\bv) \otimes_\base \qzastplus(\bv) \twoheadlongrightarrow \qzast(\bv), \qquad \qquad \qzastplus(\bv)\otimes_\base \GT(\bv) \twoheadlongrightarrow \qzast(\bv)
    $$
\end{Lemma}

\subsection{Grading by $\ZZ^\vertexset\times \ZZ$}
\label{ssec: gradings}
In this section we extend the $\ZZ^\vertexset$-grading from \S \ref{ssec: gradings 0}, by incorporating an additional $\ZZ$-grading following a similar strategy to \cite[Remark 2.8]{BFN1}.

The rings $\base = H_{\CC^\times \times \bF}^\bullet(pt) = \kk[\tfrac{1}{2}\hbar, \theta_e]$ and $\tGT(\bv) = H_{\bT\times \CC^\times\times \bF}^\bullet(\pt) = \kk[\tfrac{1}{2}\hbar, \theta_e, w_{i,r}]$ are $\ZZ$--graded by {half} of the cohomological degree, or in other words by giving their generators $\tfrac{1}{2} \hbar, \nu_a, w_{i,r}$ degree 1.   This restricts to the half cohomological grading on the subrings $\GT^{\bla}(\bv) = H_{Z_{\bG}(\bla)\times \CC^\times \times \bF}^\bullet(pt)$ from \eqref{eq: GT for bla}.  We will write $\deg f$ for the degree of $f\in \GT^{\bla}(\bv)$.  

Recall our notation $|\bla| = (|\la_i|)_{i \in \vertexset} \in \NN^\vertexset$ for the vector of sizes of $\bla = (\la_i)_{i \in\vertexset} \in \Partitions^\vertexset$, and also the pairing $(\bla, \bmu)$ on multipartitions $\bla, \bmu \in \Partitions^{\vertexset}$ from (\ref{eq: pairing on tuples of partitions}). 
\begin{Theorem}
\label{thm: grading}
There is a $\ZZ^{\vertexset} \times \ZZ$--grading on $\qzast(\bv)$, such that the subalgebra $\qzastplus(\bv)$ is graded.  For $\bla \in \Partitions_{\leq \bv}^\vertexset$ and homogeneous $f \in \GT^\bla(\bv)$, we have:
$$
\deg \M_{\bla}(f) = \big( |\bla|, \ \deg f +  \tfrac{1}{2}(\bla,\bla) \big)
$$
In particular, the degrees of FMOs are given by
$$
\deg \M_\bm(f) = \big(\bm, \ \deg f + \tfrac{1}{2}(\bm,\bm) \big)
$$
\end{Theorem}
Note that since $\qzast(\bv)$ is generated by FMOs, the grading is uniquely determined by the previous equation.

\begin{Remark}
\label{rem: degrees GT}
The ring $\GT(\bv)$ inherits a $\ZZ^\vertexset\times \ZZ$--grading via the above theorem and the embedding $\GT(\bv)\subset \qzast(\bv)$. Explicitly, for $f \in \GT(\bv)$ we  may identify  $f = \M_{\mathbf{0}}(f) \in \qzast(\bv)$, and it has degree:
$$
(\mathbf{0}, \deg f) \in \ZZ^{\vertexset} \times \ZZ
$$
In other words, the $\ZZ$--grading on $\GT(\bv)$ is trivially extended to a $\ZZ^\vertexset\times \ZZ$--grading.
\end{Remark}
\begin{proof}
The $\ZZ^\vertexset$-part of the grading has been defined already in \S \ref{ssec: gradings 0}, so it remains to obtain the $\ZZ$-grading. 

We first consider the $\Delta$--grading on $\qzast(\bv)$ \cite[Remark 2.8]{BFN1}.  This is a $\ZZ$--grading such that
$$
\deg_\Delta \M_{\bla}(f) = 2 \deg f + 2\Delta(\bla),
$$
where
\begin{equation}
\label{eq: Delta grading1}
2\Delta(\bla) = - 2\sum_{\alpha \in \Delta^+_\bG} | \langle \alpha, \bla\rangle| +  \sum_{\chi} |\langle\chi, \bla\rangle| \dim \bN(\chi)
\end{equation}
Here $\Delta^+_\bG$ is the set of positive roots of $\bG$, $\chi$ runs over the weights of the representation $\bN$, and we write $\langle \alpha,\bla\rangle$  (resp.~$\langle\chi, \bla\rangle$) for the pairing of the root $\alpha$ (resp.~weight $\chi$) with the coweight $\bla$.  

Let us compute $\Delta(\bla)$ explicitly, cf.~\cite[\S 5]{Nak7}.  To this end it will be useful to write $\bla = (1^{\bm_1} 2^{\bm_2}\cdots)$ in exponential notation, following (\ref{eq: product exponential notation}). Note that when the partition $\lambda_i$ is viewed as a  coweight of $\GL(V_i)$, the integer $0$ appears $\sm_{i,0}= \sv_i - |\la_i|$ many times.    With this in mind, a direct calculation using (\ref{eq: Delta grading1}) gives:
\begin{align}
2 \Delta(\bla) & =  - \sum_{i \in \vertexset} \sum_{1 \leq r < s \leq \sv_i} 2 (\lambda_{i,r} - \lambda_{i,s}) + \sum_{a \in \arrowset} \sum_{\substack{1 \leq r \leq \sv_{\source(a)}, \\ 1 \leq s \leq \sv_{\target(a)}}} | \lambda_{\target(a), s} - \lambda_{\source(a),r}| \nonumber \\
& = -\sum_{i \in \vertexset} \sum_{0 \leq k < \ell} (\ell - k) 2 \sm_{i,k} \sm_{i,\ell} + \sum_{a \in \arrowset} \sum_{0 \leq k < \ell} (\ell-k) \big( \sm_{\source(a),k} \sm_{\target(a),\ell} + \sm_{\source(a),\ell} \sm_{\target(a), k} \big) \nonumber \\
& = - \sum_{i \in \vertexset}  \langle\langle \lambda_i, \lambda_i \rangle\rangle + \sum_{a \in \arrowset} \langle\langle \lambda_{\source(a)}, \lambda_{\target(a)}\rangle\rangle \label{eq: Delta grading2}
\end{align}
where we have adopted the following notation
$$
\langle\langle \lambda_i, \lambda_j \rangle \rangle = \sum_{0 \leq k < \ell} (\ell-k) (\sm_{i,k} \sm_{j,\ell} + \sm_{i,\ell} \sm_{j,k} \big)
$$
Note that $k= 0$ is included in the summation, and hence this notation implicitly depends on $\bv$ since the integers $\sm_{i,0}$ do. We have the following lemma, which follows by applying equation (\ref{eq: Hua pairing}):
\begin{Lemma}
For any $i, j \in \vertexset$, we have:
$$
2 \langle \lambda_i, \lambda_j \rangle +  \langle \langle \lambda_i, \lambda_j \rangle \rangle = \sv_j |\lambda_i| + \sv_i |\lambda_j| 
$$
\end{Lemma}
We can now calculate the difference between \emph{two times} our desired $\ZZ$-grading and the $\Delta$--grading, using (\ref{eq: Delta grading2}) and the previous lemma:
\begin{align*}
(\bla, \bla) - 2 \Delta(\bla) &= 2\Big( \sum_{i \in \vertexset}\langle \lambda_i,\lambda_i\rangle - \sum_{a\in \arrowset} \langle \lambda_{\source(a)}, \lambda_{\target(a)}\rangle\Big) - \Big( -\sum_{i \in \vertexset}\langle \langle\lambda_i,\lambda_i\rangle\rangle - \sum_{a\in \arrowset} \langle\langle \lambda_{\source(a)}, \lambda_{\target(a)}\rangle\rangle\Big) \\
& = \sum_{i \in \vertexset}\big(2\langle \lambda_i,\lambda_i\rangle + \langle\langle \lambda_i, \lambda_i \rangle\rangle \big) - \sum_{a \in \arrowset} \big( 2\langle \lambda_{\source(a)}, \lambda_{\target(a)}\rangle +\langle \langle \lambda_{\source(a)}, \lambda_{\target(a)}\rangle\rangle \big) \\
& = \sum_{i \in \vertexset} 2 \sv_i |\lambda_i| - \sum_{a \in \arrowset}\big( |\lambda_{\source(a)}| \sv_{\target(a)} + |\lambda_{\target(a)}| \sv_{\source(a)} \big) \\
& = (|\bla|, \bv)
\end{align*}
We see that this difference depends linearly on the connected component $|\bla| \in \NN^{\vertexset} \cong \pi_0(\cR^+)$.  By the same reasoning as \cite[Remark 2.8(2)]{BFN1}, it follows that there is a $\ZZ$--grading $\deg'$ on $\qzast(\bv)$ defined by $\deg' \M_{\bla}(f) = 2 \deg f + (\bla,\bla)$. Since this grading is even we may safely divide it by two, thus defining our desired $\ZZ$-grading and completing the proof.
\end{proof}

\begin{Remark}
    \label{rmk:gradingunderGKLO}
    Via the GKLO embedding $\qzast(\bv) \hookrightarrow \GKLO{\bv}$, the  above grading on $\qzast(\bv)$ is the restriction of one defined on $\GKLO{\bv}$, with degrees assigned to its generators as follows:
    $$
    \deg \tfrac{1}{2}\hbar = \deg \theta_e = \deg w_{i,r} = (\mathbf{0}, 1), \qquad \deg \sfu_{i,r} = \big( \alpha_i, \sv_i - \sum_{e : i \rightarrow j} \sv_j -\delta_{i,j}) \big)
    $$
\end{Remark}

The $\ZZ^\vertexset$--part of the grading on $\qzast(\bv)$ from the theorem is closely related to the filtration  by the poset $\Partitions_{\leq \bv}^\vertexset$, as explained in  Proposition \ref{prop: filtration defined}. By construction, the $\ZZ^\vertexset\times \ZZ$-grading is also compatible with this filtration. We immediately  obtain the following result.
\begin{Corollary}
There is an induced $\ZZ^\vertexset \times \ZZ$--grading on $\gr \qzast(\bv)$ given by $$\deg (f \sfr_{\bla}) = \big( |\bla|, \ \deg f +  \tfrac{1}{2}(\bla,\bla) \big)$$ for all $\bla \in \Partitions_{\leq \bv}^\vertexset$ and homogeneous $f \in \GT^\bla(\bv)$.
\end{Corollary}

\begin{Remark}
The $\ZZ^\vertexset \times \ZZ$--grading from Theorem \ref{thm: grading} corresponds to an action of $\GG_m^{\vertexset}\times \GG_m$ on the ring $\qzast(\bv)$.  Given any tuple $\sigma = (\sigma_i)_{i \in \vertexset} \in \ZZ^{\vertexset}$, we may define a new action of $\GG_m$ on $\qzast(\bv)$ via the embedding $\GG_m \rightarrow \GG_m^{\vertexset}\times \GG_m$ defined by $ s \mapsto (s^{\sigma}, s)$. The proof of the previous theorem shows that the $\Delta$--degree arises in this way, up to doubling degrees, with
$$
\sigma_i = 2 \sv_i - \sum_{a \in \arrowset, s(a) = i} \sv_{\target(a)} - \sum_{a \in \arrowset, \target(a) = i} \sv_{\source(a)}
$$ 
The homological grading on $\qzast(\bv)$ also arises in a similar way, see \cite[Remark 2.8(2)]{BFN1}.
\end{Remark}

%%%%%%%%%%%
\subsection{Adding defect and closed embeddings}
\label{ssec:adce}

Consider $\bv, \bv' \in \NN^\vertexset$ such that $\bv \geq \bv'$.  It will be convenient to denote $\bv'' = \bv - \bv'$.  Note that $\bv'' \in \NN^\vertexset$. Recall the ring $\base = \kk[\tfrac{1}{2}\hbar, \theta_e]$ from  (\ref{eq: GT with no Weyl}). Then there is an isomorphism of $\base$--algebras:
\begin{equation}
\label{eq:polyringstensor}
\tGT(\bv) \stackrel{\sim}{\longrightarrow} \tGT(\bv') \otimes_{\base} \tGT(\bv'')
\end{equation}
by identifying $\tGT(\bv')$ (resp.~$\tGT(\bv'')$) with the subalgebra of $\tGT(\bv)$ in variables $w_{i,r}$ for $1\leq r \leq \sv'_i$ (resp.~for $\sv'_i < r \leq \sv_i$.)  We will denote the effect of this map using Sweedler notation:
\begin{equation}
f \longmapsto \sum f^{(1)} \otimes f^{(2)}
\end{equation}
We will also consider the surjective homomorphism 
\begin{equation}
    \label{eq: directed system of symmetric polynomials}  
    \tGT(\bv) \twoheadrightarrow \tGT(\bv')
\end{equation}
defined by specializing $w_{i,r} = 0$ for $i\in I$ and $\sv'_i < r \leq \sv_i$. For $f \in \tGT(\bv)$, we will denote its image under this map using the notation  $f \longmapsto \widetilde{f}$.

By restricting the above homomorphisms, we obtain maps:
\begin{equation}
\label{eq: restriction to GT}
\GT(\bv) \ \hookrightarrow \ \GT(\bv') \otimes_{\base} \GT(\bv''), \qquad \qquad \GT(\bv) \ \twoheadrightarrow \ \GT(\bv')
\end{equation}
Similarly, given $\bla \in \Partitions_{\leq \bv'}^\vertexset$ we obtain  homomorphisms
\begin{equation}
\label{eq: restriction to GT 2}
\GT^{\bla}(\bv) \ \hookrightarrow \ \GT^{\bla}(\bv') \otimes_{\base} \GT(\bv''), \qquad \qquad \GT^{\bla}(\bv) \ \twoheadrightarrow \ \GT^{\bla}(\bv')
\end{equation}
We continue to denote the effects of these homomorphisms on elements by $f \mapsto \sum f^{(1)} \otimes f^{(2)}$ and $f\mapsto \widetilde{f}$, respectively.

\begin{Theorem}
\label{thm: defect and embeddings}
Fix $\bv, \bv' \in \NN^\vertexset$ with $\bv \geq \bv'$, and write $\bv'' = \bv - \bv'$ as above.
\begin{enumerate}[label=(\alph*),topsep=3pt,itemsep = 5pt]
\item There is a ring homomorphism $\qzast(\bv) \rightarrow \qzast(\bv') \otimes_{\base} \GT(\bv'')$,  determined by:
$$
\M_\bm(f) \ \mapsto \ \left\{ \begin{array}{cl} \sum \M_{\bm}(f^{(1)}) \otimes f^{(2)}, & \text{if } \bm \leq \bv', \\ 0, & \text{otherwise} \end{array} \right.
$$
for all $\mathbf{0} \leq \bm \leq \bv$ and $f \in \GT^{\bmu}(\bv)$.  It restricts to $\GT(\bv) \hookrightarrow \GT(\bv') \otimes_{\base} \GT(\bv'')$ from (\ref{eq: restriction to GT}).

\item There is a {surjective} ring homomorphism $\Phi_{\bv, \bv'}: \qzast(\bv) \twoheadrightarrow \qzast(\bv')$,  determined by:
$$
\M_\bm(f) \ \mapsto \ \left\{ \begin{array}{cl} \M_\bm(\widetilde{f}), &\text{if } \bm \leq \bv', \\ 0, & \text{otherwise} \end{array} \right.
$$
for all $\mathbf{0} \leq \bm \leq \bv$ and $f \in \GT^{\bmu}(\bv)$. This map restricts to $\GT(\bv) \twoheadrightarrow \GT(\bv')$ from (\ref{eq: restriction to GT}), and also restricts to a surjection $\Phi_{\bv, \bv'}^+:  \qzastplus(\bv) \twoheadrightarrow \qzastplus(\bv')$.

\end{enumerate}
\end{Theorem}
We will mainly be interested in the surjective homomorphisms from part (b) of the theorem.  However, the \textbf{adding defect homomorphisms} from part (a)  are simpler to construct and play an important role in the proof of part (b).

\begin{Remark}
    Specialized at $\tfrac{1}{2}\hbar = \theta_e = 0$ for all $ e\in \arrowset$, we recover the closed embedding and adding defect maps for Zastava spaces from \cite[\S 2.5 and \S 4.4]{MW2}. Note that in the present article we allow our quiver $Q$ to contain loops; the results and proofs from \cite{MW2} extend easily to this generality, if we define zastava spaces via the Coulomb branch construction.
\end{Remark}
\begin{proof}[\mbox{Proof of Theorem \ref{thm: defect and embeddings}}]
We follow the same strategy as \cite[Theorem 2.57]{MW2}. To prove (a), we construct the desired homomorphism as the left hand side of a commutative diagram:
\begin{equation}
\label{eq: adding defect defining commutative diagram}
\begin{tikzcd}
\qzast(\bv) \ar[d,dashed] \ar[r,hook] & \tilde{\cA}_{\hbar}^+(\bv)_{\loc} \ar[d,"\varphi"] \\
\qzast(\bv') \otimes_{\base} \GT(\bv'') \ar[r,hook] & \big(\tilde{\cA}_{\hbar}^+(\bv')\otimes_{\base} \tGT(\bv'')\big)_{\loc}
\end{tikzcd}
\end{equation}
Here $\big(\tilde{\cA}_{\hbar}^+(\bv')\otimes_{\base} \tGT(\bv'')\big)_{\loc}$ denotes the Ore  localization of $\tilde{\cA}_{\hbar}^+(\bv')\otimes_{\base} \tGT(\bv'')$ at the multiplicative set generated by the union
$$
\begin{array}{c} \left\{ w_{i,r} - w_{i,s} + n \hbar : i \in \vertexset, 1 \leq r \neq s \leq \sv_i, n \in \ZZ \right\} \\
\bigcup\left\{ w_{\target(a),s} - w_{\source(a),r} + \nu_a + (n-\tfrac{1}{2})\hbar : a \in \arrowset, 1 \leq r\leq \sv'_{\source(a)}, \sv'_{\target(a)} < s \leq \sv_{\target(a)},  n \in \ZZ \right\}
\end{array}
$$
The horizontal arrows in the above commutative diagram are GKLO embeddings from \S \ref{ssec: GKLO embedding}.  Meanwhile, the arrow $\varphi$ on the right hand side of the diagram is given by:
\begin{Lemma}
\label{lem:rationalembedding}
There is homomorphism of $\base$--algebras
$$
\varphi: \tilde{\cA}_{\hbar}^+(\bv)_{\loc} \longrightarrow \big(\tilde{\cA}_{\hbar}^+(\bv')\otimes_{\base} \tGT(\bv'')\big)_{\loc}
$$
defined by sending $w_{i,r} \mapsto w_{i,r}$, and by sending
$$
\sfu_{i,r} \mapsto \left\{ \begin{array}{cl} \frac{\displaystyle{\prod_{s=\sv_i'+1}^{\sv_i}(w_{i,r}-w_{i,s})}}{\displaystyle{\prod_{a \in \arrowset: \source(a) = i} \prod_{s=\sv'_{\target(a)}+1}^{\sv_{\target(a)}} (w_{\target(a), s} - w_{i, r} + \nu_a - \tfrac{1}{2}\hbar)}} \sfu_{i,r} & \ \ \text{if } 1 \leq r \leq \sv'_i, \\ 0 & \ \ \text{if } \sv'_i < r \leq \sv_i \end{array} \right.
$$
\end{Lemma}
The proof of the lemma appears below.  To complete the proof of part (a) we must show the existence of the left downward arrow in the commutative diagram (\ref{eq: adding defect defining commutative diagram}).  The restriction of $\varphi$ to $\qzast(\bv)$ is certainly a ring homomorphism.  Since the FMOs $\M_\bm(f)$ generate $\qzast(\bv)$, to prove that 
$$
\varphi( \qzast(\bv)) \subseteq \qzast(\bv') \otimes_{\base} \GT(\bv'')
$$
it suffices to show that $\varphi(\M_\bm(f))$ has the claimed form from the statement of the theorem. This follows by direct calculation, similarly to the proof of \cite[Theorem 2.57]{MW2}.

Part (b) follows from part (a), by defining $\Phi_{\bv, \bv'}$ as the composition
$$
\qzast(\bv) \stackrel{\text{(a)}}{\longrightarrow} \qzast(\bv') \otimes_{\base} \GT(\bv'') \ \twoheadrightarrow \qzast(\bv') \otimes_\base \base \cong \qzast(\bv')
$$
where the middle map is defined via the surjection $\GT(\bv'') \twoheadrightarrow \GT(\mathbf 0) = \base$ from \eqref{eq: directed system of symmetric polynomials}.  This composition has the claimed effect on the generators $\M_\bm(f)$ of $\qzast(\bv)$, and is surjective since all generators $\M_\bm(\widetilde{f})$ of $\qzast(\bv')$ are in the image. By restricting our  attention to $f \in \GT^\bmu_\res(\bv)$, by the same reasoning we see that $\Phi_{\bv,\bv'}$  restricts to a surjection $\Phi_{\bv,\bv'}^+ : \qzastplus(\bv) \twoheadrightarrow \qzastplus(\bv')$.
\end{proof}

\begin{proof}[Proof of Lemma \ref{lem:rationalembedding}]
The proposed definition of $\varphi(\sfu_{i,r})$ involves a rational expression in $w_{i,r}$ and variables $w_{j,s}$ such that $\sv'_j < s \leq \sv_j$.  These variables all commute with $\sfu_{k,t}$ for any $k \in \vertexset$ and $1 \leq t \leq \sv'_k$ such that $(k,t) \neq (i,r)$. Using this the defining relations of $\tilde{\cA}_\hbar^+(\bv)$ from \S \ref{ssec: GKLO embedding} are readily verified. 
\end{proof}

A key property of the  surjective homomorphisms from part (b) of the theorem is that they form directed systems:

\begin{Proposition}
\label{prop: surjections respect gradings}
The map $\Phi_{\bv, \bv'}: \qzast(\bv) \twoheadrightarrow \qzast(\bv')$ respects $\ZZ^\vertexset \times \ZZ$--gradings.
We have $\Phi_{\bv, \bv} = \operatorname{Id}$ and  $\Phi_{\bv, \bv''} = \Phi_{\bv',\bv''} \circ \Phi_{\bv,\bv'}$ whenever $\bv \geq \bv' \geq \bv''$.  

By restriction, analogous claims hold for the maps $\Phi_{\bv,\bv'}^+ : \qzastplus(\bv) \twoheadrightarrow \qzastplus(\bv')$.
\end{Proposition}
\begin{proof}
The homomorphisms $f \mapsto \widetilde{f}$ from (\ref{eq: restriction to GT}) are $\ZZ$--graded, and $\Phi_{\bv,\bv'}(\M_\bm(f))$ is equal to $\M_\bm(\tilde{f})$ or zero. Inspecting the formulas for the $\ZZ^\vertexset\times \ZZ$--grading from Theorem \ref{thm: grading} it is immediate that
$$
\deg \M_\bm(f) = \deg \Phi_{\bv,\bv'} \big( \M_\bm(f) \big)
$$
Since the $\M_\bm(f)$ generate $\qzast(\bv)$, this proves that $\Phi_{\bv, \bv'}$ is graded.  The fact that these maps define a directed system is also immediate from their effect $\M_\bm(f) \mapsto \M_\bm(\widetilde{f})$ on FMOs, since the maps $f\mapsto \widetilde{f}$ form a directed system.
\end{proof}

%%%%%%%%%
\subsection{Defining ideals and stabilization}
Recall from Remark \ref{rmk: filtration by all partitions} that we may interpret each $\qzast(\bv)$ as a filtered graded ring, filtered by the poset $(\Partitions^\vertexset, \leq)$ of multipartitions.

\begin{Lemma}
\label{lem:surjectionsarefilt}
Let $\bv \geq\bv'$. Then the surjection $\Phi_{\bv, \bv'}: \qzast(\bv)  \twoheadrightarrow \qzast(\bv')$ is {strictly} filtered:
$$
\Phi_{\bv, \bv'}\big( \qzast(\bv)^{\leq \bla} \big) \ = \ \qzast(\bv')^{\leq \bla} \quad \text{for any } \bla \in \Partitions^\vertexset.
$$
The associated graded map $\gr \Phi_{\bv,\bv'} : \gr \qzast(\bv) \twoheadrightarrow \gr \qzast(\bv')$ is given explicitly by 
$$
f \sfr_{\bla} \ \longmapsto \ \left\{ \begin{array}{cl} \tilde{f} \sfr_{\bla}, & \text{if } \bla \in \Partitions_{\leq \bv'}^\vertexset, \\ 0, & \text{otherwise,}\end{array} \right.
$$
for any $\bla \in \Partitions_{\leq \bv}^\vertexset$ and $f\in \GT^{\bla}(\bv)$, where as per usual $f\mapsto \tilde{f}$ denotes image under $\GT^{\bla}(\bv) \twoheadrightarrow \GT^{\bla}(\bv')$.
\end{Lemma}
\begin{proof}
Take $\bla \in \Partitions^\vertexset$. If $\bla \notin \Partitions_{\leq \bv}^\vertexset$ then $\qzast(\bv)^{\leq \bla} = \qzast(\bv')^{\leq \bla} =\{ 0\}$, so there is nothing to check.  So suppose that $\bla \in \Partitions_{\leq \bv}^\vertexset$.   By Corollary \ref{cor: filtration generated by FMOs}, $\qzast(\bv)^{\leq \bla}$ is spanned by elements of the form
$$
x \ = \ \M_{\bm_1}(f_1) \M_{\bm_2}(f_2) \cdots \M_{\bm_k}(f_k)
$$
where the corresponding multipartitions $\bmu_j = (1^{\bm_j})$ satisfy $\bmu_1+\ldots + \bmu_k \leq \bla$.  Under $\Phi_{\bv,\bv'}$ each FMO maps to an FMO or to zero.  More precisely, we have two main cases:

First, suppose that  $\bla \in \Partitions_{\leq \bv'}^\vertexset$. If some $\bm_j \not\leq \bv'$, then $\Phi_{\bv,\bv'}( \M_{\bmu_j}(f_j)) = 0$ so $\Phi_{\bv, \bv'}(x) = 0$.  Else if all $\bm_j \leq \bv'$, then 
$$
\Phi_{\bv,\bv'}(x) = \M_{\bm_1}(\widetilde{f_1}) \M_{\bm_2}(\widetilde{f_2}) \cdots \M_{\bm_k}(\widetilde{f_k})
$$
Elements of this form span $\qzast(\bv')^{\leq \bla}$, so we have $\Phi_{\bv, \bv'}\big( \qzast(\bv)^{\leq \bla} \big) =  \qzast(\bv')^{\leq \bla} $.

Next, suppose that $\bla \notin \Partitions_{\leq \bv'}^\vertexset$.  Then by (\ref{eq: partitions length decreases under dominance}), we have $\bm_1 \not\leq \bv'$ and therefore  $\Phi_{\bv,\bv'}(x) = 0$.  So in this case we see that $\Phi_{\bv,\bv'}\big( \qzast(\bv)^{\leq \bla}\big) = 0$, as desired.

Finally, each element $f \sfr_{\bla}$ has a lift $\M_\bla(f) \in \qzast(\bv)$ which is a sum of elements of the above form with $\bmu_1 + \ldots + \bmu_k = \bla$. The formula for $\gr \Phi_{\bv,\bv'}$ follows.
\end{proof}

\begin{Theorem}
Let $\bv' \leq \bv$. The  kernel of the ring homomorphism $\Phi_{\bv,\bv'}: \qzast(\bv) \rightarrow \qzast(\bv)$ is generated by the following FMOs:
\begin{enumerate}[label=(\roman*),topsep=3pt,itemsep = 5pt]
    \item $\M_\bm(f)$ for $ \mathbf{0} \leq \bm \leq \bv'$, such that $f$ is in the kernel of the homomorphism $\GT^{\bmu}(\bv) \twoheadrightarrow \GT^{\bmu}(\bv')$ from (\ref{eq: restriction to GT}).
    \item $\M_\bm(f)$ for all $\mathbf{0} \leq \bm \leq \bv$ such that $\bm \not\leq \bv'$, and all elements $f\in \GT^{\bmu}(\bv)$.
\end{enumerate}
In other words, the kernel of $\Phi_{\bv,\bv'}$ is generated by those FMOs which are in the kernel. 
\end{Theorem}

\begin{proof}
Denote $K = \operatorname{Ker} \Phi_{\bv,\bv'}$, and consider the induced filtration
$K^{\leq \bla} = K \cap \qzast(\bv)^{\leq \bla}$. By the previous lemma the map $\Phi_{\bv,\bv'}$ is strictly filtered, and we see that
$$
\gr K = \operatorname{Ker}( \gr \Phi_{\bv, \bv'} ) = \bigoplus_{\bla \in \Partitions_{\leq \bv}^\vertexset} K^\bla \sfr_\bla,
$$
where $K^\bla \subset \GT^{\bla}(\bv)$ is the subset (in fact, ideal) defined by
\begin{equation}
K^\bla \ = \ \left\{ \begin{array}{cl} \operatorname{Ker}\big( \GT^{\bla}(\bv) \twoheadrightarrow \GT^{\bla}(\bv')\big), & \text{ for } \bla \in \Partitions_{\leq \bv'}^\vertexset \\  \GT^{\bla}(\bv), & \text{ otherwise } \end{array} \right.
\end{equation}
We claim that $\gr K$ is generated by the elements $f \sfr_{\bmu}$, where $\bmu = (1^\bm)$ and either:
\begin{enumerate}[label=(\roman*)]
\item $\bm \leq \bv'$ and $f \in \GT^{\bmu}(\bv)$ is in the kernel of $ \GT^{\bmu}(\bv) \twoheadrightarrow \GT^{\bmu}(\bv')$, or
\item  $\bm \not\leq \bv'$ and $f\in \GT^{\bmu}(\bv)$ is any element.
\end{enumerate}
Assuming this claim for the moment, the theorem follows: the FMOs in the statement of the theorem are simply lifts to $K$ of the above generators of $\gr K$, which therefore generate $K$ itself.

To prove the claim: first notice that the elements of the form (i) or (ii) are indeed inside $\gr K$.  Next, recall from the proof of Theorem \ref{generation by FMOs} that for $\bla \in \Partitions_{\leq \bv}^\vertexset$, with column lengths encoded by $\bmu_1,\ldots, \bmu_k$ where $\bmu_j = (1^{\bm_j})$, we have
\begin{equation}
\label{eq: iterated product shift}
f_1\sfr_{\bmu_1}  f_2 \sfr_{\bmu_2} \cdots f_k\sfr_{\bmu_k}  \ = \ f_1(w)  f_2(w + \bmu_1 \hbar) \cdots f_k(w + (\bmu_1+\ldots + \bmu_{k-1})\hbar) \sfr_\bla
\end{equation}
for any $f_j \in \GT^{\bmu_j}(\bv)$.  It suffices to consider such products where at least one of the $f_j$ is of the form (i) or (ii) above, and show that they span the ideal $K^\bla$.

Assume first that $\bla \notin \Partitions_{\leq \bv'}^\vertexset$.  Then its tuple of first column lengths $\bm_1 \not\leq \bv'$, and so we are in case (ii): we have $f_1 \sfr_{\bmu_1} \in \gr K$ for \emph{any} element $f_1 \in \GT^{\bmu_1}(\bv)$. As in the proof of Theorem \ref{generation by FMOs}, the elements (\ref{eq: iterated product shift}) span all of $\GT^{\bla} \sfr_\bla = K^\bla$.

Next assume that $\bla \in \Partitions_{\leq \bv'}^\vertexset$.  We consider elements (\ref{eq: iterated product shift}) where some $f_j$ is of type (i), i.e.~$f_j$ is in the kernel of $\GT^{\bmu_j}(\bv) \twoheadrightarrow \GT^{\bmu_j}(\bv')$.  It is an exercise in partially symmetric polynomials to show that these elements span $K^\bla = \operatorname{Ker}\big( \GT^{\bla}(\bv) \twoheadrightarrow \GT^{\bla}(\bv')\big)$, which we leave to the reader. 
\end{proof}

\begin{Remark}
    Similar claims hold for the restricted maps $\Phi_{\bv,\bv'}^+ : \qzastplus(\bv) \twoheadrightarrow \qzastplus(\bv')$. In this case, the kernel is generated entirely by elements of type (ii) from the previous theorem: the elements $\M_\bm(f)$ for all $\mathbf 0 \leq \bm \leq \bv$ such that $\bm \not\leq \bv'$, and for all $f \in \GT_{\res}^\bmu(\bv)$.  Indeed, no elements of type (i) are needed since $\GT^\bmu(\bv)_\res \stackrel{\sim}{\rightarrow} \GT^\bmu(\bv')_\res$ is an isomorphism if $\bm \leq \bv'$.
\end{Remark}

We end this section by showing that the maps $\Phi_{\bv,\bv'}$ stabilize with respect to the $\ZZ^\vertexset\times \ZZ$-gradings, in the following sense:

\begin{Lemma}
    \label{lem:stabilization}
    Fix $(\bk, d) \in \ZZ^\vertexset\times \ZZ$. Let  $\bv, \bv' \in \NN^\vertexset$ with $\bv \geq \bv'$, and suppose that the following conditions are satisfied:
    \begin{enumerate}
        \item $\bv' \geq \bk$
        \item For each $i\in \vertexset$, we have $\bv'_i \geq d - \displaystyle{\min_{\substack{\bla \vdash \bk}} \left\{ \tfrac{1}{2}(\bla, \bla) - \ell(\la_i) \right\}}$
    \end{enumerate}
    Then the restriction of $\Phi_{\bv,\bv'} : \qzast(\bv) \twoheadrightarrow \qzast(\bv')$  defines an isomorphism $\qzast(\bv)_{(\bk, d)} \xrightarrow{\sim } \qzast(\bv')_{(\bk,d)}
    $.
\end{Lemma}
\begin{proof}

    The filtration $\qzast(\bv)_\bk = \bigcup_{\bla \in \Partitions^\vertexset, ~ \bla \vdash \bk} \qzast(\bv)^{\leq \bla}$ from Proposition \ref{prop: filtration defined} splits, e.g.~using the elements $\M_\bla(f)$ from \eqref{eq: general MOs}. In other words, there is an isomorphism of $\ZZ$-graded graded $\base$-modules
    \begin{equation}
        \label{eq: split filtration}
        \qzast(\bv)_\bk = \bigoplus_{\bla \in \Partitions^\vertexset, ~ \bla \vdash \bk} \GT^\bla(\bv)\{\tfrac{1}{2} (\bla,\bla)\},
    \end{equation}
    where $\{\cdot\}$ denotes a grading shift. 
    Moreover, the results of this subsection show that these decompositions can be chosen compatibly with the surjection $\Phi_{\bv,\bv'}$.

    Now, our assumption $\bv' \geq \bk$ guarantees that any $\bla \in \Partitions^\vertexset$ with $\bla \vdash \bk$ automatically satisfies $\bla \in \Partitions^\vertexset_{\leq \bv'}$.  In other words, the number of summands in \eqref{eq: split filtration} stabilizes.   All that is left is to consider the individual surjections $\GT^\bla(\bv) \twoheadrightarrow \GT^\bla(\bv')$, which are  isomorphisms in sufficiently small degrees. Accounting for grading shifts, one deduces stabilization under the assumption (2).
\end{proof}

\begin{Remark}
    \label{rem:stabilization}
    Fix $\bk \in \NN^\vertexset$. Then the same argument as above, but applied to restricted dressing rings $\GT^\bla_\res(\bv) \xrightarrow{\sim} \GT^\bla_\res(\bv')$, shows that $\Phi_{\bv,\bv'} $ restricts to an isomorphism of $\ZZ^\vertexset$-homogeneous components ${\qzastplus(\bv)_\bk \xrightarrow{\sim} \qzastplus(\bv')_\bk}$ for any  $\bv \geq \bv' \geq \bk$.
\end{Remark}

%%%%%%%%%%%%%%%%%%%%%%%%%%%%%%%%%%%%%%%%%%%%%%%%%%%%%%%%%%%%%%%%%
\section{Quantized limit zastava}
\label{sec:quantized_limit_zastava}

\subsection{Limits of graded rings}
Throughout this section we will work in the category of $\ZZ^\vertexset\times \ZZ$--graded rings. Let  $\big(B(\bv), \varphi_{\bv, \bv'}\big)$ be an inverse system in this category, indexed by the set $\NN^\vertexset$ with its partial order $\geq $ from \S \ref{sssec:Partitions}.  Thus, each $B(\bv) = \bigoplus_{(\bn, d) \in \ZZ^\vertexset\times \ZZ} B(\bv)_{(\bn, d)}$ is graded, and  for any $\bv \geq \bv'$ we have a graded homomorphism $\varphi_{\bv, \bv'} : B(\bv) \rightarrow B(\bv')$ satisfying the usual compatibilities.  In this case, we may consider the {limit}:
\begin{equation}
    B = \lim_{\substack{\longleftarrow}} B(\bv)    
\end{equation}
Recall that, by definition, for any other $\ZZ^\vertexset\times \ZZ$--graded ring $S$ we have
$$
    \operatorname{Hom}( S, B) = \lim_{\substack{\longleftarrow}} \operatorname{Hom}\big(S, B(\bv) \big)
$$
There is a simple description of this ring in terms of its homogeneous components:

\begin{Lemma}
    \label{lem:componentsoflimit}
    There is a natural isomorphism of $\ZZ^\vertexset \times \ZZ$--graded rings:
    $$B \cong \bigoplus_{(\bn, d) \in \ZZ^\vertexset \times \ZZ} \Big( \displaystyle{\lim_{\substack{\longleftarrow}}}~B(\bv)_{(\bn, d)} \Big)$$
\end{Lemma}
Note that the limits on the right side of this statement are of abelian groups, and that their direct sum carries a natural ring structure.

\begin{Remark}
    Let $\widehat{B}$ denote the limit of our inverse system, but taken instead in the  category of (ordinary, ungraded) rings. Under our assumptions $\widehat{B}$ naturally carries an action of the group $\GG_m^\vertexset\times \GG_m$, and we may identify $B \subset \widehat{B}$ with the subset of finite vectors, i.e.~elements which are finite sum of weight vectors.
\end{Remark}

The rings $\GT(\bv)$ from (\ref{eq: GT subalgebra}) form a  directed system as above:
they are $\ZZ^\vertexset\times \ZZ$--graded by Remark \ref{rem: degrees GT}, and the surjective maps $\GT(\bv) \twoheadrightarrow \GT(\bv')$ from (\ref{eq: restriction to GT}) are graded and clearly form a directed system. We may thus consider their limit:
\begin{equation}
    \label{eq: GT limit 1}
    \GT = \lim_{\longleftarrow} \GT(\bv)
\end{equation}
More generally, for any tuple of partitions  $\bla \in \Partitions_{\leq \bv}^\vertexset$ we have the ring $\GT^\bla(\bv)$ from \eqref{eq: GT for bla} which is $\ZZ^\vertexset\times \ZZ$--graded analogously to the above.  Using the maps $\GT^\bla(\bv) \twoheadrightarrow \GT^{\bla}(\bv')$ from \eqref{eq: restriction to GT 2}, we define
\begin{equation}
    \label{eq: GT limit 2}
    \GT^\bla = \lim_{\longleftarrow} \GT^\bla(\bv)
\end{equation}
Finally, we may repeat this construction for the restricted subrings $\GT^\bla_\res(\bv) \subseteq \GT^\bla(\bv)$ from \S \ref{ssec: restricted dressing and positive part}.  Recall from \eqref{eq: directed system of restricted symmetric polynomials} that for $\bv \geq \bv'$ and $\bla \in \Partitions_{\leq \bv'}^\vertexset$ the restricted map $\GT^\bla_\res(\bv) \xrightarrow{\sim} \GT_\res^\bla(\bv')$ is simply an isomorphism. We define
\begin{equation}
    \label{eq: GT limit 3}
    \GT^\bla_\res = \lim_{\longleftarrow} \GT_\res^\bla(\bv)
\end{equation}
Then there is a natural isomorphism $\GT^\bla_\res \cong \GT^\bla_\res(|\bla|)$.

The following result follows from basic properties of rings of partially symmetric polynomials:
\begin{Lemma}
    \label{lem: GT bla free}
    $\GT^\bla$ is a free module over $\GT$, and multiplication defines an isomorphism
    $$
    \GT \otimes_\base \GT^\bla_\res  \stackrel{\sim}{\longrightarrow} \GT^\bla
    $$
\end{Lemma}

\begin{Remark}
    Following \cite[\S I.2]{Macdonald}, recall the construction of the ring of symmetric functions (defined here over the base ring $\base$ from \eqref{eq: base ring}) as a limit in the category of $\ZZ$--graded rings:
    $$
    \lim_{\substack{\longleftarrow}} \base[w_1,\ldots,w_\sv]^{S_\sv}
    $$
    Observe that $\GT$ is nothing but a tensor product of $\vertexset$-many copies of this ring, while $\GT^\bla$ is the tensor product over $i \in \vertexset$ of the rings
    $$
    \lim_{\substack{\longleftarrow}} R[w_{i,1}, \ldots, w_{i, \sv_i}]^{S_{\sm_{i,\sv_i}} \times \cdots \times S_{\sm_{i,1}} \times S_{\sm_{i,0}}}
    $$
        Here we have used exponential notation $\la_i = (1^{\sm_{i,1}}2^{\sm_{i,2}}\cdots)$ and written $\sm_{i,0} = \sv_i - |\la_i|$ as in \eqref{eq:Sbladefined}. In particular, in the above limit  $\sm_{i,0} \rightarrow \infty$ while $\sm_{i,1}, \sm_{i,2},\ldots$ remain fixed.
\end{Remark}

\subsection{The quantized limit zastava}

The rings $\qzast(\bv)$ also fit into the framework of the previous subsection, since the maps $\Phi_{\bv, \bv'}: \qzast(\bv) \twoheadrightarrow \qzast(\bv')$ define a directed system of $\ZZ^\vertexset \times \ZZ$--graded rings by Proposition \ref{prop: surjections respect gradings}, and similarly for their restrictions $\Phi_{\bv, \bv'}^+: \qzastplus(\bv) \twoheadrightarrow \qzastplus(\bv')$.
This brings us to our main definitions:
\begin{Definition}
    The \emph{quantized limit zastava} is the algebra
    $$
    \qzast = \lim_{\longleftarrow} \qzast(\bv)
    $$
    Its \emph{positive part} is the algebra
    $$
    \qzastplus = \lim_{\longleftarrow} \qzastplus(\bv).
    $$
    We will sometimes write $\qzast_\kk$ and $\qzastplus_\kk$ to indicate the choice of base ring $\kk$.
\end{Definition}

In the limit, the inclusions $\qzastplus(\bv) \subset \qzast(\bv)$ yield $\qzastplus \subset \qzast$.  Similarly,  the inclusions $\GT(\bv) \subset \qzast(\bv)$ from \eqref{eq: integrable system} give $\GT \subset \qzast$.  Since these  inverse systems consist entirely of surjective maps, there are natural surjections $\GT \twoheadrightarrow \GT(\bv)$, $\qzast \twoheadrightarrow \qzast(\bv)$ and $\qzastplus \twoheadrightarrow \qzastplus(\bv)$ fitting into commutative diagrams
\begin{equation}
    \begin{tikzcd}
        \GT \ar[d,two heads] \ar[r,hook] & \qzast \ar[d,two heads]  & \qzastplus \ar[l,hook'] \ar[d, two heads] \\ \GT(\bv) \ar[r,hook] & \qzast(\bv) & \qzastplus(\bv) \ar[l, hook']
    \end{tikzcd}
\end{equation}
The results of \S \ref{sec: prelim} yield some additional basic properties of  $\qzast$, which we summarize:

\begin{Theorem}
    \label{thm:limzastproperties}    
    \begin{enumerate}
        \item Let $\bm \geq \mathbf 0$ and $\bmu = (1^\bm)$.  For any $f \in \GT^\bmu$ there is a FMO $\M_\bm(f) \in \qzast$, whose defining property is that under each natural map $\qzast \twoheadrightarrow \qzast(\bv)$ we have
        $$
        \M_\bm(f) \ \mapsto \ \left\{ \begin{array}{cl} \M_\bm(\widetilde{f}), &\text{if } \bm \leq \bv, \\ 0, & \text{otherwise} \end{array} \right.
        $$
        where $f \mapsto \widetilde{f}$ denotes  the natural map $\GT^\mu \twoheadrightarrow \GT^\bmu(\bv)$. 

        \item The ring $\qzast$ is generated by the elements $\M_\bm(f)$ for $\bm\geq \mathbf 0$ and $f \in \GT^\bmu$, and its subring $\qzastplus$ is generated by the elements $\M_\bm(f)$ where $f \in \GT^\bla_{\text res}$.

        \item $\qzast$ is a  filtered graded ring, filtered by $(\Partitions^\vertexset, \leq)$, and there is an isomorphism
        $$
        \gr \qzast \cong \bigoplus_{\bla \in \Partitions^\vertexset} \GT^\bla \cdot \sfr_\bla
        $$
        where the right side has multiplication as in \eqref{eq: rels in gr}. This isomorphism identifies the subrings $\gr \qzastplus \cong \bigoplus_{\bla \in \Partitions^\vertexset} \GT^\bla_\res \cdot \sfr_\bla$.
        \item $\qzastplus$ and $\qzast$ are free modules over $\base$, as well as over $\kk$. They behave well under base-change, in particular
        $$
        \qzast_\kk = \qzast_\ZZ \otimes_\ZZ \kk, \qquad \qzastplus_\kk = \qzastplus_\ZZ \otimes_\ZZ \kk
        $$
        and there are natural inclusions $\qzast_\ZZ \subset \qzast_\QQ$ and $\qzastplus_\ZZ \subset \qzastplus_\QQ.$        
        \item $\qzast$ is free as a left (or right) module over $\GT$, and multiplication gives isomorphisms
        $$
        \GT \otimes_\base \qzastplus  \stackrel{\sim}{\longrightarrow} \qzast, \qquad \qzastplus \otimes_\base \GT \stackrel{\sim}{\longrightarrow} \qzast
        $$
                
        \item The kernel of the surjection $\qzast \twoheadrightarrow \qzast(\bv)$ is generated by the following FMOs:
        \begin{enumerate}[label=(\roman*),topsep=3pt,itemsep = 5pt]
            \item $\M_\bm(f)$ for $ \mathbf{0} \leq \bm \leq \bv$, such that $f$ is in the kernel of the homomorphism $\GT^{\bmu} \twoheadrightarrow \GT^{\bmu}(\bv)$.
            \item $\M_\bm(f)$ for all $\bm \not\leq \bv$, and all elements $f\in \GT^{\bmu}$.
        \end{enumerate}

        \item Let $(\bk, d) \in \ZZ^I \times \ZZ$.  If $\bv'$ satisfies the conditions of Lemma \ref{lem:stabilization}, then the map $\qzast_{(\bk,d)} \twoheadrightarrow \qzast(\bv')_{(\bk,d)}$ is an isomorphism.

        \item Let $\bk \in \NN^I$. If $\bv \geq \bk$, then the map on $\ZZ^\vertexset$-homogeneous components $\qzastplus_{\bk} \rightarrow \qzastplus(\bv)_{\bk}  $ is an isomorphism. In particular $\qzastplus_\bk \xrightarrow{\sim} \qzastplus(\bk)_\bk$.
    \end{enumerate}
\end{Theorem}
\begin{proof}
    These properties are straightforward limit versions of the results from the previous section.  Part (1) holds by construction of the limit, and the fact that the surjections $\Phi_{\bv,\bv'}: \qzast(\bv) \twoheadrightarrow \qzast(\bv')$ respect FMOs.  By Lemma \ref{lem:surjectionsarefilt}  the maps $\Phi_{\bv,\bv'}$ are strictly filtered, and the associated graded map $\gr \Phi_{\bv,\bv'}$ is given explicitly on the basese $f\sfr_\bla$.  Part (3) of the theorem is then immediate, which also implies (4). Given (1) and (3), we can repeat the proof of Proposition \ref{generation by FMOs} to prove part (2). We note that part (5) is a consequence of Lemma \ref{lem: truncated triangular decomp}, the isomorphism from Lemma \ref{lem: GT bla free}, and part (3). Finally, part (7) follows from Lemma \ref{lem:stabilization}, and similarly for part (8) using Remark \ref{rem:stabilization}.
\end{proof}

We note that part (8) of the theorem shows that $\qzastplus = \displaystyle{\lim_{\longleftarrow}} \qzastplus(\bv)$ may be equivalently thought of as the limit in the category of $\ZZ^\vertexset$-graded rings, since the $\ZZ$-gradings stabilize. 

\begin{Remark}
    \label{rem: limit monopole operators}
    As for the algebras $\qzast(\bv)$, in the ring $\qzast$ we may define elements $\M_\bla(f)$ for any $\bla \in \Partitions^\vertexset$ and $f \in \GT^\bla$, as sums of products of FMOs. Although the definition of $\M_\bla(f)$ depends on some choices in general, it is unique modulo lower filtered terms, and  lifts the element $f \sfr_\bla \in \gr \qzast$.  In particular, one can construct  bases for $\qzast$ and $\qzastplus$ using such elements.
\end{Remark}

%%%%%%%%%%%%%%%%%%%%%%%%%%%%%%%%%%%%%%%%%%%%%%%%%%%%%%%%%%%%%%%%%
\section{The monopole formula and Kac polynomials}
\label{sec:monopolesandkac}

%%%%%%%%%
\subsection{The monopole formula}
Recall from Theorem \ref{thm: grading} that the ring $\qzast(\bv)$ has a $\ZZ^\vertexset \times \ZZ$--grading. Denoting its homogeneous components by $\qzast(\bv)_{(\bk, d)}$ for $(\bk, d) \in \ZZ^{\vertexset} \times \ZZ$, each component is a free module over $\kk$, and the Hilbert series of $\qzast(\bv)$ is defined as:
\begin{equation}
\fJ_\bv(\bz,q) \ = \ \sum_{(\bk, d) \in \ZZ^\vertexset \times \ZZ} \big( \dim_\kk \qzast(\bv)_{(\bk, d)} \big) \bz^\bk  q^d  \ \ \in \ \ \ZZ(\!(q)\!) [\![ z_i : i \in \vertexset ]\!]
\end{equation}
Here we have used multi-index notation $\bz^\bk = \prod_i z_i^{k_i}$. 

Given a tuple $\bm = (\sm_i)_{i \in \vertexset} \in \NN^\vertexset$ of non-negative integers,  denote
\begin{equation}
(q)_\bm = \prod_{i \in \vertexset} \prod_{r = 1}^{\sm_i} (1-q^r)
\end{equation}
We will also write $(q)_{\boldsymbol{\infty}} = \prod_{i \in \vertexset} \prod_{r \geq 1} (1-q^r)$
\begin{Remark}
    \label{eq: Hilb for base and GT}
    We will also consider Hilbert series of other  $\ZZ^\vertexset\times\ZZ$--graded rings. For example, the Hilbert series of the rings $\base$, $\GT(\bv)$  and $\GT$ are given, respectively, by:
\begin{equation*}
    \frac{1}{(1-q)^{|\arrowset|+1}}, \quad \frac{1}{(1-q)^{|\arrowset|+1}} \frac{1}{(q)_\bv},  \quad \text{and}\quad  \frac{1}{(1-q)^{|\arrowset|+1}} \frac{1}{(q)_{\boldsymbol{\infty}}}
\end{equation*}
\end{Remark}

By considering monopole operators as functions on the Coulomb branch, Cremonesi-Hanany-Zaffaroni \cite{CHZ2} argued that  the Hilbert series of any (quantized) Coulomb branch is given  by the \textbf{monopole formula}.  It is proven in \cite[\S 2(iii)]{BFN1} that the BFN construction of Coulomb branches satisfies this formula. 

These considerations apply for the subring $\qzast(\bv)$ of the quantized Coulomb branch, yielding:

\begin{Theorem}
\label{thm: monopole formula}
The Hilbert series  of $\qzast(\bv)$ and $\qzastplus(\bv)$ are given, respectively, by
\begin{align*}
    \fJ_\bv(\bz,q) &= \frac{1}{(1-q)^{|\arrowset|+1}} \sum_{\bla \in \Partitions^\vertexset_{\leq \bv}} \frac{q^{ \frac{1}{2} (\bla,\bla)}}{ \prod_{k \geq 0}  (q)_{\bm_k}} \bz^{|\bla|}, \\
    \text{and } \quad \fJ^+_\bv(\bz,q) &= \frac{1}{(1-q)^{|\arrowset|+1}} \sum_{\bla \in \Partitions^\vertexset_{\leq \bv}} \frac{q^{ \frac{1}{2} (\bla,\bla)}}{ \prod_{k \geq 1}  (q)_{\bm_k}} \bz^{|\bla|},
\end{align*}
where we write $\bla = (1^{\bm_1} 2^{\bm_2}\cdots)$ in exponential notation and denote $\bm_{0} = \bv - |\bla|$.
\end{Theorem}

\begin{proof}
These formulas counts the degrees of $\kk$-bases of monopole operators $\M_{\bla}(f)$.  Note that the leading factors $(1-q)^{-|\arrowset|-1}$ encode the Hilbert series of the ring $\base$, as in Remark \ref{eq: Hilb for base and GT}.

In more detail, by the arguments from \cite[\S 2(iii)]{BFN1} and \cite[Remark 2.8(2)]{BFN1}, for the grading from Theorem \ref{thm: grading} the Hilbert series of $\qzast(\bv)$ is given by 
\begin{equation}
\label{eq: monopole formula}
\fJ_\bv(q,z) \ = \  \sum_{\bla \in \Partitions^\vertexset_{\leq \bv}} q^{ \frac{1}{2}( \bla, \bla) } P_{\bv, \bla}(q) \bz^{|\bla|},
\end{equation}
where  $P_{\bv,\bla}(q)$ denotes the Hilbert series for the $\ZZ$--graded ring $\GT^{\bla}(\bv)$ from (\ref{eq: GT for bla}).  Recall that  viewed as a dominant coweight of $\GL(V_i)$, we have 
$$
\la_i = (\cdots, \underbrace{2,\ldots,2}_{\sm_{i,2}},\underbrace{1,\ldots,1}_{\sm_{i,1}}, \underbrace{0,\ldots,0}_{\sm_{i,0}})
$$
Thus, the stabilizer of $\la_i$ in the Weyl group $S_{\sv_i}$ is $\prod_{k \geq 0} S_{\sm_{i,k}}$. Using this we see that the partially-symmetric polynomial ring $\GT^\bla(\bv) = \tGT(\bv)^{S_{\bla}}$ has Hilbert series 
$$
P_{\bv, \bla}(q) = \frac{1}{(1-q)^{|\arrowset|+1}} \frac{1}{\prod_{i \in \vertexset} \prod_{k \geq 0} \prod_{a = 1}^{\sm_{i,k}}( 1 - q^a)}.
$$
This gives the Hilbert series $\qzast(\bv)$.  The case of $\qzastplus(\bv)$ is similar, using the rings $\GT_\res^\bla(\bv)$.
\end{proof}

\begin{Remark}
Writing $\bla = (1^{\bm_1} 2^{\bm_2} \cdots)$ in exponential notation and defining $\bm_0 = \bv - |\bla|$, we may think of the monopole formula as a sum over the exponents $\bm_k$.  Using the rightmost formula in (\ref{eq: pairing on tuples of partitions}) for the pairing  $\tfrac{1}{2} (\bla,\bla)$, we obtain:
$$
    \fJ_\bv(\bz, q) = \frac{1}{(1-q)^{|\arrowset|+1}}  \sum_{\substack{\bm_0, \bm_1, \ldots \in \NN^\vertexset, \\  \bv =\sum_{k\geq 0}  \bm_k }} \frac{q^{\tfrac{1}{2} \sum_{i,j\in \vertexset} c_{ij} \sum_{k,\ell \geq 1} \min\{ k,\ell\} \sm_{i,k} \sm_{j,\ell}}}{\prod_{k \geq 0} (q)_{\bm_k}}  \bz^{{\sum_{k\geq 1} k \bm_k}}
$$
This is \emph{exactly} a {fermionic sum} on a semi-infinite interval in the sense of Feigin-Feigin-Jimbo-Miwa-Mukhin \cite[\S 2.2]{FFJMM}.	  Applying \cite[Proposition 2.4]{FFJMM}, we conclude that the Hilbert series $\fJ_\bv(\bz,q)$ satisfy \textbf{fermionic recursion}:
$$
\fJ_\bv(\bz,q) = \sum_{\substack{\bu \in \NN^\vertexset \\ \mathbf{0} \leq \mathbf{u} \leq \bv}} \frac{q^{\frac{1}{2}(\mathbf{u}, \mathbf{u})} \bz^{\mathbf{u}}}{(q)_{\bv - \bu}} \fJ_{\mathbf{u}}(\bz,q)
$$
By induction, this formula uniquely determines the $\fJ_\bv(\bz, q)$ starting from the case $\fJ_{\mathbf{0}}(\bz,q) = (1-q)^{-|\arrowset|-1}$. 
In finite ADE types and affine type A, this recovers a result of Braverman-Finkelberg \cite{BF3}, who proved (by different techniques) that the Hilbert series of zastava spaces satisfy fermionic recursion.  We note that a proof very similar to ours for finite ADE types appears in recent work of Labelle \cite[Thm.~2.1]{Labelle}, which is also based on the monopole formula.
\end{Remark}

\begin{Remark}
    Any tuple of partitions $\bla \in \Partitions_{\leq \bv}^\vertexset$ can also be encoded in terms of the tuples $\bc_1, \bc_2,\ldots \in \NN^\vertexset$ which encode its column lengths.  In other words, writing $\bc_s = (\bc_{i,s})_{i \in \vertexset}$, the conjugate partition to $\la_i$ is precisely $\la'_i = (\bc_{i,1}, \bc_{i,2},\ldots)$.  Note that these column lengths  $\bv \geq \bc_1 \geq \bc_2 \geq \ldots$ are non-increasing and satisfy $ \bc_s = \mathbf{0}$ for $s \gg 0$,    and the set of such sequences is in bijection with  $\Partitions_{\leq \bv}^\vertexset$.   Viewing the monopole formula as a sum over such sequences, we obtain:
    $$
        \fJ_\bv(\bz,q) = \frac{1}{(1-q)^{|\arrowset|+1}} \sum_{\substack{\bc_1,\bc_2,\ldots \in \NN^\vertexset \\ \bv \geq \bc_1 \geq \bc_2 \geq \ldots \\ \bc_s =0 \text{ for } s \gg 0 }} \frac{q^{\frac{1}{2} \sum_{s \geq 1} (\bc_s, \bc_s)}}{(q)_{\bv - \bc_1} \prod_{s \geq 1} (q)_{\bc_s - \bc_{s+1}}} \bz^{\sum_{s\geq 1} \bc_s}
    $$
\end{Remark}

\subsection{Kac polynomials and Hua's formula}
\label{ssec: Kac and Hua}
For any quiver $Q = (\vertexset, \arrowset)$ one can associate a root system $\Delta_Q$, see for example \cite[\S 1.1]{Kac2}.  In particular $\Delta_Q$ can be thought of as subsets of the root lattice $\ZZ^\vertexset$, with the simple roots $\alpha_i$ corresponding to the standard basis. We'll denote the set of positive roots by $\Delta_Q^+ = \Delta_Q \cap \NN^\vertexset$.

Kac \cite[Proposition 1.15]{Kac2} associated a polynomial $A_\alpha( q) \in \ZZ[q]$ to each positive root $\alpha \in \Delta_Q^+$, which counts the number of absolutely decomposable representations of dimension $\alpha$ for $Q$ over finite fields: the celebrated \textbf{Kac polynomials}.   Kac showed that $A_\alpha(q)$ is a monic polynomial of degree $1 - \frac{1}{2}(\alpha,\alpha)$, and  conjectured that the coefficients of $A_Q(\alpha,q)$ are non-negative integers, which was subsequently proven by Hausel-Letellier-Rodriguez-Villegas \cite{HLRV}.
We may thus  write 
\begin{equation}
    \label{eq:KacPolynomials}
    A_Q(\alpha, q) = \sum_{s=0}^{1-\frac{1}{2}(\alpha,\alpha)} t_s^\alpha q^s,
\end{equation}
where all $t_s^\alpha \in \NN$ and $t_{1-\frac{1}{2}(\alpha,\alpha)}^\alpha = 1$.

Hua established the following remarkable formula, which provides a generating function for Kac polynomials:

\begin{Theorem}[\mbox{\cite[Theorem 4.9]{Hua}}]
There is an equality of formal series:
$$
\sum_{\bla \in \Partitions^\vertexset} \frac{q^{\frac{1}{2}(\bla, \bla)}}{\prod_{i\in \vertexset} \prod_{k\geq 1} (q)_{\bm_k} }\bz^{|\bla|} \ = \ \prod_{\alpha \in \Delta_Q^+} \prod_{r\geq 0} \prod_{s=0}^{1 - \frac{1}{2} (\alpha,\alpha)} ( 1 -  q^{-r-s} \bz^\alpha)^{t_s^\alpha} 
$$
where we write each multipartition $\bla = (1^{\bm_1} 2^{\bm_2}\cdots)$ in exponential notation.
\end{Theorem}
The left hand side of Hua's formula is naturally an element of $\ZZ(q)[\![z_i : i \in \vertexset]\!]$, while the right hand side is its Laurent expansion in $\ZZ(\!(q^{-1})\!)[\![z_i : i \in \vertexset]\!]$. We note that the  formulation given above differs from Hua's original by replacing  $q$ with $q^{-1}$ throughout.

For our purposes, we need the Laurent expansion of Hua's formula in the opposite direction, i.e.~as an element of $\ZZ(\!(q)\!)[\![z_i : i \in \vertexset]\!]$:
\begin{Corollary}
    \label{cor: version of Hua}
    There is an equality of formal series:
    $$
    \sum_{\bla \in \Partitions^\vertexset} \frac{q^{\frac{1}{2}(\bla, \bla)}}{\prod_{i\in \vertexset} \prod_{k\geq 1} (q)_{\bm_k} }\bz^{|\bla|} \ = \ \prod_{\alpha \in \Delta_Q^+} \prod_{r \geq 1} \prod_{s = 0}^{1 - \frac{1}{2}(\alpha,\alpha)}\frac{1}{(1-q^{r-s} \bz^\alpha)^{t^\alpha_s}}
    $$
\end{Corollary}

\begin{proof}
    Rewrite  the right hand side of Hua's formula using the plethystic exponential $\operatorname{Exp}$ (see for example \cite[Appendix A]{Mozgovoy} for its definition and properties):
    $$
    \sum_{\bla \in \Partitions^\vertexset} \frac{q^{\frac{1}{2}(\bla, \bla)}}{\prod_{i\in \vertexset} \prod_{k\geq 1} (q)_{\bm_k} }\bz^{|\bla|} = \operatorname{Exp}\left( \frac{1}{q^{-1}-1}\sum_\alpha A_\alpha(q^{-1}) \bz^\alpha \right)
    $$
    Written in this way, the Laurent expansion in $\ZZ(\!(q)\!)[\![z_i : i \in \vertexset]\!]$  is seen to be:
    $$
    \operatorname{Exp}\left( \frac{q}{1-q}\sum_\alpha A_\alpha(q^{-1}) \bz^\alpha \right) = \prod_{\alpha \in \Delta^+_Q} \prod_{r \geq 1 } \prod_{s=0}^{1 - \frac{1}{2}(\alpha,\alpha)} \frac{1}{(1-q^{r-s}\bz^\alpha)^{t^\alpha_s}}
    $$
\end{proof}

\subsection{The Hilbert series of the quantized limit zastava}
\label{ssec:hilb-for-limit-zastava}
We now come to the main results of this section. 
\begin{Theorem}
The Hilbert series for the quantized limit zastava $\qzast$  and its positive part  $\qzastplus$ are well-defined, and given explicitly by:
\begin{align*}
\fJ(q,z) &=  \frac{1}{(1-q)^{|\arrowset|+1}}  \frac{1}{(q)_{\boldsymbol{\infty}}}\sum_{\bla \in \Partitions^\vertexset} \frac{q^{\frac{1}{2}(\bla, \bla)}}{\prod_{i\in \vertexset} \prod_{k\geq 1} (q)_{\bm_k} }z^{|\bla|}   \\
\text{and} \quad \fJ^+(q,z) &=\frac{1}{(1-q)^{|\arrowset|+1}} \sum_{\bla \in \Partitions^\vertexset} \frac{q^{\frac{1}{2}(\bla, \bla)}}{\prod_{i\in \vertexset} \prod_{k\geq 1} (q)_{\bm_k} }z^{|\bla|}
\end{align*}
\end{Theorem}
The difference between $\fJ$ and $\fJ^+$ reflects the tensor product decomposition $\qzast = \GT \otimes_\base \qzastplus$, since the character of the ring $\GT$ is precisely $(1-q)^{-|\arrowset|-1} (q)_{\boldsymbol{\infty}}^{-1}$ while that of $\base $ is $(1-q)^{-|\arrowset|-1}$.

\begin{proof}
Both formulas follow by counting bases of monopole operators $\M_\bla(f)$, defined as in Remark \ref{rem: limit monopole operators}.  Alternatively, these formulas we  simply take  the limit as $\bv \rightarrow \infty$ of  Theorem \ref{thm: monopole formula}. The formulas there  are easily seen to stabilize, and the factors of $(q)_{\bm_0}$ which appear in the denominator of $\fJ_\bv(z,q)$ limit to give $(q)_{\boldsymbol{\infty}}$.

\end{proof}

Combining this theorem with Hua's identity as in Corollary \ref{cor: version of Hua}, we obtain the main result of this section:
\begin{Corollary}
    \label{cor:HilbandKac}
    The Hilbert series of $\qzast$ and $\qzastplus$ are given by:
    \begin{align*}
        \fJ(z,q) & = \frac{1}{(1-q)^{|\arrowset|+1}}  \frac{1}{(q)_{\boldsymbol{\infty}}} \prod_{\alpha \in \Delta_Q^+} \prod_{r\geq 1} \prod_{s=0}^{1 - \frac{1}{2} (\alpha,\alpha)} \frac{1}{( 1 -  q^{r-s} z^\alpha)^{t_s^\alpha}} \\
        \text{and} \quad  \fJ^+(z,q) & = \frac{1}{(1-q)^{|\arrowset|+1}}   \prod_{\alpha \in \Delta_Q^+} \prod_{r\geq 1} \prod_{s=0}^{1 - \frac{1}{2} (\alpha,\alpha)} \frac{1}{( 1 -  q^{r-s} z^\alpha)^{t_s^\alpha}}
    \end{align*}
\end{Corollary}

\begin{Remark}
    \label{rmk:finiteADE}
    Let $\quiver$ be of finite ADE type, with associated finite-dimensional simple Lie algebra $\fg_\quiver$.  Then $\Delta^+_\quiver$ is the set of positive roots for $\fg_\quiver$, and the Kac polynomials $A_\alpha(q) = 1$ for all $\alpha \in \Delta_\quiver^+$. Thus, in this case the Hilbert series of $\qzast$ is given by:
    $$
    \fJ(z,q) = \frac{1}{(1-q)^{|\arrowset|+1}} \frac{1}{(q)_{\boldsymbol{\infty}}} \prod_{\alpha \in \Delta^+_\quiver} \prod_{r\geq 1} \frac{1}{1-q^r z^\alpha}
    $$
    Suppose that $\QQ \subseteq \kk$. Then we may define the (negative) Borel Yangian $Y_\hbar(\fb)$ in its Drinfeld-Gavarini integral form following \cite[\S 5.4]{FKPRW}, as the Rees algebra of an algebra with generators $H_i^{(r)}$ and $F_i^{(r)}$ and relations as in \cite[Definition 3.1]{FKPRW}, \cite[\S A.5]{FT2}. Similarly to \cite[Appendix B]{BFN2}, there are surjective maps 
    $$
    Y_\hbar(\fb)[\theta_e : e \in \arrowset] \twoheadlongrightarrow \qzast(\bv)
    $$
    for all $\bv \in \NN^\vertexset$. We obtain a natural map $Y_\hbar(\fb)[\theta_e : e \in \arrowset] \longrightarrow \qzast$ to the limit quantized zastava, which is an isomorphism by comparing Hilbert series.  This isomorphism identifies subalgebras $Y_\hbar(\fh)[\theta_e] \cong \GT$ and $Y_\hbar(\fu)[\theta_e] \cong \qzastplus$.
\end{Remark}

%%%%%%%%%%%%%%%%%%%%%%%%%%%%%%%%%%%%%%%%%%%%%%%%%%%%%%%%%%%%%%%%%
\section{Shuffle algebras}
\label{sec: shuffle}

In this section, we will assume that our base ring $\kk$ contains $\QQ$. This is because the formulas in Theorem \ref{thm:FFT-homomorphism} and Theorem \ref{thm:lowest-part-of-filtration-is-fmo}
require division by integers. See \S \ref{ssec:integralform} for a variation with $\kk = \ZZ$.

\subsection{The big shuffle algebra}

For every dimension vector $\bk \in \NN^\vertexset$, consider a tensor product of symmetric polynomial rings over the base ring $\base$ from \S \ref{sec: integrable system}:
\begin{equation}
    \label{eq:preshu1}
    \preshu_\bk = \base[x_{i,a} : i \in \vertexset, 1 \leq a \leq \bk_i]^{S_\bk},
\end{equation}
where the product of symmetric groups $S_\bk = \prod_{i \in I} S_{\bk_i}$ acts with $S_{\bk_i}$ permuting the variables $x_{i,a}$.  
Put differently, $\preshu_\bk$ is the equivariant cohomology ring $H_{ \prod_i \GL(\bk_i) \times \CC^\times \times \bF}^\bullet(pt)$.  There is a natural $\ZZ$-grading on $\preshu_\bk$  by $\frac{1}{2}$-homological degree, or in other words with the generators $\frac{1}{2}\hbar, \theta_e, x_{i,a}$ all in degree 1.

For any pair of vertices $i,j \in \vertexset$, define the following shuffle kernel:
\begin{align}
  \label{eq:3}
  \zeta_{i,j}(x) = \left( \frac{x - \hbar}{x} \right)^{\delta_{i,j}} \prod_{ e : i \rightarrow j} (-x + \theta_e - \tfrac{1}{2} \hbar) \prod_{e:j \rightarrow i} (x + \theta_e + \tfrac{1}{2}\hbar)
\end{align}
For  $f\in \preshu_{\bk'}$ and $g \in \preshu_{\bk''}$, their \textbf{shuffle product} $f\ast g \in \preshu_{\bk'+\bk''}$ is defined to be:
\begin{equation}
    \label{eq:shuffleprod}    
    \sum_{\sigma \in S_{\bk'+\bk''}/(S_{\bk'} \times S_{\bk''})} \sigma \Bigg( f(\{x_{i,a}\}^{i\in \vertexset}_{1 \leq a \leq \bk_i'}) g(\{x_{i,b}\}^{i\in \vertexset}_{\bk_i' < b \leq \bk_i'+\bk_i''} ) \prod_{i,j \in \vertexset} \prod_{\substack{ 1 \leq a \leq \bk_i', \\ \bk_j' < b \leq \bk_j' + \bk_j''}} \zeta_{i,j}(x_{i,a} - x_{j,b})\Bigg)
\end{equation}
Although a priori this expression is a rational function because of the presence of denominators in \eqref{eq:3}, it is well-known to be regular.  It is also well-known that the shuffle product defines an associative operation, with unit element $1 \in \preshu_{\mathbf 0}$.

\begin{Definition}
The {\bf big shuffle algebra} is the vector space
\begin{align}
  \label{eq:5}
  \preshu = \bigoplus_{\bk \in \NN^I} \preshu_{\bk}
\end{align}
endowed with the associative ring structure given by the shuffle product~\eqref{eq:shuffleprod}.  
\end{Definition}

The ring $\preshu$ carries a $\ZZ^\vertexset\times \ZZ$-grading defined as follows.  Let $\bk \in \NN^\vertexset$, and let  $f \in \preshu_\bk$ be a homogeneous element of degree $d$ with respect to the $\frac{1}{2}$-homological grading mentioned above.  Then we set:
\begin{equation}
    \label{eq:gradingonpreshu}
    \deg f = \Big(\bk, d - \sum_{e \in \arrowset} \bk_{\source(e)} \bk_{\target(e)} + \sum_{i \in \vertexset} \bk_i\Big) \ \in  \ZZ^\vertexset\times \ZZ
\end{equation}
A straightforward calculation using \eqref{eq:3} and \eqref{eq:shuffleprod} shows that this makes $\preshu$ into a $\ZZ^\vertexset\times \ZZ$-graded ring.

\subsection{Specializing variables}

\subsubsection{Specializing variables at a multiset}
Recall our conventions on multisets from \S \ref{subsubsec:multisets}. Consider also variables $\{ w_{i,r} \}_{i \in I, r \in [\bv_i]}$. Let us fix an $I$-colored multiset $A = (A_i)_{i \in \vertexset}$, and let $\bk = |A|$. For any $f \in \preshu_{\bk}$, we define its specialization $f\vert_A$ as follows: for each $i \in I$, we specialize  the variables $x_{i,a}$ for $1 \leq a \leq \bk_i$ to the union of arithmetic progressions
\begin{align}
  \label{eq:12}
  \bigcup_{r \in [\bv_i]} \{ w_{i,r} , \cdots , w_{i,r} + (A_i(r) - 1) \hbar \}
\end{align}
Similar specializations have been considered in \cite[Def.~1.6]{FHHSY}, \cite[Thm.~4.11]{FT2}, \cite[Thm.~B.17]{Frassek-Tsymbaliuk-2022},  \cite[Def.~3.1]{Negut2}, and \cite[Def.~2.5]{JN26}.

Now consider the case when $C$ is a multiset such that $A$ is a multisubset of $C$. That is, for all $i \in I$ and $r \in [\bv_i]$, we have $A_i(r) \leq C_i(r)$. Write $\bk' = |A|$ and $\bk = |C|$, and define $\bk''$ so that $\bk = \bk' + \bk''$. 
For $g \in \preshu_{\bk''}$, we define its specialization $g \vert_{C:A}$ as follows: for each $i \in I$, we specialize the variables $x_{i,a}$ for $1 \leq a \leq \bk_i$ to the values
\begin{align}
  \label{eq:13}
  \bigcup_{r \in [\bv_i]} \{ w_{i,r} + A_i(r) , \cdots , w_{i,r} + (C_i(r) - 1) \hbar \}
\end{align}
Observe that specifying $C$ is equivalent to specifying the unique multiset $B$ such that $C = A + B$. However, we warn the reader that $g\vert_{C:A}$ is not equal to $g\vert_B$ unless $A$ is the empty multiset.

The following Theorem explains how the above specializations of variables interacts with the shuffle product.

\begin{Theorem}
  \label{thm:specializing-variables-and-shuffle-product}
  Fix a dimension vector $\bv \in \NN^I$. Let $\bk',\bk'' \in \NN^I$. If $f \in \preshu_{\bk'}$ and $g \in \preshu_{\bk''}$, then for all $I$-colored multisets $C$ supported on $[\bv]$ with $|C| = \bk' + \bk''$, we have:
  \begin{align*}
    (f \star g )\vert_{C} = \sum_{A,B: A+B = C, |A| = \bk', |B| = \bk''} \frac{C!}{A! B!}  f\vert_A g\vert_{C:A}
    \prod_{i \in I} \prod_{r,s \in [\bv_i], r \neq s} \frac{(w_{i,r} - w_{i,s} - C_i(s) \hbar)^{\overline{A_i(r)}}}{(w_{i,r} - w_{i,s} - A_i(s)\hbar)^{\overline{A_i(r)}}} \cdot \\
    \prod_{e \in E} \prod_{r = 1}^{\bv_{s(e)}} \prod_{s = 1}^{\bv_{t(e)}}  \frac{(w_{t(e),s} - w_{s(e),r} + \theta_e - \tfrac{1}{2} \hbar)^{\underline{A_{s(e)}(r)} \cdot \overline{C_{t(e)}(s)}}}{ (w_{t(e),s} - w_{s(e),r} + \theta_e  -\tfrac{1}{2}\hbar )^{\underline{A_{s(e)}(r)} \cdot \overline{A_{t(e)}(s)}}} \cdot 
 \frac{ (w_{t(e),s} - w_{s(e),r} + \theta_e + \tfrac{1}{2}\hbar)^{\underline{C_{s(e)}(r)} \overline{A_{t(e)}(s)} }} { (w_{t(e),s} - w_{s(e),r} + \theta_e + \tfrac{1}{2}\hbar)^{\underline{A_{s(e)}(r)} \overline{A_{t(e)}(s)} } }
  \end{align*}
\end{Theorem}

\begin{proof}
    The definition of the shuffle $f \star g$ is a sum over all the ways of redistributing the arguments of $f$ and $g$. However, because of the factors $(\frac{x-\hbar}{x})^{\delta_{i,j}}$ in \eqref{eq:3}, a term will vanish if it involves an argument of $f$ set equal to an argument of $g$ plus $\hbar$. Therefore, when we compute $(f \star g)\vert_C$, the only non-vanishing terms come exactly from decompositions $A + B = C$ as in the statement of the theorem and specializations $f\vert_A$ and $g\vert_{C:A}$ as above.  (That is, the variables corresponding to $f$ and $g$ are specialized as in  \eqref{eq:12} and \eqref{eq:13}, respectively.) 
    The statement of the theorem follows by computing the corresponding specializations of the remaining factors $\zeta_{i,j}(x_{i,a}-x_{j,b})$ from the shuffle product \eqref{eq:shuffleprod}    .

    For $i = j$, the leading factor $\frac{x-\hbar}{x}$ in \eqref{eq:3} contributes terms which telescope: for each $r, s \in [\bv_i]$ we get a factor of
    \begin{align*}
         \prod_{\substack{0 \leq k < A_i(r), \\ A_i(s) \leq \ell < C_i(s)}}\frac{(w_{i,r} + k \hbar) - (w_{i,s} + \ell \hbar) - \hbar}{(w_{i,r} + k \hbar) - (w_{i,s} + \ell)}  =   \frac{(w_{i,r} - w_{i,s} - C_i(s) \hbar)^{\overline{A_i(r)}}}{(w_{i,r} - w_{i,s} - A_i(s)\hbar)^{\overline{A_i(s)}}}
    \end{align*}
    When $r = s$, this further simplifies to $\frac{C_i(r)!}{A_i(r)! B_i(r)!}$, and thus $\frac{C!}{A!B!}$ when taken over $i\in \vertexset$ and $r \in [\bv_i]$.

    The remaining factors in \eqref{eq:3} correspond to arrows between $i$ and $j$.  For each $e:i\rightarrow j$ and all $r \in [\bv_i], s \in [\bv_j]$, we get:
    \begin{align*}
        \prod_{\substack{0 \leq k < A_i(r), \\ A_j(s) \leq \ell < C_j(s)}} \big(-(w_{i,r}+k\hbar)+(w_{j,s}+\ell\hbar) + \theta_e - \tfrac{1}{2}\hbar\big) & = \big(w_{j,s}-w_{i,r} + \theta_e +( A_j(s) -\tfrac{1}{2})\hbar\big)^{\underline{A_i(r)}\cdot \overline{B_j(s)}}\\
        & = \frac{ (w_{j,s} - w_{i,r} + \theta_e - \frac{1}{2} \hbar)^{\underline{A_i(r)}\cdot \overline{C_j(s)}}}{(w_{j,s} - w_{i,r} + \theta_e - \frac{1}{2} \hbar)^{\underline{A_i(r)}\cdot \overline{A_j(s)}}}
    \end{align*}
    Similarly, for each  $e:j \rightarrow i $ and all $r \in [\bv_i], s \in [\bv_j]$ we get :
    $$
        \prod_{\substack{0 \leq k < A_i(r), \\ A_j(s) \leq \ell < C_j(s)}} \big((w_{i,r}+k\hbar)-(w_{j,s}+\ell\hbar) + \theta_e + \tfrac{1}{2}\hbar\big) = \frac{ (w_{i,r} - w_{j,s} + \theta_e + \frac{1}{2} \hbar)^{\overline{A_i(r)}\cdot \underline{C_j(s)}}}{(w_{i,r} - w_{j,s} + \theta_e + \frac{1}{2} \hbar)^{\overline{A_i(r)}\cdot \underline{A_j(s)}}}
    $$
    Up to some mild reindexing, this gives the claimed formula for $(f\ast g)|_C$ and completes the proof.
\end{proof}

\subsection{The homomorphism to GKLO difference operators}
\label{ssec: GKLO hom}

We have the following theorem due to Frassek and Tsymbaliuk.

\begin{Theorem}[\mbox{\cite[Theorem B.17]{Frassek-Tsymbaliuk-2022}}]
  \label{thm:FFT-homomorphism}
  Let $\bv \in \NN^\vertexset$  There is a homomorphism  of $\ZZ^\vertexset\times \ZZ$-graded rings

  \begin{equation}
    \label{eq:54}
\Phi^+_\bv: \preshu \rightarrow \GKLO{\bv}^{S_\bv}  
  \end{equation}
   defined as follows: for any $\bk \in \NN^I$ and $f \in \preshu_\mathbf{k}$, we define
\begin{align}
  \label{eq:FFT-homomorphism}
  \begin{split}
\Phi^+_\bv(f) & = \sum_{A : | A | = \mathbf{k}} \frac{1}{A!} f|_A \frac{1}{\prod_{i\in I} \prod_{r,s \in [\bv_i], r \neq s}(w_{i,r}-w_{i,s} - A_i(s)\hbar)^{\overline{A_i(r)}}} \\
& \quad\ \ \  \ \ \  \times \frac{1}{\prod_{e \in E} \prod_{\substack{r \in [\bv_{s(e)}] \\ s \in [\bv_{t(e)}]}} (w_{t(e),s} - w_{s(e),r} + \theta_e + \frac{1}{2} \hbar)^{\underline{A_{s(e)}(r)}\cdot \overline{A_{t(e)}(s)-1}}} \su^A 
  \end{split}
\end{align}

\end{Theorem}
Observe that for $\bk \leq \bv$, restricting $\Phi^+_\bv$ to the $\bk$-degree piece of $\preshu$ yields an injection:
\begin{equation}
  \label{eq:53}
  \Phi_\bv^+ : \preshu_{\bk} \hookrightarrow \left(\GKLO{\bv}^{S_{\bv}}\right)_{\bk}
\end{equation}

We mention that the format of our formula is different than that of \cite[Theorem B.17]{Frassek-Tsymbaliuk-2022}. The formula presented here is an easy consequence of Theorem \ref{thm:specializing-variables-and-shuffle-product}. We also note that Frassek and Tsymbaliuk work in a different generality than us. They consider arbitrary simple Lie algebras, that is, finite-dimensional but possibly not simply-laced. By contrast, we work with arbitrary quiver, which includes all symmetric Kac-Moody types, but not the non-symmetric ones. 

Finally, we mention that a $K$-theoretic counterpart to Theorem \Ref{thm:FFT-homomorphism} was obtained earlier by Finkelberg and Tsymbaliuk \cite[\S 4]{FT2}. For that reason, we call the homomorphism in Theorem \Ref{thm:FFT-homomorphism} the {\bf FFT homomorphism} where FFT is short for Finkelberg-Frassek-Tsymbaliuk.  See \cite[Theorem 4.1]{JN26} for a K-theoretic generalization to arbitrary quivers.

\subsection{Wheel conditions}
\label{ssec:wheel-conditions}

In general, the image of the FFT homomorphism does not land in $\qzast(\bv)$. However, there is a subalgebra $\cS \subseteq \preshu$ which does have this property. To specify $\cS$ we make the following definition.

\begin{Definition}
\label{def:refined-wheel-conditions}
Let $\bk \in \NN^I$. We say $f \in \preshu_\bk$ satisfies the {\bf refined wheel conditions} if for all $\bv \in \NN^I$ and all multisets $A$ supported on $\bv$ with $|A| = \bk$, the following conditions hold.

\begin{enumerate}[label=(\roman*)]
\item For every edge $e : i \rightarrow j$, all $r \in [\bv_i]$ and all $s\in [\bv_j]$, the specialization $f|_A$ is divisible by
$$
(w_{j,s} - w_{i,r} + \theta_e - \frac{1}{2}\hbar)^{\underline{A_i(r) - 1} \cdot \overline{A_j(s)}} \qquad \text{if } \ A_i(r) \geq A_j(s),
$$
and divisible by 
$$
(w_{j,s} - w_{i,r} + \theta_e + \frac{1}{2}\hbar)^{\underline{A_i(r)} \cdot \overline{A_j(s) - 1}} \qquad \text{if } \  A_i(r) \leq A_j(s).
$$

\item For every edge loop $e : i \rightarrow i$ and all $r \in [\bv_i]$, the specialization $f|_A$ is divisible by
$$
(\theta_e - \tfrac{1}{2}\hbar)^{\underline{A_i(r)-1} \cdot \overline{A_i(r)}} =
(\theta_e + \tfrac{1}{2}\hbar)^{\underline{A_i(r)} \cdot \overline{A_i(r)-1}}
$$
\end{enumerate}
\end{Definition}

\begin{Remark}
    In part (i) of the above definition, it is equivalent to say that $f|_A$ is divisible by ${(w_{j,s} - w_{i,r} + \theta_e - \frac{1}{2}\hbar)^{\underline{A_i(r) - 1} \cdot \overline{A_j(s)}}}$ and by $(w_{j,s} - w_{i,r} + \theta_e + \frac{1}{2}\hbar)^{\underline{A_i(r)} \cdot \overline{A_j(s) - 1}}$. Because of Lemma \ref{lem:doublePochhdiv}, one of these polynomials divides the other, depending on which of $A_i(r)$ and $A_j(s)$ is larger.  Part (i) thus simply expresses the divisibility by the greater of these two polynomials in each case.
\end{Remark}

We make a few comments about this definition. First, part (ii) of the above definition is the special case of part (i) when $i = j$ and $r = s$ and is therefore redundant. However, we will often treat this case differently in proofs. Second, in the above conditions, if the multiplicity of a multiset is $0$ at some index, then we may have an exponent in a Pochhammer symbol of $-1$ appearing. In this case, the divisibility condition is vacuous. Finally, because of the symmetry of $f\in \preshu_\bk$, it is sufficient to consider multipartitions $A$ instead of general multisets.

\begin{Remark}
    Part (i) of the above definition can be rephrased using the function $\chi_{m,n} : \ZZ \rightarrow \NN$ from \cite[Definition 2.5]{JN26}:  for every edge $e : i \rightarrow j$, all $r \in [\bv_i]$ and all $s\in [\bv_j]$, the specialization $f|_A$ is divisible by
    $$
        \prod_{c \in \ZZ} (w_{j,s} - w_{i,r} + \theta_e - \tfrac{1}{2}\hbar + c \hbar)^{\chi_{A_i(r), A_j(s)}(c)}
    $$
\end{Remark}

\begin{Definition}
\label{def:shu}
We define the {\bf shuffle algebra} $\cS \subseteq \preshu$ to be the subspace of all $f \in \preshu$ satisfying the refined wheel conditions. Note that by definition $\shu = \bigoplus_{\bk} \shu_\bk$.
\end{Definition}

The following follows directly from Theorem \Ref{thm:specializing-variables-and-shuffle-product}.

\begin{Theorem}
The shuffle algebra $\shu \subseteq \preshu$ is a $\ZZ^\vertexset\times \ZZ$-graded subring.
\end{Theorem}

\subsubsection{Comparison with 3-variable wheel conditions}
Motivated by the  work of Feigin-Odesskii on elliptic quantum groups \cite{FO}, many works on shuffle algebras utilize 3-variable wheel conditions, see for example \cite[Definition 2.9]{Negut1} or \cite[(B.11)]{Frassek-Tsymbaliuk-2022}.  

In the setting of Definition \ref{def:refined-wheel-conditions} above, these arise from those $A$ which are ``almost generic'', in the sense that there is a single pair $(i,r)$ for which $A_i(r) = 2$, while all other $A_j(s) = 1$.  
Explicitly, this leads to the following:

\begin{Definition}
    \label{def:3-variable-wheel-conditions}
Consider an element $ f\in \preshu_\bk$.  We say that $f$ satisfies the \textbf{3-variable wheel conditions} if for all arrows $e:i \rightarrow j$, the function  $f$ vanishes when we specialize at either
\begin{enumerate}
\item $x_{j,b}-  x_{i,a} + \theta_e - 1/2 \hbar = 0$ and $x_{i,c} = x_{i,a} + \hbar$, or 
\item $x_{j,a} - x_{i,b}+ \theta_e + 1/2 \hbar = 0$ and $x_{j,c} = x_{j,a} + \hbar$.
\end{enumerate}
In both cases we run over all possible indices $a,b,c$ such that $a \neq c$, with the additional stipulation that $a\neq b \neq c$ if $i = j$.
\end{Definition}

\begin{Definition}
We define  $\shu' \subseteq \preshu$ to be the subspace of all $f \in \preshu$ which  satisfy the $3$-variable  wheel conditions.
\end{Definition}

It is well-known that $\shu'$ is a subring of $\preshu$.  Our goal is to establish the following comparison result between the shuffle algebras $\shu$ and $\shu'$:

\begin{Theorem}
    \label{thm:shufflealgcomparison}
    For any quiver, there is an inclusion $\shu \subseteq \shu'$.   Moreover:
    \begin{enumerate}[(a)]
    \item If the quiver $\quiver$  contains no edge loops, then $\shu = \shu'$.
    \item For a general quiver, the above inclusion becomes an equality after localization at the  multiplicative subset of the base ring $\base$ generated by the elements:
      \begin{align}
        \label{eq:47}
        \left\{ \theta_e + (n+\tfrac{1}{2})\hbar : e \in \arrowset \text{ with } \source(e) = \target(e), \ n \in \ZZ \right\}
      \end{align}
    \end{enumerate}
\end{Theorem}

We will need the following lemma, which is well-known to experts.

\begin{Lemma}
For all $f \in \shu'_\bk$ and all $A$ of size $\bk$, we have:
\begin{enumerate}[(i)]
\item For every arrow $e : i \rightarrow j$, the specialization $f|_A$ is divisible by
$$
(w_{j,s} - w_{i,r} + \theta_e - \frac{1}{2}\hbar)^{\underline{A_i(r) - 1} \cdot \overline{A_j(s)}}
$$
for all indices $1 \leq r \leq \bv_i$ and $1 \leq s\leq \bv_j$  such that $(i,r) \neq (j,s)$.

\item  For every arrow $e : i \rightarrow j$, the specialization $f|_A$ is divisible by
$$
(w_{j,s} - w_{i,r} + \theta_e + \frac{1}{2}\hbar)^{\underline{A_i(r)} \cdot \overline{A_j(s) - 1}}
$$
for all indices $1 \leq r \leq \bv_i$ and $1 \leq s\leq \bv_j$ such that $(i,r) \neq (j,s)$.

\end{enumerate}
\end{Lemma}

\begin{proof}
We will only prove part (i), as the proof of (ii) is similar.  Let us fix an arrow  $e: i \rightarrow j$ and indices $1 \leq r \leq \bv_i$ and $1 \leq s\leq \bv_j$ such that $(i,r) \neq (j,s)$.   When we evaluate $f|_A$, recall that the variables $x_{i,a} $  for $1 \leq a \leq k_i$ are specialized at the following values:
$$
\bigcup_{r \in [\bv_i]} \left\{ w_{i,r}, w_{i,r}+ \hbar, \ldots, w_{i,r} + ( A_i(r) -1) \hbar \right\}
$$
Since $f$ is symmetric in these variables, we are free to choose the precise way in which this evaluation is performed, i.e.~how each variable is actually evaluated within this set.  Let us fix such a choice.  Then for any index $0 \leq p < A_i(r)$  there is a unique corresponding variable $x_{i,a}$ which will be evaluated at $w_{i,r}+p\hbar$.   
As an intermediate step towards computing $f|_A$, we will evaluate only those variables $x_{i,a}$ corresponding to $w_{i,r} + p \hbar$ for $ 0 \leq p < A_i(r)$, leaving all other variables $x_{k,c}$ unchanged. Denote the result of this partial evaluation by $f|_{A_i(r)}$.

For each $0\leq q < A_j(s)$, there is similarly a unique corresponding $x_{j,b}$ which will be evaluated at $w_{j,s} + q \hbar$.   For each such $x_{j,b}$, by the wheel condition (1) we see that $f|_{A_i(r)}$ vanishes when we specialize at 
\begin{equation}
\label{eq: divisible 1}
x_{j,b} - (w_{i,r}+ p \hbar) + \theta_e - \frac{1}{2} \hbar = 0,
\end{equation}
for any $0 \leq p < A_i(r) - 1$. Indeed, in $f|_{A_i(r)}$ we have evaluated $x_{i,a} = w_{i,r}+ p \hbar$ for some index $a$, and also $x_{i,c} = w_{i,r} + (p+1)\hbar$ for some index $c \neq a$, and in particular we have specialized $x_{i,c} = x_{i,a} + \hbar$.  (Note that if $i= j$ then $a\neq b \neq c$ by our assumption that $(i,r) \neq (j,s)$.) It follows that $f|_{A_i(r)}$ is divisible by
$$
\prod_{0 \leq p < A_i(r)-1} (x_{j,b} - (w_{i,r}+p\hbar) + \theta_e - \tfrac{1}{2}\hbar) = (x_{j,b} - w_{i,r} + \theta_e - \tfrac{1}{2} \hbar)^{\underline{A_i(r)-1}}
$$
Taking the product over all $b$, we conclude that $f|_{A_i(r)}$ is divisible by
\begin{equation}
\label{eq: divisible 2}
\prod_b (x_{j,b} - w_{i,r} + \nu_e - \frac{1}{2}\hbar)^{\underline{A_i(r)-1}}
\end{equation}
Now, when we further specialize the variables $x_{j,b}$ to $w_{j,s}+ q \hbar $ for $ 0 \leq q < A_j(s)$, we conclude that the result is divisible by $(w_{j,s} - w_{i,r} + \nu_e - \frac{1}{2}\hbar)^{\underline{A_i(r) - 1} \cdot \overline{A_j(s)}}$.  As $f|_A$ is the result of further specialization, it inherits this divisibility property.
\end{proof}

\begin{proof}[Proof of Theorem \ref{thm:shufflealgcomparison}]
    As noted before Definition \ref{def:3-variable-wheel-conditions}, the 3-variable wheel conditions are special cases of the refined wheel conditions. Thus $\shu \subseteq \shu'$.

    The previous lemma shows that the 3-variable wheel conditions imply almost all of the refined wheel conditions.  The only case which is not covered is where $(i,r) = (j,s)$, which  corresponds to the refined wheel condition from part (ii) of Definition \ref{def:refined-wheel-conditions}, and thus divisibility by $(\theta_e-\tfrac{1}{2}\hbar)^{\underline{A_i(r)}\cdot\overline{A_i(r)-1}}$.  This divisibility  becomes irrelevant  if we localize at the multiplicative set generated by all $\theta_e-\tfrac{1}{2}\hbar+n\hbar$, which proves part (b).     On the other hand, if $\quiver$ does not contain any loops then there are no such pairs $(i,r) = (j,s)$ to work with, since the refined wheel conditions require an edge $e:i\rightarrow j$.  This proves (a).    
\end{proof}

\subsection{Filtrations and FMOs}

We can define a filtration on graded pieces of the big shuffle algebra $\preshu$ indexed by $\Partitions^I$ as follows. 
\begin{Definition}
    For $\bla \in \Partitions^\vertexset$, we define:
    $$
        \preshu^{\leq \bla} = \left\{ f \in \preshu_{|\bla|} : f|_{\bmu} = 0 \text{ for all } \bmu \in \Partitions^\vertexset \text{ with } |\bmu| = |\bla|, \text{ unless } \bmu \leq \bla \right\}
    $$
\end{Definition}

\begin{Remark}
    \label{rem:filtrationmulticompositions}
    Since $f$ is symmetric, this definition entails that $f|_A = 0$ for all multicompositions $A$ of $|\bla|$, unless $\bmu \leq \bla$ where $\bmu$ is the multipartition corresponding to $A$.
\end{Remark}

Similar filtrations on shuffle algebras have appeared previously, from various viewpoints, for example \cite[Definition 1.7]{FHHSY} and \cite[\S 3.3]{Negut2}.  Our present motivation is precisely to pull-back the filtrations  $\qzast(\bv)^{\leq \bla}$ on quantized zastava spaces  from Section \ref{sec: filtration}, under FFT homomorphisms:

Note that $\preshu^{\leq \bla} \subseteq \preshu_{|\bla|}$ by construction, and also that $\preshu_{\bk} = \preshu^{\leq \bk}$ since $\bk$ interpreted as a one-row multipartition is maximal for the dominance order. 
For each $\bv \in \NN^I$, it is clear that this filtration is compatible under $\Phi^+_\bv$ with the filtration on $\GKLO{\bv}^{S_{\bv}}$.
Moreover, because the $\Phi_\bv^+$ is injective in degrees less than or equal to $\bv$, the filtration on partitions of size less than or equal to $\bv$ is pulled back from the filtration on $\GKLO{\bv}^{S_{\bv}}$. Allowing $\bv$ to grow, we obtain the following.

\begin{Proposition}
    Equipped with the above filtration, $\preshu$ and $\shu$ are filtered graded rings.
\end{Proposition}

The following theorem shows that the lowest steps of this filtration exactly correspond to FMOs. 
\begin{Theorem}
\label{thm:lowest-part-of-filtration-is-fmo}
Let $\bm \in \NN^\vertexset$. Then the filtered piece $\preshu^{\leq (1^{\bm})}$ consists precisely of elements $f \in \preshu_\bm$ of the form 
\begin{equation}
  \label{eq:15}
  f = \tilde{f} \cdot \prod_{i \in \vertexset} \prod_{a,b \in [\sm_i], a \neq b} (x_{i,a} - x_{i,b} - \hbar)  
\end{equation}
for some $\tilde{f} \in \preshu_{\bm}$. All elements of the form \eqref{eq:15} lie in $\cS_{\bm}$, i.e. $\preshu^{\leq (1^{\bm})} = \cS^{\leq (1^{\bm})}$.

Finally, let $\bv \in \NN^\vertexset$ and $f \in \shu^{\leq (1^\bm)}$ be as above.  If $\bv \not\geq \bm$, then $\Phi^+_\bv(f) = 0$.  Else if $\bv \geq \bm$, observe that we may identify $\preshu_{\bm}$ with $\GT^{(1^\bm)}(\bv)_{\res}$.  Using this identification, we have:
\begin{equation}
  \label{eq:16}
  \Phi^+_{\bv}\left( f\right)  = \frac{1}{\bm!} \cdot \M_{\bm}(\tilde{f}) 
  =  \frac{1}{\bm!} \cdot \M_{\bm}\left(\frac{f}{\prod_{i \in \vertexset} \prod_{a,b \in [\sm_i], a \neq b} (w_{i,a} - w_{i,b} - \hbar)  }\right) \in \qzast(\bv)
\end{equation}
\end{Theorem}
The elements \eqref{eq:15} and the key formula \eqref{eq:16} have appeared previously in \cite{FT2,Tsymbaliuk2,JN26}.

\begin{proof}
    First, consider an element $f \in \preshu_\bm$  of the form \eqref{eq:15}.  For any multipartition $\bla$ of $\bm$ with $\bla > (1^\bm)$ or equivalently $\bla \neq (1^\bm)$, it is clear that $f|_\bla = 0$. Thus $f \in \preshu^{\leq (1^\bm)}$, and trivially satisfies the refined wheel conditions for $\bla\neq (1^\bm)$.  Since the refined wheel conditions for $\bla =(1^\bm)$ are vacuous, it also follows  that $f \in \shu_\bm$.  In particular $f \in \shu \cap \preshu^{\leq (1^\bm)} = \shu^{\leq (1^\bm)}$.

    Conversely, let $f \in \preshu_\bm$ and fix $i \in \vertexset$. Consider the multipartition $\bla = (\bla_j)_{j \in \vertexset}$ defined by $\bla_j = (1^{\sm_j})$ for $j \neq i$ and $\bla_i = (1^{\sm_i-1} 2)$.  Then $f|_\bla$ is defined by evaluating the variables $\{x_{j,a}\}_{a \in [\sm_j]}$ to $\{w_{j,r}\}_{r \in [\sm_j]}$ for $j \neq i$, while the variables $\{x_{i,a}\}_{a \in [\sm_i]}$ are evaluated to 
    $$
    \{ w_{i,1}, w_{i,1}+\hbar\} \cup \{w_{i,2},\ldots, w_{i,\sm_i-1}\}
    $$
    If $f \in \preshu^{\leq (1^\bm)}$, then $f|_\bla = 0$.  It follows from the symmetry of $f$ that it is divisible by 
    $$
    \prod_{a, b \in [\sm_i], a\neq b} (x_{i,a} - x_{i,b} - \hbar)
    $$
    This holds for all $i\in\vertexset$, and thus $f$ has the form \eqref{eq:15} as claimed. This implies by the previous argument that  $f \in \shu^{\leq (1^\bm)}$, and thus that $\preshu^{\leq (1^\bm} = \shu^{\leq (1^\bm)}$.

    Finally, take $f \in \shu^{\leq (1^\bm)}$. We have shown that $f$ has the form \eqref{eq:15}.  Since $f|_{\bla} = 0$ for all $\bla \neq (1^\bm)$, the only multicompositions $A$ which can contribute to $\Phi^+_\bv(f)$ correspond to the multipartition $\bla = (1^\bm)$, i.e.~where the multiplicities $A_i(r) \leq 1$ for all $(i,r)$. That is, in the language of multisets, $A$ is an honest \emph{subset} of $[\bv]$. If $\bv \not\geq \bm$ there are no such subsets $A$ of size $\bm$, so $\Phi^+_\bv(f) = 0$.  So suppose that $\bv \geq \bm$. Then, $\Phi^+_{\bv}(f)$ is a sum over subsets of $[\bv]$ of size $\bm$ which exactly recovers $ \frac{1}{\bm!} \cdot \M_{\bm}(\tilde{f})$ via formula \eqref{eq: FMOs under GKLO}.

\end{proof}

\subsection{Regularity}

\begin{Theorem}
  \label{thm:regularity-of-phi-bv}
  Fix $\bv \in \NN^I$. The FFT homomorphism $\Phi^+_\bv: \preshu \longrightarrow \GKLO{\bv}^{S_{\bv}}$ restricts to a surjective map of filtered graded rings
  \begin{equation}
    \label{eq:11}
   \Phi^+_\bv: \cS \twoheadlongrightarrow \qzastplus(\bv).
  \end{equation}
  Moreover, for all $\bv', \bv \in \NN^I$ with $\bv' \leq \bv$, the following diagram commutes: 
\begin{center}
\begin{tikzcd}%[row sep=tiny]
& \qzastplus(\bv) \arrow[dd,"{\Phi^+_{\bv,\bv'}}"] \\
\cS \arrow[ur,"\Phi^+_{\bv}"] \arrow[dr,"\Phi^+_{\bv'}"'] & \\
& \qzastplus(\bv')
\end{tikzcd}
\end{center}
In particular, we obtain a natural map to the limit:
\begin{equation}
  \label{eq:30}
  \Phi^+ : \cS \longrightarrow \qzastplus 
\end{equation}
\end{Theorem}

First, let us introduce the following notation.

\begin{Definition}
Fix $\bv \in \NN^I$ and $\bk \in \NN^I$, and let $f \in \preshu_{\bk}$. We define 
  \begin{equation}
    \label{eq:48}
    \supp{\Phi^+_{\bv}\left(f\right)} = \left\{ \bla \in \cP^\vertexset_{\leq \bv} \suchthat |\bla| = \bk \textand f\vert_{\bla} \neq 0   \right\}
  \end{equation}
  and:
  \begin{equation}
    \label{eq:49}
   \overline{\supp{\Phi^+_{\bv}\left(f\right)}} = \left\{ \bmu \in \cP^\vertexset_{\leq \bv} \suchthat \bmu \leq \bla \text{ for some } \bla \in \supp{\Phi^+_{\bv}\left(f\right)} \right\}
  \end{equation}
\end{Definition}

We begin with the following technical result:
\begin{Lemma}
  \label{lem:sl-2-divisibility}
  Fix $\bv \in \NN^I$ and $\bk \in \NN^I$, and let $f \in \preshu_{\bk}$. If $\bla$ is maximal in $\overline{\supp(\Phi^+_\bv(f)})$, then $f\vert_{\bla}$ is divisible by the linear polynomials
  \begin{equation}
    \label{eq:18}
    w_{i,r} - w_{i,s} + p \hbar
  \end{equation}
  for all $i \in I$, $r,s \in [\bv_i]$ with $r < s$ and all
  \begin{equation}
    \label{eq:19}
    p \in [\bla_i(r)-\bla_i(s) + 1,\bla_i(r)] \cup [-\bla_i(s), -1]
  \end{equation}
\end{Lemma}

\begin{proof}
We may identify multicompositions $A$ supported on $[\bv]$ with weights  of the representation $\bigotimes_{i \in \vertexset} S^{\bk_i}(V_i)$ of $\prod_{i \in \vertexset} \GL(V_i)$. In particular $\bla$ is a weight of this representation. 

Fix $i \in \vertexset$ and $1 \leq r < s \leq \sv_i$, and consider the subalgebra $\fsl_2 \subset \mathfrak{gl}(V_i)$ corresponding to the root $\alpha = \eps_{i,r} - \eps_{i,s}$. Consider  the $\alpha$--string of weights through $\lambda$, and in particular the elements of this string which are weights of $\bigotimes_{i \in \vertexset} S^{\bk_i}(V_i)$  but are not in $\overline{\supp(\Phi^+_\bv(f)})$.  By the representation theory of $\fsl_2$ these are precisely:
\begin{equation}
\label{eq:outside Vla}
\left\{ \lambda + q \alpha : 1 \leq q \leq \lambda_{i,s} \right\} \cup \left\{ \lambda - q \alpha : \lambda_{i,r} -\lambda_{i,s} + 1 \leq q \leq \lambda_{i,s} \right\}
\end{equation}
By the definition of $\overline{\supp(\Phi^+_\bv(f)})$ we  have $f|_{\lambda\pm q \alpha} = 0$ for all  elements $\lambda \pm q \alpha$ of this set. 

Fix $1 \leq q \leq \lambda_{i,s}$.  Introduce an auxiliary variable $u$, and consider the polynomial $g(u)$ defined by evaluating $f$ as follows: we evaluate the variables $\{x_{j,b}\}_{j \neq i, b \in [\bk_j]}$ of $f$ exactly as in the definition of $f|_\bla$, and we evaluate the variables $\{x_{i,a}\}_{a \in [\bk_i]}$ at the set
$$
\left\{w_{i,t},\ldots,w_{i,t} + (\lambda_{i,t}-1)\hbar\right\}_{t \neq s}\cup \left\{ w_{i,s}, \ldots, w_{i,s} + (\lambda_{i,s}-q-1)\hbar\right\} \cup \{u, u+\hbar,\ldots, u+(q-1)\hbar\}
$$
Now observe that  $f|_{\lambda+q\alpha} $ is obtained by evaluating $g(u)$ at $u = w_{i,r}+\lambda_{i,r} \hbar$. Since $f|_{\lambda+q\alpha} =0$, it follows that that $g(u)$ is divisible by $w_{i,r} + \lambda_{i,r} \hbar - u$.  Similarly $f|_\lambda$ comes by evaluating $g(u)$ at $u = w_{i,s} + (\lambda_{i,s}-q)\hbar$. It follows that $f|_\lambda$ is divisible by $w_{i,r} - w_{i,s} + (\lambda_{i,r} - \lambda_{i,s} +q)\hbar$.  
Taken over all $1 \leq q \leq \lambda_{i,s}$, we see that  $f|_\lambda$  is divisible by $w_{i,r}-w_{i,s} + p\hbar$ for all $\lambda_{i,r}-\lambda_{i,s}+1 \leq p \leq \lambda_{i,r}$. 

A similar proof using $f|_{\lambda - q\alpha} = 0$ for $\lambda_{i,r} -\lambda_{i,s}+1 \leq q \leq \lambda_{i,r}$ covers the case  $-\lambda_{i,s} \leq p \leq -1$.
\end{proof}

We will also need the following elementary result:
\begin{Lemma}
  \label{lem:pochammer-n-m-p-identity}
  Let $n,m \in \NN$ with $m \geq n$, and let $x$ and $y$ be indeterminates. Then
  \begin{equation}
    \label{eq:22}
    \frac{(x-y- n\hbar)^{\overline{m}} (y - x - m \hbar)^{\overline{n}}}{(x-y)^{\overline{m-n}}} = (-1)^n \prod_{p \in [m-n+1,m] \cup [-n,-1]} (x-y+p\hbar)
  \end{equation}
  
\end{Lemma}

\begin{proof}[Proof of Theorem \ref{thm:regularity-of-phi-bv}]

  Let $h \in \cS_{\bk}$. We argue by induction on $\overline{\supp}(\Phi^+_\bv(h))$ that $\Phi^+_\bv(h) \in \qzastplus(\bv)$ and that $\Phi^+_{\bv'}(h) = \Phi^+_{\bv,\bv'}(\Phi^+_\bv(h))$. Surjectivity then follows by Theorem \ref{thm:lowest-part-of-filtration-is-fmo}. The base case of the induction is when $\Phi^+_\bv(h)$ is supported on a single one-column partition, which is the content of Theorem \ref{thm:lowest-part-of-filtration-is-fmo}.

Let $\bla$ be maximal in $\overline{\supp(\Phi^+_\bv(h))}$.
By Lemma \ref{lem:sl-2-divisibility}, we can write
\begin{equation}
  \label{eq:23}
  h\vert_\bla = \tilde{h} \cdot \prod_{i \in I} \prod_{r,s \in [\bv_i], r < s}  \prod_{p \in [\bla_i(r)-\bla_i(s) + 1,\bla_i(r)] \cup [-\bla_i(s), -1]}(w_{i,r} - w_{i,s} + p)
\end{equation}
for some $\tilde{h} \in \GT^{\bla}(\bv)_{\res}$.
Applying Lemma \ref{lem:pochammer-n-m-p-identity}, we have:
\begin{equation}
  \label{eq:24}
  h\vert_\bla \cdot \frac{1}{\prod_{i\in I} \prod_{r,s \in [\bv_i], r \neq s}(w_{i,r}-w_{i,s} - \bla_i(s)\hbar)^{\overline{\bla_i(r)}}} = (-1)^{\text{sign}} \tilde{h} \cdot
  \frac{1}{ \prod_{i \in I} \prod_{r,s \in [\bv_i], r < s} (w_{i,r} - w_{i,s} )^{\overline{\bla_i(r) - \bla_i(s)} } } 
\end{equation}

Now we apply the refined wheel conditions. Fix an $e : i \rightarrow j$, and $r \in [\bv_i]$ and $s \in [\bv_j]$. Our goal is to divide $\tilde{h}$ by all the denominators indexed by $e$,$r$, and $s$ appearing in \eqref{eq:FFT-homomorphism} and obtain exactly the numerators indexed by $e$,$r$, and $s$ in \eqref{eq:20}.

Consider the case where $\bla_i(r) \geq 1$ and $\bla_i(s) \geq 1$. The first subcase is when $\bla_i(r) \leq \bla_j(s)$. In this case, there is no numerator factor in \eqref{eq:20}, and the refined wheel condition (ii) tells us that $\tilde{h}$ is divisible by $
(w_{j,s} - w_{i,r} + \theta_e + \frac{1}{2}\hbar)^{\underline{A_i(r)} \cdot \overline{A_j(s) - 1}}$, which is exactly the denominator factor appearing in \eqref{eq:FFT-homomorphism}.

The second subcase is when $\bla_i(r) > \bla_j(s)$. Then the refined wheel condition (i), together with part (b) of Lemma \ref{lem:doublePochhdiv}, tell us that $\tilde{h}$ is divisible by $(w_{j,s} - w_{i,r} + \theta_e + \frac{1}{2}\hbar)^{\underline{A_i(r)} \cdot \overline{A_j(s) - 1}}$ times
\begin{equation}
  \label{eq:25}
  (w_{j,s} - w_{i,r} + \theta_e - \tfrac{1}{2} \hbar)^{\underline{\bla_{i}(r) - \bla_{j}(s)}}
\end{equation}
which is exactly the numerator factor indexed by $e$,$r$, and $s$ in \eqref{eq:20}.

The final case is when $\bla_j(s) = 0$ (the case $\bla_i(r) = 0$ is vacuous). In this case, the refined wheel conditions are vacuous, and the factor indexed by $e$, $r$, and $s$ in \eqref{eq:FFT-homomorphism} appears in the numerator as
\begin{equation}
  \label{eq:26}
(w_{j,s} - w_{i,r} + \theta_e - \tfrac{1}{2} \hbar)^{\underline{\bla_{i}(r) }}  
\end{equation}
which is exactly corresponding numerator factor in \eqref{eq:20}.

Summarizing, we see that by dividing $\tilde{h}$ by all of these factors (and multiplying by $(-1)^{\mathrm{sign}}$), we can find $f \in \GT^{\bla}(\bv)_{\res}$ such that:
\begin{align*}
  \tilde{h} \frac{1}{\prod_{e \in E} \prod_{\substack{r \in [\bv_{s(e)}] \\ s \in [\bv_{t(e)}]}} (w_{t(e),s} - w_{s(e),r} + \theta_e + \frac{1}{2} \hbar)^{\underline{\bla_{s(e)}(r)}\cdot \overline{\bla_{t(e)}(s)-1}}}  = \\ f \cdot  \prod_{e \in E} \prod_{r \in [\bv_{s(e)}],s \in [\bv_{t(e)}], \bla_{s(e)}(r) > \bla_{t(e)}(s)} (w_{t(e),s} - w_{s(e),r} + \theta_e - \tfrac{1}{2} \hbar)^{\underline{\bla_{s(e)}(r) - \bla_{t(e)}(s)}} 
\end{align*}

Finally, we can find $f_1^{(\eta)} \in \GT^{(1^{\bc_1})}(\bv)_{\res}, \ldots, f_\ell^{(\eta)} \in \GT^{(1^{\bc_\ell})}(\bv)_{\res}$, where $\bc_1, \ldots, \bc_\ell$ are the columns of $\bla$ and with $\eta$ lying in a finite set, so that
\begin{equation}
  \label{eq:42}
f = \sum_{\eta} f_1^{(\eta)}(w) f^{(\eta)}_2(w + \hbar \bmu_1) \cdots f_\ell^{(\eta)}(w + \hbar (\bmu_1 + \cdots + \bmu_{\ell-1}) )
\end{equation}
We therefore conclude that $\Phi^+_\bv(h)$ has the same $u^\bla$-coefficent as:
\begin{equation}
  \label{eq:43}
\sum_{\eta} \frac{1}{\bla!} M_{\bm_1}(f^{(\eta)}_1) \cdots M_{\bm_\ell}(f^{(\eta)}_\ell)
\end{equation}
By Theorem \ref{thm:lowest-part-of-filtration-is-fmo}, we can find
$g_1^{(\eta)} \in \cS^{\leq 1^{\bc_1}}, \ldots, g_\ell^{(\eta)} \in \cS^{\leq 1^{\bc_\ell}}$
with $\eta$ varying over the same finite set as before so that setting
\begin{equation}
  \label{eq:44}
  g = \sum_{\eta} g_1^{(\eta)} \cdots g_{\ell}^{(\eta)}
\end{equation}
we have
\begin{equation}
  \label{eq:45}
  \Phi^+_\bv(g) = \sum_{\eta} \frac{1}{\bla!} M_{\bm_1}(f^{(\eta)}_1) \cdots M_{\bm_\ell}(f^{(\eta)}_\ell)
\end{equation}
In particular, by Theorem \ref{thm:lowest-part-of-filtration-is-fmo}, the Theorem holds for $g$. Finally, since $\Phi^+_\bv(h-g)$ has strictly smaller support than $\Phi^+_\bv(h)$, we have our result by induction. 

\end{proof}

\subsection{lsomorphism with $\qzastplus$}

Our goal in this section is to prove that the map $\Phi^+: \shu \rightarrow \qzastplus$ is an isomorphism.
\begin{Proposition}
  \label{prop:shuf-bk-to-qzastplus-bk-isom}
For all $\bk \in \NN^I$, the map:
\begin{equation}
  \label{eq:9}
  \Phi^+_{\bk} : \cS_{\bk} \rightarrow \qzastplus(\bk)_{\bk} 
\end{equation}
 is an isomorphism. 
\end{Proposition}

\begin{proof}
  By Theorem \ref{thm:regularity-of-phi-bv}, this map is surjective. It is injective because the composed map
  \begin{equation}
    \label{eq:46}
  \Phi^+_{\bk} : \cS_{\bk} \rightarrow \qzastplus(\bk)_{\bk} \hookrightarrow \left(\GKLO{\bk}^{S_\bk}\right)_{\bk}
  \end{equation}
  is injective. 
\end{proof}

Combining Theorem \ref{thm:regularity-of-phi-bv}, part (8) of Theorem \ref{thm:limzastproperties}, and Proposition \ref{prop:shuf-bk-to-qzastplus-bk-isom}, we obtain the following, which is the main theorem of this section.

\begin{Theorem}
\label{thm:isoshuaplus}
The map
$$
  \Phi^+ : \cS \rightarrow \qzastplus 
$$
is an isomorphism of $\ZZ^\vertexset\times \ZZ$-graded $\base$-algebras, and respects filtrations by $\Partitions^\vertexset$.
\end{Theorem}

Recall from part (2) of Theorem \ref{thm:limzastproperties}, that the algebra $\qzastplus$ is generated by FMOs $\M_\bm(f)$ with restricted dressings $f\in \GT^\bmu_{\res}$, where $\bmu = (1^\bm)$ and $\bm \in \NN^\vertexset$.  Via Theorem \ref{thm:lowest-part-of-filtration-is-fmo}, we have seen that these FMOs are identified with certain explicit  shuffle algebra elements. We deduce the following:

\begin{Corollary}
    The shuffle algebra $\shu$ is generated by its elements of the form 
    $$
        f = \tilde{f} \cdot \prod_{i \in \vertexset} \prod_{a,b \in [\sm_i], a \neq b} (x_{i,a} - x_{i,b} - \hbar) \ \in \ \shu_\bm
    $$
where $\tilde{f} \in \preshu_{\bm}$ and $\bm \in \NN^\vertexset$.
\end{Corollary}

\subsubsection{Integral forms} 
\label{ssec:integralform}
Throughout \S \ref{sec: shuffle} we have required that the base ring $\kk$ has $\QQ \subseteq\kk$, because reciprocals of factorials appear throughout. One way to avoid this is to define an integral form of the shuffle algebra, via explicit generators.

First, let us take $\kk = \QQ$ and define the corresponding shuffle algebra $\shu_\QQ$ as in Definition \ref{def:shu}.  By Theorem \ref{thm:isoshuaplus} we know that $\shu_\QQ \cong \qzastplus_\QQ$ is isomorphic to the quantized limit zastava over $\QQ$.  
Next, consider the subalgebra $\shu_\ZZ \subset \shu_\QQ$ generated by the elements of the form
\begin{equation}
    \label{eq:intshugens}
    f = \bm! \cdot \tilde{f} \cdot\prod_{i \in \vertexset} \prod_{a,b \in [\sm_i], a \neq b} (x_{i,a} - x_{i,b} - \hbar),
\end{equation}
over all  $\bm \in \NN^\vertexset$ and \emph{integral} elements $\tilde{f} \in \preshu_{\bm}$, i.e.~having integral coefficients. By Theorem \ref{thm:lowest-part-of-filtration-is-fmo}, we see that for $f$ as in \eqref{eq:intshugens} we have $\Phi^+_\bv(f) = \M_\bm(\tilde{f})$.  These FMOs are generators for $\qzastplus_\ZZ(\bv)$. By passing to the limit we deduce that $\shu_\ZZ \cong \qzastplus_\ZZ $, and the following diagram:
\begin{equation}
    \begin{tikzcd}
        \shu_\QQ \ar[r,"\Phi^+"',"\sim"] & \qzastplus_\QQ \\
        \shu_\ZZ \ar[u,hook] \ar[r,"\sim"]  &  \qzastplus_\ZZ \ar[u,hook]  
    \end{tikzcd}
\end{equation}
\begin{Remark}
    Although this provides a definition of $\shu_\ZZ$ (and thus $\shu_\kk$ for any $\kk$, by base change), we do not have an explicit criterion to detect whether an element of $\shu_\QQ$ lies in $\shu_\ZZ$.
\end{Remark}

\subsubsection{Proof  of Theorem \ref{thm:grpositivepart}}
\label{ssec:proofgrAplus}
We are now ready to complete the proof of Theorem \ref{thm:grpositivepart} by showing that $\gr \qzastplus(\bv) \subseteq \bigoplus_{\bla \in \Partitions_{\leq \bv}^\vertexset} \GT_{\res}^\bla(\bv) \cdot \sfr_\bla$. In other words, we must compute the leading term of elements of $\qzastplus(\bv)^{\leq \bla}$.  By linearity, it suffices to do so when $\kk = \ZZ$.  Since $\qzastplus_\ZZ(\bv) \subseteq \qzastplus_\QQ(\bv)$, it further suffices to consider the case $\kk = \QQ$. 

So, take $\kk = \QQ$. As we have seen above, the FFT map defines a surjection $\Phi_\bv^+:\shu \twoheadrightarrow \qzastplus(\bv)$. Given $h \in \shu_\bk$ with $\Phi_\bv^+(h) \in \qzastplus(\bv)^{\leq\bla}$, the element   $\bla \in\overline{\supp(\Phi^+_\bv(h))}$ is maximal (or else the coefficient of $\sfu^\bla$ is zero). The proof of Theorem \ref{thm:regularity-of-phi-bv} constructed a certain element $f \in \GT^\bla_\res(\bv)$,  and the calculations there show that
$$
\Phi_\bv^+(h) + \qzastplus(\bv)^{<\bla} = f \cdot r_\bla 
$$
in $\gr \qzastplus(\bv)$.  This proves that $\gr \qzastplus(\bv) \subseteq \bigoplus_{\bla \in \Partitions_{\leq \bv}^\vertexset} \GT_{\res}^\bla(\bv) \cdot \sfr_\bla$ for $\kk = \QQ$, as required.

\subsection{The extended shuffle algebra and $\qzast$}
\label{ssec:extshu}

Consider a polynomial ring over the base $\base$, with countably many generators:
\begin{equation}
    \shuzero = \base[\Qvar_i^{(r)} : i \in \vertexset, r \geq 1]
\end{equation}
To match shuffle algebra notations, we will sometimes denote the multiplication in this ring by $\ast$.  It is also  convenient to denote $Q_i^{(0)} = 1$, and to define the formal series 
\begin{equation}
    \label{eq:Qdef}
    Q_i(u) = \sum_{r \geq 0}{Q_i^{(r)} u^{-r}} \in \shuzero[\![u^{-1}]\!]
\end{equation}

Define the \textbf{extended shuffle algebra} to be the following enlargement of $\shu$:
\begin{equation}
    \extshu = \shuzero \otimes_\base \shu
\end{equation}
We will identify  $\shuzero \subset \extshu$ via $x \mapsto  x\otimes 1$, and similarly $\shu \subset \extshu $ via $f \mapsto 1 \otimes f$. Define a  multiplication $\star$ on $\extshu$, such that $\shuzero$ and $ \shu$ are  subrings with $ x\star f = x \otimes f$ for $x \in \shuzero$ and $f \in \shu$, and via the following relation: for each degree $\bk \in \NN^\vertexset$ and all $f \in \shu_\bk$ we have
\begin{equation}
    \label{eq:actionofQ}
    f \star  \Qvar_i(u) =   \Qvar_i(u)  \star \Big(  f \prod_{a = 1}^{\bk_i} \frac{u-x_{i,a}-\hbar}{u-x_{i,a}} \Big)
\end{equation}
More precisely, here both sides are expanded as  power series in $u^{-1}$, and we demand that their coefficients are equal.   Importantly, note that the coefficients of $f \prod_a \frac{u-x_{i,a}-\hbar}{u-x_{i,a}}$ are indeed elements of $\shu_\bk$, since  $\shu_\bk \subseteq \preshu_\bk$ is defined by the refined wheel conditions and thus an \emph{ideal}.

There is a  $\ZZ^\vertexset\times \ZZ$-grading on $\shuzero$ defined by setting $\deg \Qvar_i^{(r)} = (\mathbf{0}, r)$.  Then $\extshu$ is also naturally $\ZZ^\vertexset\times \ZZ$-graded, via the  tensor product of the gradings on $\shuzero$ and $\shu$.
\begin{Lemma}
    The multiplication $\star $ makes $\extshu$ into a $\ZZ^\vertexset\times\ZZ$-graded ring.
\end{Lemma}

\begin{Remark}
    Said differently, $\extshu = \shuzero \# \shu$ is a smash product.  Namely, $\shuzero$ is a Hopf algebra over $\base$ with comultiplication defined by $\Delta\big( \Qvar_i(u) \big) = \Qvar_i(u) \otimes \Qvar_i(u)$, and there is an action of $\shuzero$ on $\shu$ given by $ \Qvar_i(u) \cdot f = f \prod_{a=1}^{\bk_i} \frac{u-x_{i,a} - \hbar}{u-x_{i,a}}$ for $f \in \shu_\bk$. This action makes $\shu$ into an $\shuzero$-module algebra, and thus we may define the smash product $\shuzero \# \shu$, corresponding to the following extension of \eqref{eq:actionofQ}: for any $x \in \shuzero$ and $f\in \shu_\bk$ we have
    $$
    f \star x = \sum x_{(1)} \star (x_{(2)} \cdot f),
    $$
    where we have used Sweedler notation $\Delta(x) = \sum x_{(1)} \otimes x_{(2)}$.
\end{Remark}

The FFT homomorphisms naturally extend to the algebra $\extshu$, as was similarly observed  in the trigonometric setting by Jindal-Negut \cite[\S 4.2]{JN26}.  The following is the main result of this subsection:
\begin{Theorem}
    \label{thm:isoextshu}
    Let $\bv \in \NN^\vertexset$.  Then the FFT homomorphism $\Phi^+_\bv: \shu \rightarrow \qzastplus(\bv)$ extends to a surjective map $\Phi_\bv: \extshu \rightarrow \qzast(\bv)$, such that 
    $$
    \Qvar_i(u) \longmapsto \prod_{r=1}^{\sv_i} ( 1 -w_{i,r} u^{-1}) 
    $$
    These maps are compatible with the directed system of $\ZZ^\vertexset\times\ZZ$-graded rings $\qzast(\bv)$, and the induced map in the limit is an isomorphism:
    $$\Phi : \extshu \xrightarrow{\ \ \sim \ \ } \qzast$$
\end{Theorem}
\begin{proof}
    To prove that the extension $\Phi_\bv: \extshu \rightarrow \qzast(\bv)$ is well-defined, it suffices to verify compatibility with the relation \eqref{eq:actionofQ}. For any multicomposition $A$ of $\bk$, we have:
    $$
    \Big(f\prod_{a=1}^{\bk_i} \frac{u-x_{i,a} - \hbar}{u-x_{i,a}} \Big)\Bigg|_A = f|_A \prod_{r = 1}^{\sv_i} \prod_{p=0}^{A_i(r)-1} \frac{u-w_{i,a}-(p+1)\hbar}{u-w_{i,a}-p\hbar} = f|_A \prod_{r=1}^{\sv_i}\frac{u-w_{i,r} - A_i(r) \hbar}{u-w_{i,r}}
    $$
    Noting  that $\sfu^A \cdot \prod_{r=1}^{\sv_i}(u-w_{i,r})  = \prod_{r=1}^{\sv_i}(u-w_{i,r} - A_i(r)\hbar)$, it is straightforward to verify \eqref{eq:actionofQ} and thus the map $\Phi_\bv$ is well-defined.

    Next, the map $\Phi_\bv : \extshu \rightarrow \qzast(\bv)$ is surjective because of Lemma \ref{lem: truncated triangular decomp}, since $\shuzero \twoheadrightarrow \GT(\bv)$ and $\shu \twoheadrightarrow \qzastplus(\bv)$ are both surjective.    Moreover, the maps $\Phi_\bv: \extshu \rightarrow \qzast(\bv)$ are compatible with the directed system of graded rings $\qzast(\bv)$: this is true when restricted to $\shu$ by Theorem \ref{thm:regularity-of-phi-bv}, and clearly holds for the new generators $\Qvar_i^{(r)} \in \extshu$.  

    In the limit, we thus obtain a map $\Phi: \extshu \rightarrow \qzast$.  It restricts to the previous isomorphism $\Phi^+:\shu \rightarrow \qzastplus$, and also is seen to restrict to an isomorphism $\Phi^0:\shuzero \rightarrow \GT$.  Overall, the map $\Phi$ is nothing but the tensor product of these two isomorphisms, proving the final claim:
    $$
    \begin{tikzcd}[column sep = 1.5cm]
        \extshu = \shuzero \otimes_\base \shu \ar[r,"\Phi^0 \otimes \Phi^+","\sim"'] & \GT \otimes_\base \qzastplus = \qzast
    \end{tikzcd}
    $$

\end{proof}

%%%%%%%%%%%%%%%%%%%%%%%%%%%%%%%%%%%%%%%%%%%%%%%%%%%%%%%%%%%%%%%%%
\section{Spherical generation and Negu{\c t}'s conjecture}
\label{sec:spherical}
Throughout this section we will continue to assume that $\QQ \subseteq \kk$. 
We will say that an $\NN^\vertexset$-graded algebra $A$ is \emph{spherically generated} if it is generated by its homogeneous components $A_{\mathbf{0}}$ and $A_{\alpha_i}$ for $i \in \vertexset$, where $\{\alpha_i\}_{i \in \vertexset}$ denotes the standard basis for $\NN^\vertexset$.  
If $A$ is an algebra over the base ring $\base $ from \eqref{eq: base ring}, we will denote  $A_\loc = A\otimes_\base \operatorname{Frac}\base$, where $\operatorname{Frac} \base$ is the fraction ring of $\base$.  Since $\base$ is trivially $\NN^\vertexset$-graded, the algebra $A_\loc$ is again $\NN^\vertexset$-graded.

The goal of this section is to show how the results we have obtained so far prove the following Theorem, which was conjectured by Negu{\c t} in \cite{Negut-2025}. 

\begin{Theorem}[\mbox{\cite[Conjecture 2.12]{Negut-2025}}]
  \label{thm:negut-conjecture}
The algebra $\cS'_\loc$ is spherically generated.
\end{Theorem}

By Theorem \ref{thm:shufflealgcomparison}, it suffices to show that $\cS_\loc$ is spherically generated. By Theorem, \ref{thm:isoshuaplus} it further suffices to show that $\qzastplus_\loc$ is spherically generated. 

\begin{Remark}
    In fact, we will prove a finer statement: rather than localizing at the full fraction field $\operatorname{Frac} \base$, it suffices to localize at the multiplicative set $M$ from Definition \ref{def:multiplicativeset} below. 
    
    In particular,  if the underlying graph of our quiver $\quiver$ is a tree then $\shu[M^{-1]}] = \shu[\hbar^{-1}]$, see part (3) of Remark \ref{rmk:multiplicativeset}. Thus $\shu[\hbar^{-1}] = \shu'[\hbar^{-1}]$ is spherically generated in this case. 
\end{Remark}

\subsection{Spherical generation for $\qzast(\bv)$}

Fix $\bv \in \NN^\vertexset$. We will define elements $Q_i^{(r)} \in \GT(\bv)$ and $ F_i^{(r)} \in \qzastplus(\bv)$ for $i\in \vertexset$, $r\geq0$ as coefficients of the following formal series in $\qzast(\bv)[\![z^{-1}]\!]$:
\begin{equation}
    Q_i(z)= \sum_{r\geq 0} Q_i^{(r)} z^{-r} =  \prod_{s=1}^{\sv_i} (1-z^{-1} w_{i,s}), \qquad   \   F_i(z) = \sum_{r \geq 1} F_i^{(r)} z^{-r} = \M_{e_i}\Big(\frac{1}{z-w_{i,1}}\Big)
\end{equation}
Note that $Q_i^{(0)} = 1$ and $F_i^{(0)} = 0$, and that $Q_i^{(r)}$ is precisely the image of the same-named element of $\extshu$ under the map $\Phi_\bv:\extshu\twoheadrightarrow\qzast(\bv)$ from \S \ref{ssec:extshu}.  Using Theorem \ref{thm: grading}, we see that  these elements are homogeneous for the $\ZZ^\vertexset\times\ZZ$-grading on $\qzast(\bv)$, with  degrees 
    \begin{equation}
        \label{eq:QandFdegrees}
        \deg Q_i^{(r)} = (\mathbf{0}, r) \quad \text{ and } \quad \deg F_i^{(r)} = (\alpha_i, r-a_{ii})
    \end{equation}
    where $a_{ii}$ is the number of arrows $e:i \rightarrow i$ in our quiver.  

We note the following useful relation, written in series form, which follows from similar  calculations to those in \cite[\S B(iii)]{BFN1} and \cite[\S C(iv)]{FT}:
\begin{Lemma}
    \label{lem:QFrels1}
    For any quiver and any $i,j \in \vertexset$, the following relation holds in $\qzast(\bv)$:
    $$
    (u-v)[Q_i(u), F_j(v)] = \delta_{ij} \hbar \big( F_j(v) - F_j(u) \big) Q_i(u)
    $$
\end{Lemma}

Before proceeding, we introduce some notation. Consider the doubled quiver $\overline{\quiver} = (I, \arrowset\sqcup \arrowset^{op})$;  for each arrow $e \in \arrowset$ we add an opposite arrow $\overline{e} \in \arrowset^{op}$ with source and target swapped.  Given $h \in \arrowset\sqcup \arrowset^{op}$, let 
\begin{equation}
\theta_h =  \left\{ \begin{array}{cl} \theta_e, & \text{if } h = e \in \arrowset, \\ - \theta_e, & \text{if } h = \overline{e} \in \arrowset^{op} \end{array} \right. 
\end{equation}
Consider a closed path $p = (h_1,\ldots,h_\ell)$ in $\overline{Q}$, that is: a sequence of arrows $h_1,h_2,\ldots,h_\ell \in \arrowset\sqcup \arrowset^{op}$ such that $\target(h_n) = \source(h_{n+1})$ for $1 \leq n < \ell$ and $\target(h_\ell) = \source(h_1)$.  We define a corresponding (linear) element of  $\base$ as follows:
\begin{equation}
    \label{eq:elementsfrompaths}
    y_p = \theta_{h_1} + \ldots + \theta_{h_\ell} - \frac{\ell}{2} \hbar
\end{equation}
By considering all possible closed paths, we obtain a subset of $\base$:
\begin{equation}
    \label{eq:multiplicativeset}
    \left\{ y_p \ : \ p \text{ is a closed path in } \overline{\quiver} \right\}
\end{equation}
\begin{Definition}
    \label{def:multiplicativeset}
    $M \subset \base$ is defined as the multiplicative set generated by the subset \eqref{eq:multiplicativeset}.
\end{Definition}

\begin{Remark}
    \label{rmk:multiplicativeset}
    \begin{enumerate}
        \item If the quiver contains no arrows (i.e.~$\arrowset = \emptyset$), then the set $M = \{1\}$ is trivial.

        \item Suppose that $\arrowset\neq \emptyset$, and let $e \in \arrowset$ be any element. Then $p = (e, \overline{e})$ is a closed path with corresponding element $y_p= - \hbar$.  Thus, ignoring signs, we have $\{\hbar^n : n \geq 0\} \subseteq M$.
        
        \item Suppose that the underlying graph of $\quiver$ is a tree, which is not of type $A_1$. Then for any closed path $p$ in $\overline{\quiver}$, arrows must appear in pairs consisting of an arrow $e\in \arrowset$ and its opposite $\overline{e}$: every step along $p$ must be reversed in order to have a closed path.  It follows that $y_p = - \frac{\ell}{2} \hbar$ is a non-zero multiple of $\hbar$.  Thus, ignoring units in $\QQ\subseteq \base$, we have $$\{ \hbar^n : n \geq 0\} = M $$

        \item For any quiver $\quiver$, the multiplicative set $M$ contains the multiplicative set \eqref{eq:57}, up to signs. Indeed, for any $e \in \arrowset$ with $\source(e) = \target(e)$, we may consider the closed path $p = (\overline{e}, e,\ldots,  \overline{e}, e,\overline{e})$ of length $2n+1$.  Then $y_p = -\theta_e -(n+\tfrac{1}{2})\hbar$ is in $M$.
    \end{enumerate}
\end{Remark}

The next theorem is a variation  on a result of the second author.

\begin{Theorem}
    \label{thm:sphericalgenv}
    For each $\bv \in \NN^\vertexset$, the algebra $\qzast(\bv)[M^{-1}]$ is spherically generated with respect to its $\NN^\vertexset$-grading.  More precisely, it is generated over $\base[M^{-1}]$ by the elements $Q_i^{(r)}, F_i^{(r)}$ for $i\in \vertexset$ and $r\geq 1$.
    \smallskip
    
    Moreover, for each $\bm \in \NN^\vertexset$ there exists an element $h_\bm \in M$, independent of $\bv$, such that 
    $$
    h_\bm \cdot \M_\bm(f)  = \M_\bm(h_\bm f)
    $$
    is inside the spherical subalgebra of $\qzast(\bv)$ for all $\bv \geq \bm$ and all $f \in \GT^\bmu(\bv)$, where $\bmu = (1^\bm)$. That is, $\M_\bm(h_\bm f)$ is expressible over $\base$ in terms of the elements $Q_i^{(r)}, F_i^{(r)}$.
\end{Theorem}
\begin{proof}
    The first part is a generalization of \cite[Theorem 3.7]{Weekes} to incorporate arbitrary quivers, as well as  additional flavour symmetry (i.e.~the parameters $\theta_e$ for $e \in \arrowset$). The only ingredient in the proof that must be modified is \cite[Proposition 3.5]{Weekes}. We now explain the modifications,  following the notation of \emph{loc.cit.}:
    
    First note that, by linearity, it suffices to take $\kk = \QQ$.  In our setting (cf.~equation \eqref{eq: rels in gr} above), the key relation \cite[equation (21)]{Weekes} becomes:
    \begin{equation}
        \label{eq:keyspheq}
        \sfr_{\eps_{i,a}} \sfr_{\lambda-\eps_{i,a}} = \Bigg( \prod_{e:i\rightarrow j}\prod_{\substack{1 \leq b \leq \sv_j \\ \lambda_{i,a} \leq \lambda_{j,b}}} (w_{j,b} - w_{i,a} + \theta_e -\tfrac{1}{2}\hbar) \prod_{e: j \rightarrow i} \prod_{\substack{1 \leq b \leq \sv_j \\ \lambda_{i,a} \leq \lambda_{j,b}}}(w_{i,a} -w_{j,b} + \theta_e +\tfrac{1}{2}\hbar\Bigg) \sfr_\lambda
    \end{equation}
    Here $\lambda = (\lambda_{i,a})_{i \in \vertexset, 1 \leq a \leq \bv_i}$ is a tuple of non-negative integers:  a non-negative coweight of $\bG$, not necessarily dominant. Consider the set $S$ of pairs $(i,a)$ such that $\lambda_{i,a}$ is maximal. For $(i,a) \in S$, the polynomial on the right side of \eqref{eq:keyspheq} simplifies to
    \begin{equation}
        \label{eq:keyspheq2}
        \prod_{e : i \rightarrow j} \prod_{(j,b) \in S} (w_{j,b} - w_{i,a} + \theta_e - \tfrac{1}{2}\hbar) \prod_{e : j \rightarrow i} \prod_{(j,b) \in S} (w_{i,a} - w_{j,b} + \theta_e + \tfrac{1}{2}\hbar)
    \end{equation}
    Taken over all $(i,a) \in S$,  these polynomials define an ideal $\mathcal{I}$ inside the ring 
    \begin{equation}
        \label{eq:keyspheq3}
        \QQ[\hbar, \theta_e, w_{i,a} : e \in \arrowset \text{ and } (i,a) \in S]        
    \end{equation}
    We claim that, after localizing at the multiplicative set\footnote{More precisely, localization at those elements  of $M$ which lie in the subring \eqref{eq:keyspheq3}.} $M$, the vanishing locus $\mathbb{V}(\mathcal{I})$ over $\CC$ is  empty. By the Nullstellensatz it follows that $\mathcal{I}$ is the unit ideal. By induction, any $\sfr_\lambda$ can be expressed in terms of the $\sfr_{i,a}$ with coefficients in $\QQ[\hbar, \theta_e, w_{i,a}][M^{-1}]$. This is a direct generalization of \cite[Proposition 3.5]{Weekes}, and the proof of \cite[Theorem 3.7]{Weekes} now goes through to prove that $\qzast(\bv)[M^{-1}]$ is spherically generated.

    To prove the above claim, consider a point $(\hbar, \theta_e, w_{i,a}) \in \mathbb{V}(\mathcal{I})$. From this point we construct a decorated quiver with vertex set $S$.   Arrows in this quiver are decorated by $\arrowset\sqcup \arrowset^{op}$, as follows: there is an arrow $(i,a) \xrightarrow{e} (j,b)$ if  $w_{i,a} = w_{j,b} + \theta_e -\tfrac{1}{2}\hbar $ where $e \in \arrowset$ with $e: i\rightarrow j$, and there is an arrow $(i,a) \xrightarrow{\overline{e}} (j,b)$ if $w_{i,a} = w_{j,b} -\theta_e - \tfrac{1}{2}$ where $e\in \arrowset$ with $e:j\rightarrow i$. By the definition of $\mathcal{I}$, there is \emph{at least one}  outgoing arrow from every vertex $(i,a) \in S$ of this quiver. It follows that there is at least one oriented cycle in the quiver. Starting from any vertex $(i,a)$ on this cycle, we obtain a closed path $p$ in our doubled quiver $\overline{\quiver}$.  Substituting the above equations into one another as we move along the cycle, we can eliminate each subsequent variable $w_{j,b}$ with $(j,b) \neq (i,a)$. After returning to our starting vertex, we obtain:
    $$
        w_{i,a} = w_{i,a} + y_p
    $$
    where $y_p$ is defined as in \eqref{eq:elementsfrompaths}. This equation contradicts the fact that we have localized at $M$, and thus $y_p \neq 0$ on $\mathbb{V}(\mathcal I)$. Therefore $\mathbb{V}(\mathcal I)$ is empty, as claimed.

    Finally, we must explain the existence of the elements $h_\bm$ from the second part of the statement of the theorem.  Given $\bm \in \NN^\vertexset$, the corresponding coweight is $\la = (1^\bm)$, and thus the set ${S = \{(i,a) : i \in \vertexset, 1 \leq a \leq \sm_i\}}$.  Crucially, this set is independent of $\bv$, as is the ring from \eqref{eq:keyspheq3}.  The same is true at all steps in the inductive process above which expresses $\sfr_\lambda$ in terms of the elements $\sfr_{i,a}$. Therefore there exists some $h_\bm \in M$, independent of $\bv$, such that $h_\bm \sfr_\lambda$ can be expressed integrally in terms of $\sfr_{i,a}$ and the ring \eqref{eq:keyspheq3}. Finally, following the strategy of \cite[Theorem 3.7]{Weekes}, one sees that $\M_\bm(h_\bm f)$ lies in the spherical subalgebra of $\qzast(\bv)$, completing the proof.
\end{proof}

\begin{Remark}
    For a given $\qzast(\bv)$ it is sufficient to localize at some finitely-generated multiplicative subset of $M$, dependent on $\bv$. Indeed, the set $S$ in the proof has size $\leq\sum_i \sv_i$. Thus we need only consider closed paths $p$ in $\overline{Q}$ of length $\ell \leq \sum_i \sv_i$, and to invert the corresponding elements $y_p \in \base$.
    
\end{Remark}

\subsection{Passing to the limit}
The elements $Q_i^{(r)}, F_i^{(r)} \in \qzast(\bv)$ from the previous section naturally pass to the limit, yielding elements $Q_i^{(r)} \in \GT$ and $F_i^{(r)} \in \qzastplus$.  Recall also that $\base \subset \qzast$ is a central subalgebra, and thus we may localize $\qzast$ at the multiplicative set $M\subset \base$ from Definition \ref{def:multiplicativeset}.
\begin{Theorem}
    \label{thm:sphericalgen}
    The algebra $\qzast[M^{-1}]$  is spherically generated with respect to its $\NN^\vertexset$-grading. More precisely, it is generated over $\base[M^{-1}]$ by the elements $Q_i^{(r)}$ and $F_i^{(r)}$ for $i \in \vertexset$ and $r\geq 1$.
\end{Theorem}
\begin{proof}
    Since $\qzast$ is generated by FMOs, it suffices  to show that every $\M_\bm(f) \in \qzast$ is in the spherical subalgebra of $\qzast[M^{-1}]$.  So let us fix $\bm \in \NN^\vertexset$ and a homogeneous element $f \in \GT^\bmu$ where $\bmu = (1^\bm)$.  
    
    By the previous theorem there exists an element $h_\bm \in M$ such that the image of $h_\bm\M_\bm(f) =\M_\bm(h_\bm f)$ is in the spherical subalgebra of  $\qzast(\bv)$, for all $\bv \geq \bm$. Since $h_\bm$ is a product of linear polynomials and $f$ is homogeneous,  the element $\M_\bm(h_\bm f)$ is homogeneous of some degree $(\bm, d) \in \ZZ^\vertexset\times \ZZ$. Any expression for $\M_\bm(h_\bm f)$ over $\base[M^{-1}]$ in terms of the elements $Q_i^{(r)}$ and $F_i^{(r)}$ can therefore be interpreted as equality in $\qzast_{(\bk, d)}$.
    By the stabilization from part (7) of Theorem \ref{thm:limzastproperties}, we may choose $\bv$ sufficiently large that $\qzast_{(\bk, d)} \xrightarrow{\sim} \qzast(\bv)_{(\bk, d)}$ is an isomorphism.  This reduces us to the situation of Theorem \ref{thm:sphericalgenv}, which concludes the proof.

\end{proof}

Finally, to prove Theorem \ref{thm:negut-conjecture}, it suffices to prove the following.

\begin{Theorem}
The algebra $\qzastplus[M^{-1}]$ is spherically generated with respect to its $\NN^\vertexset$-grading.  More precisely, it is generated over $\base[M^{-1}]$ by the  elements  $F_i^{(r)}$ for $i \in \vertexset$ and $r \geq 1$.
\end{Theorem}

\begin{proof}
  Recall that $\qzast = \GT\otimes_\base \qzastplus$, and thus $\qzast[M^{-1}] = \GT[M^{-1}] \otimes_{\base[M^{-1}]} \qzastplus[M^{-1}]$. We have seen that $\qzast[M^{-1}]$ is generated over $\base[M^{-1}]$ by elements $Q_i^{(r)}$ and $F_i^{(r)}$.
  
  We also have the relation from Lemma \eqref{lem:QFrels1}. From this relation, we see that any word in the $Q_i^{(r)}$ and $F_j^{(s)}$ can be rewritten as a linear combination of words where all $Q$'s appear to the left of all $F$'s. We therefore can apply the following lemma to complete the proof.
\end{proof}

\begin{Lemma}
  Let $A$ be an associative algebra with subalgebras $A^0$ and $A^>$ such that
  \begin{enumerate}[(i)]
  \item $A = A^0 \otimes A^>$
  \item $A^0$ is a polynomial ring in generators $\{q_i\}_{i \in S}$ 
  \item there are elements $\{f_j\}_{j \in T}$ for $f_j \in A^0$ such that $\{q_i\}_{i \in S} \cup \{f_j\}_{j \in T}$ generate $A$
  \item every word in $\{f_j\}_{j \in T}$ and $\{q_i\}_{i \in S}$ can be rewritten as a linear combination of words where the elements $q_i$ appear to the left of the $f_j$ for $i \in S, j\in T$.
  \end{enumerate}
 Then $\{f_j\}_{j \in T}$ generate $A^>$.
\end{Lemma}

\begin{proof}
  For each finitely supported function $m : S \rightarrow \NN$, let $q^m$ be the corresponding monomial in the elements $\{q_i\}_{i \in S}$. Let us write $0$ for the function that takes value zero on every element of $S$. So $q^0 =1$. The $q^m$ form a basis of $A^0$. Therefore, every element $y \in A$ can uniquely be written as a finite sum
  \begin{equation}
    \label{eq:27}
    y = \sum_{m} q^m \cdot y_m
  \end{equation}
  where $y_m \in A^{>}$. Moreover, $y \in A^>$ if and only if $y_m = 0$ for all $m \neq 0$.

  Let $x \in A^>$. We can write $x$ as a sum of words in the $q_i$ and $f_j$ for $i \in S, j\in T$ because these elements generate $A$. By our assumption (iv), we can rewrite all the words appearing so that the $q_i$ appear to the left of all the $f_j$. Therefore, we can write
  \begin{equation}
    \label{eq:28}
    x = \sum_{m} q^m \cdot x_m
  \end{equation}
  where $x_m$ are linear combinations of words in the $f_j$. Because $x \in A^>$, we conclude that $x = x_0$.
\end{proof}

% %%%%%%%%%%%%%%%%%%%%%%%%%%%%%%%%%%%%%%%%%%%%%%%%%%%%%%%%%%%%%%%%%

\appendix

\section{Rising and falling notation for Pochhammer symbols}
\label{appendix-sec:rising-and-falling-notation-for-pochhammer-symbols}
We will use an $\hbar$-version Knuth's rising and falling notation for Pochhammer symbols \cite{ConcreteMathematics}. That is, let $n \in \ZZ_{>0}$. We define
\begin{align}
  \label{eq:2}
  x^{\overline{n}} = x (x + \hbar) \cdots (x + (n-1)\hbar)
\end{align}
and:
\begin{align}
  \label{eq:6}
  x^{\underline{n}} = x (x - \hbar) \cdots (x - (n-1)\hbar)
\end{align}
For $0$, we have:
\begin{align}
  \label{eq:7}
  x^{\overline{0}} = x^{\underline{0}} = 1
\end{align}
Meanwhile, for negative numbers, we define
\begin{align}
  \label{eq:8}
  x^{\overline{-n}} = \frac{1}{ (x-n\hbar)^{\overline{n}}}  \quad \text{and} \quad   x^{\underline{-n}} = \frac{1}{ (x+n\hbar)^{\underline{n}}}
\end{align}
Finally, we will often consider ``double'' Pochhammer symbols:
\begin{equation}
    \label{eq:doublePochh}
    x^{\underline{m}\cdot \overline{n}}  = \prod_{\substack{0 \leq r <m, \\ 0 \leq s < n}} \big(x+(s-r)\hbar\big)
\end{equation}
Note that $x^{\underline{m}\cdot \overline{n}} = \prod_{0\leq r < m}(x-r\hbar)^{\overline{n}} = \prod_{0 \leq s < n} (x+s\hbar)^{\underline{m}}$.  We will make use of the following:
\begin{Lemma}
    \label{lem:doublePochhdiv} Consider integers $m,n \geq 0$.
    \begin{enumerate}[(a)]
    \itemsep0.5em 
        \item If $m \leq n$, then $x^{\underline{m-1}\cdot \overline{n}}$ divides $(x+\hbar)^{\underline{m}\cdot \overline{n-1}}$ with quotient $(x+\hbar)^{\overline{n-m}}$.
        \item If $m\geq n$, then  $(x+\hbar)^{\underline{m}\cdot \overline{n-1}}$ divides $x^{\underline{m-1}\cdot \overline{n}}$ with quotient $x^{\underline{m-n}}$.
    \end{enumerate}
    In particular, we have $x^{\underline{n-1} \cdot \overline{n}} = (x+\hbar)^{ \underline{n}\cdot \overline{n-1}} $.
\end{Lemma}
\begin{proof}
    
    We include a short diagrammatic proof.  Consider the $m\times n$ matrix $X$ with $(r,s)$-entry given by $x+(s-r)\hbar$.  Then $x^{\underline{m-1}\cdot \overline{n}}$ (resp.~$(x+\hbar)^{\underline{m}\cdot \overline{n-1}}$) is equal to the  product of the entries of $X$ with its bottom row (resp.~left column) removed, highlighted in red (resp.~in blue):
    $$
    \begin{pNiceMatrix}[margin]
        \Block[borders={bottom,left,right,top,tikz=red}]{4-4}{} x & \Block[borders={bottom,left,right,top,tikz=blue}]{5-3}{} x+\hbar & \cdots & x+ (n-1)\hbar \\
        x-\hbar & x & \cdots & x+(n-2) \hbar \\
        \vdots & \vdots & \ddots & \vdots \\
        x+(2-m)\hbar & x+(3-m)\hbar & \cdots & x+(n-m+1) \hbar \\
        x+(1-m)\hbar & x+(2-m)\hbar & \cdots & x + (n-m)\hbar
    \end{pNiceMatrix}
    $$
    Note that $X$ is constant on diagonals. If $m \leq n$ then $X$ is wider than it is tall, and we may ``slide'' all entries from the first column of the red matrix downwards along their diagonals, and see that they are equal to a subset of the entries of the bottom row of the blue matrix.  Thus $x^{\underline{m-1}\cdot \overline{n}}$ divides $(x+\hbar)^{\underline{m}\cdot \overline{n-1}}$, with quotient given by the product of the remaining entries of the bottom row of the blue matrix.  This proves (a), and the proof of (b) is similar.
\end{proof}

\bibliographystyle{amsalpha}
\bibliography{references}

\end{document}